\documentclass[11pt,a4paper]{article}
\usepackage{graphicx}
\usepackage{amscd}
\usepackage{amsmath,amsthm}
\usepackage{esint}
\usepackage{amsfonts}
\usepackage{amssymb}
\usepackage{stmaryrd}
\usepackage{mathrsfs}
\usepackage{latexsym}
\usepackage[all]{xy}
\usepackage{mathtools}
\usepackage{indentfirst}
\usepackage{cite}
\usepackage{bm}
\usepackage{centernot}
\usepackage{physics}
\usepackage{comment}
\usepackage{booktabs}
\usepackage{tabularx}
\usepackage{ragged2e}
\usepackage{newclude}
\usepackage{imakeidx}
\makeindex
\usepackage[dvipsnames,svgnames,x11names]{xcolor}
\usepackage[
colorlinks=true,
linkcolor=black,
anchorcolor=black,
citecolor=black,
urlcolor=black,
bookmarksopen=true
]{hyperref}
\usepackage{geometry}
\usepackage{tikz}
\usepackage{tikz-cd}
\usepackage{pgfplots}
\usetikzlibrary{
	positioning,
	arrows.meta,
	calc,
	shapes.geometric,
	shapes.misc,
	intersections,
	through,
	decorations.markings,
	decorations.pathmorphing,
	backgrounds,
	fit,
	quotes,
	angles,
	patterns
}
\newtheorem{lemma}{Lemma}[section]
\newtheorem{theorem}{Theorem}[section]

\newtheorem{remark}{Remark}[section]
\newtheorem{corollary}{Corollary}[section]

\newtheorem{proposition}{Proposition}[section]

\newtheorem{assumption}{Assumption}[section]

\newcommand{\ignore}[1]{}
\allowdisplaybreaks
\title{On the Cauchy problem for the multi-dimensional compressible Navier–Stokes–Korteweg system: Global strong solutions with arbitrarily large initial data}
\date{}
\author{
	\bf\large  Xiangdi Huang$^{a}$\thanks{E-mail addresses: xdhuang@amss.ac.cn (X. Huang); leimuxi25@mails.ucas.ac.cn (M. Lei); zhouhuitao141@163.com (H. Zhou).}, Muxi Lei$^{a}$, Huitao Zhou$^{a}$\\
	\small a. State Key Laboratory of Mathematical Sciences, Academy of Mathematics and Systems Science,\\
	\small Chinese Academy of Sciences, Beijing 100190, China;\\
}

\begin{document}
	
	\maketitle
	\begin{abstract}
		Since the pioneering work of Korteweg (1901) and the subsequent refinement of capillary fluid models by Dunn and Serrin (1985), the global existence of strong solutions to the multi-dimensional compressible Navier–Stokes–Korteweg (NSK) system with arbitrarily large initial data has stood as a formidable open problem in fluid mechanics. This challenge was recently addressed by [Gu-Huang-Meng-Zhou, arXiv:2603.11762], who established the global existence of strong solutions for arbitrarily large initial data on the periodic domain $\mathbb{T}^N$ ($N=2,3$), provided that the viscosity coefficients satisfy a BD-type algebraic relation ($\mu(\rho) = \nu \rho^\alpha, \lambda(\rho) = 2\nu(\alpha-1)\rho^\alpha$) and the Korteweg stress tensor complies with a generalized Bohm identity ($\kappa(\rho) = \varepsilon^2 \alpha^2 \rho^{2\alpha-3}$). However, the existence of global strong solutions for the Cauchy problem under these conditions has remained an open question. In this paper, we resolve this problem by proving the global existence of strong solutions for the Cauchy problem ($\mathbb{R}^N$, $N=2,3$) with arbitrarily large initial data and non-vacuum far-field density. By employing a refined truncation analysis combined with an original modified Nash-Moser type iteration scheme, we overcome the difficulties arising from the lack of integrability for the density in the whole space. This result extends the large-data theory of compressible Navier-Stokes-Korteweg equations from bounded torus $\mathbb{T}^N$ to unbounded whole space $\mathbb{R}^N$, thus applicable to more general physical settings.\\[4mm]
	\end{abstract}
	\section{Introduction}\label{sec1}
	
	In this paper, we study the Cauchy problem for the Navier--Stokes--Korteweg system in $\mathbb{R}^N$ with $N=2,3$:
	\begin{equation}\label{eq1}
		\left\{
		\begin{aligned}
			&\rho_t+\operatorname{div}(\rho u)=0,\\
			&(\rho u)_t+\operatorname{div}(\rho u\otimes u)+\nabla \rho^\gamma
			=\operatorname{div}\bigl(2\mu(\rho)\mathbb{D}u\bigr)
			+\nabla\bigl(\lambda(\rho)\operatorname{div}u\bigr)
			+\operatorname{div}\mathbb{K},
		\end{aligned}
		\right.
	\end{equation}
	where $\gamma\ge 1$, and $\rho=\rho(x,t)$, $u=(u_1,\dots,u_N)^{\top}=u(x,t)$ and $P=\rho^\gamma(\gamma\ge 1)$ denote the fluid density, velocity and pressure, respectively. Moreover,
	\[
	\mathbb{D}u=\frac{\nabla u+(\nabla u)^{\top}}{2},
	\]
	is the deformation tensor. The Korteweg stress tensor $\mathbb{K}$ is defined by
	\begin{equation}\label{eq2}
		\mathbb{K}
		=
		\Bigl(
		\rho\operatorname{div}\bigl(\kappa(\rho)\nabla\rho\bigr)
		-\frac{\rho\kappa'(\rho)-\kappa(\rho)}{2}\,|\nabla\rho|^2
		\Bigr)I
		-\kappa(\rho)\nabla\rho\otimes\nabla\rho,
	\end{equation}
	where $I$ denotes the $N\times N$ identity matrix. A direct computation yields
	\begin{equation}\label{eq3}
		\operatorname{div}\mathbb{K}
		=
		\nabla\Bigl(
		\rho\kappa(\rho)\Delta\rho
		+\frac{\kappa(\rho)+\rho\kappa'(\rho)}{2}\,|\nabla\rho|^2
		\Bigr)
		-\operatorname{div}\bigl(\kappa(\rho)\nabla\rho\otimes\nabla\rho\bigr).
	\end{equation}
	System \eqref{eq1} is based on the capillarity model originally formulated by Korteweg \cite{Korteweg}, where the stress tensor is augmented with density gradient terms. This formulation was subsequently modernized and generalized by Dunn and Serrin \cite{Dunn-Serrin}. The well-posedness of this equation has received extensive attention recently, leading to several important breakthroughs.\par
	With the choice of viscosity and capillarity coefficients satisfying
	\begin{equation}
		\mu(\rho) = \nu_0\rho, \quad \lambda(\rho) = 0, \quad \kappa(\rho) = \frac{\kappa_0^2}{\rho},
	\end{equation}
	the compressible Navier-Stokes-Korteweg equations degenerate into the quantum Navier-Stokes system. In the one-dimensional setting, some advancements have been achieved concerning the well-posedness of the system. For the case where the viscosity and capillarity coefficients are balanced ($\nu_0 = \kappa_0$), Jüngel \cite{Jüngel} established the global existence of smooth solutions in the whole space $\mathbb{R}$, provided the initial density is away from vacuum. Recently, this result was extended by Chen and Zhao \cite{Chen-Zhao} to a broader regime where $\nu_0 \ge \kappa_0$ and $\gamma \ge 1$. They obtained the global classical solutions and further investigated the large-time behavior of these solutions in the absence of vacuum. For the multi-dimensional case, the exploration of global-in-time solutions has primarily focused on weak or renormalized solutions, particularly when vacuum states are involved. In the periodic setting $\Omega=\mathbb{T}^N$, Jüngel \cite{Jüngel-2} first established the global existence of weak solutions allowing vacuum for the case $\nu_0 < \kappa_0$, under the technical constraints $\gamma \ge 1$ for $N=2$ and $\gamma > 3$ for $N=3$. Subsequently, Antonelli and Spirito \cite{Antonelli-Spirito} constructed global weak solutions in the absence of vacuum for various regimes of $\nu_0$ and $\kappa_0$, providing a standard framework for $N=2, 3$. Later, Lacroix-Violet and Vasseur \cite{Lacroix-Violet-Vasseur} proved the existence of global renormalized solutions for $\Omega=\mathbb{T}^N$ with $\gamma > 1$ and $\kappa_0 \ge 0$, a result that allowed vacuum. More recently, the focus has shifted toward the Cauchy problem in the whole space $\mathbb{R}^N$. Notably, Antonelli, Hientzsch and Spirito \cite{Antonelli-Hientzsch-Spirito} obtained global weak solutions allowing vacuum for the multi-dimensional system in $\mathbb{R}^2$ and $\mathbb{R}^3$, provided that the viscosity and capillarity coefficients satisfy $\nu_0 = \kappa_0$. However, the global-in-time existence of strong solutions to the multi-dimensional generalized NSK system remained a long-standing open problem. This challenge was first addressed by Huang, Meng and Zhang \cite{Huang-Meng-Zhang} in the periodic setting $\mathbb{T}^N$. They established the global existence of strong solutions away from vacuum for the case $\nu_0 = \kappa_0$, covering the adiabatic exponent range $\gamma \in [1, \infty)$ for $N=2$ and $\gamma \in [1, 7/3)$ for $N=3$. Subsequently, by introducing a truncation mechanism, Huang, Gu and Lei \cite{Huang-Gu-Lei} successfully extended these results to the Cauchy problem in the whole space $\mathbb{R}^N$ in the absence of vacuum, while further improving the range of $\gamma$ to $[1, 8/3)$ in the three-dimensional case.\par
	However, the situation becomes much more delicate when the system \eqref{eq1} is considered with general coefficient structure. For the one-dimensional space case, the theory has been relatively fruitful. Tsyganov \cite{Tsyganov} proved the global existence and asymptotic behaviors of weak solutions with large data away from vacuum. For systems with density-dependent coefficients, Germain and LeFloch \cite{Germain-LeFloch} established the global existence of finite energy weak solutions for the Cauchy problem and their convergence toward entropy solutions of the Euler system. Furthermore, Burtea and Haspot \cite{Burtea-Haspot}  explored the capillarity vanishing limit, while Antonelli, Bresch and Spirito \cite{Antonelli-Bresch-Spirito} obtained global weak solutions for the periodic problem with large data under certain conditions on $\mu(\rho)=\rho^\alpha$ and $\kappa(\rho)=\rho^\beta$. However, when moving to high-dimensional problems, the available results for the general NSK system are much more constrained. For density-dependent cases, Danchin and Desjardins \cite{Danchin-Desjardins} obtained global smooth solutions only for data close to a stable equilibrium. A major milestone for weak solutions was achieved by Bresch, Desjardins and Lin \cite{Bresch-Desjardins-Lin}, who established the global existence of weak solutions allowing for vacuum by introducing a novel a priori entropy estimate (the BD entropy). This framework was further extended to viscous shallow water models with capillarity \cite{Bresch-Desjardins}. More recently, Haspot \cite{Haspot} investigated both local and global strong solutions for various density-dependent coefficients, yet these results still largely rely on smallness assumptions on the initial data. The transition from the 'balanced' quantum coefficients to general coefficients often results in the loss of certain analytical structures (such as the effective velocity), making the global existence of strong solutions with large data an even more demanding problem. Recently, a groundbreaking advancement was achieved by Gu, Huang, Meng and
	Zhou \cite{Gu-Huang-Meng-Zhou}, who resolved the long-standing open problem regarding the global existence of strong solutions for the multi-dimensional NSK model with general coefficients. In their work, they provided an affirmative answer to this challenge by considering a generalized algebraic structure for the coefficients. Specifically, the authors operated under the assumption that the viscosity coefficients satisfy a BD-type algebraic relation of the form $\mu(\rho) = \nu\rho^\alpha$ and $\lambda(\rho) = 2\nu(\alpha - 1)\rho^\alpha$. Furthermore, they assumed that the Korteweg stress tensor complies with a generalized Bohm identity of the form $\kappa(\rho) = \varepsilon^2 \alpha^2 \rho^{2\alpha-3}$. Under these structural conditions, they successfully established the global existence of strong solutions for the 2D and 3D systems on the torus $\mathbb{T}^N$ with arbitrarily large regular initial data. \par
	Despite the aforementioned progress in establishing the global existence of strong solutions with arbitrarily large initial data in the periodic domain for a broad class of general NSK coefficients, the well-posedness of the Cauchy problem remains a more fundamental and pressing issue. The primary mathematical challenges arise from the unbounded measure of the whole space and the non-vanishing density at infinity, which lead to severe difficulties regarding the integrability of the solutions. Addressing these critical issues in the context of the Cauchy problem is the central objective of the present paper.\par
	Throughout the paper, we work under the following structural assumptions on the coefficients:
	\begin{equation}\label{eq4}
		\mu(\rho)=\nu\rho^\alpha,\qquad
		\lambda(\rho)=2\nu(\alpha-1)\rho^\alpha,\qquad
		\kappa(\rho)=\varepsilon^2\alpha^2\rho^{2\alpha-3},\qquad
		\nu\ge \varepsilon>0.
	\end{equation}
	In particular,
	\[
	\lambda(\rho)=2\rho\mu'(\rho)-2\mu(\rho),
	\]
	so that the viscosity coefficients satisfy the BD algebraic relation. Moreover,
	\[
	2\mu(\rho)+N\lambda(\rho)=2\nu\bigl(N\alpha-N+1\bigr)\rho^\alpha.
	\]
	Hence the physical constraints
	\[
	\mu(\rho)\ge 0,\qquad 2\mu(\rho)+N\lambda(\rho)\ge 0
	\]
	require that
	\[
	\alpha\ge \frac{N-1}{N}.
	\]
	
	To focus on the essential difficulties, we prescribe the far-field equilibrium state
	\begin{equation}\label{eq5}
		\rho(x,t)\to 1,\qquad u(x,t)\to 0,\qquad \text{as } |x|\to\infty.
	\end{equation}
	The initial data are given by
	\begin{equation}\label{eq6}
		\rho(x,0)=\rho_0(x),\qquad u(x,0)=u_0(x),\qquad x\in\mathbb{R}^N,
	\end{equation}
	and are assumed to satisfy
	\begin{equation}\label{eq7}
		0<\underline{\rho_0}\le \rho_0\le \overline{\rho_0},\qquad
		\rho_0-1\in H^3(\mathbb{R}^N),\qquad
		u_0\in H^2(\mathbb{R}^N),
	\end{equation}
	where $\underline{\rho_0}$ and $\overline{\rho_0}$ are positive constants.
	
	For later use, we introduce
	\begin{equation}\label{eq8}
		\beta=\sqrt{1-\frac{\varepsilon^2}{\nu^2}}\in [0,1).
	\end{equation}
	
	\begin{assumption}[Parameter conditions]\label{ass1}
		Assume that $(\alpha,\gamma,\beta)$ satisfies one of the following conditions:
		\begin{equation}\label{eq9}
			\begin{aligned}
				N = 2, \quad \alpha \in \left( \frac{\sqrt{5}-1}{2}, 1 \right), \quad \gamma \in [1, \infty), \quad \beta \in [0, \beta_2^+(\alpha)); \\
			\end{aligned}
		\end{equation}
		or
		\begin{equation}\label{eq10}
			N = 3, \quad \alpha \in \left( \frac{9\sqrt{3}-4\sqrt{2}}{9\sqrt{3}-2\sqrt{2}}, 1 \right), \quad \gamma \in \left[ 1, \frac{15\alpha-7}{3} \right), \quad \beta \in [0, \beta_3^+(\alpha)).
		\end{equation}
		Here, $\beta_N^+(\alpha)$ $(N=2,3)$ is defined as the unique positive solution to the quadratic equation:$$\left(\beta_N^+(\alpha)\right)^2 + \frac{2(1 - \alpha)(N\alpha^2 + \sqrt{N}\alpha + 1 - N)}{\alpha(1 - \sqrt{N}(1 - \alpha))^2} \beta_N^+(\alpha) + \frac{(1 - \alpha)(-N\alpha^2 - N\alpha + 2N - 2)}{\alpha(1 - \sqrt{N}(1 - \alpha))^2} = 0.$$
	\end{assumption}
	
	We now state the main result of the paper.
	
	\begin{theorem}\label{thm1}
		Let $N\in\{2,3\}$ and assume that Assumption \ref{ass1} holds. Then the Cauchy problem \eqref{eq1}, \eqref{eq4}-\eqref{eq7} admits a unique global strong solution $(\rho,u)$. More precisely, for any $0<T<\infty$, there exists a constant $C(T)>0$ such that
		\[
		(C(T))^{-1}\le \rho(x,t)\le C(T)
		\qquad \text{for all } (x,t)\in \mathbb{R}^N\times[0,T],
		\]
		and
		\[
		\left\{
		\begin{aligned}
			&\rho-1\in C([0,T];H^3(\mathbb{R}^N))\cap L^2(0,T;H^4(\mathbb{R}^N)),\\
			&\rho_t\in C([0,T];H^1(\mathbb{R}^N))\cap L^2(0,T;H^2(\mathbb{R}^N)),\\
			&u\in C([0,T];H^2(\mathbb{R}^N))\cap L^2(0,T;H^3(\mathbb{R}^N)),\\
			&u_t\in L^\infty(0,T;L^2(\mathbb{R}^N))\cap L^2(0,T;H^1(\mathbb{R}^N)).
		\end{aligned}
		\right.
		\]
		Here the constant $C(T)$ depends only on $T$, the parameters of the system, and the norms of the initial data.
	\end{theorem}
	\begin{remark}
		Theorem \ref{thm1} can be regarded as the extension of the periodic domain results Gu, Huang, Meng and
		Zhou \cite{Gu-Huang-Meng-Zhou} to the whole space, achieving the same range of indices. In comparison, obtaining the same indices for $\alpha$ and $\gamma$ in the whole space entails significantly more challenges, primarily due to the far-field behavior at infinity and the lack of compactness of the domain. However, according to the $\gamma$-range $\gamma \in [1, 8/3)$ for the quantum NSK system established by Huang, Gu and Lei \cite{Huang-Gu-Lei}, we observe that the limit of $(15\alpha - 7)/3$ as $\alpha \to 1$ is precisely $8/3$. This provides strong justification for the conjecture that $\gamma$ can reach the same range as in the periodic case.
	\end{remark}
	\begin{remark}
		For $N=2$, the lower bound of $\alpha$ arises from the constraints required to establish the density lower bound. In contrast, for $N=3$, the lower bound of $\alpha$ is dictated by the requirements of higher-order estimates, which impose more stringent restrictions than those stemming from the density lower bound. In the specific case where $\alpha=1$, the system reduces to the quantum Navier–Stokes equations. Due to the emergence of critical indices in this setting, the analytical techniques required differ significantly, as discussed in Huang, Gu and Lei \cite{Huang-Gu-Lei}.
	\end{remark}
	The rest of the paper is organized as follows. In Section 2, we provide some notations, the reformulations of the system, and an outline of the proof. In Section \ref{sec2}, we derive all the a priori estimates needed for Theorem \ref{thm1}. More precisely, Section \ref{subsec1} is devoted to the upper bound for the density, Section \ref{subsec2} to the positive lower bound for the density, and Section \ref{subsec3} to the higher-order estimates. Finally, in Section \ref{sec3}, we complete the proof of Theorem \ref{thm1}.
	\section{Notations, Reformulation, and  Outline of the Proof}
	\subsection{Notations}
	
	Throughout the paper, $N\in\{2,3\}$ is fixed and the spatial domain is $\mathbb{R}^N$. For brevity, whenever no confusion can arise, we abbreviate
	\[
	\|f\|_{L^p}:=\|f\|_{L^p(\mathbb{R}^N)},\qquad
	\|f\|_{W^{k,p}}:=\|f\|_{W^{k,p}(\mathbb{R}^N)},\qquad
	\|f\|_{H^k}:=\|f\|_{H^k(\mathbb{R}^N)}.
	\]
	Likewise, if $X$ is a Banach space and $I\subset[0,\infty)$ is an interval, then $L^p(I;X)$ and $C(I;X)$ denote the usual Bochner spaces. The same notation will be used for scalar-, vector-, and matrix-valued functions.
	
	Because the far-field equilibrium is given by $(\rho,u)=(1,0)$, we systematically work with the relative density $\rho-1$ rather than $\rho$ itself. In particular, all Sobolev regularity statements for the density are understood in terms of $\rho-1$. This convention is natural in the whole-space setting and is fully compatible with the far-field condition \eqref{eq5}.
	
	For a scalar function $f$, $\nabla f$ and $\Delta f$ denote the spatial gradient and Laplacian, respectively, while $\nabla^2 f$ stands for the Hessian matrix. More generally, $\nabla^k f$ denotes the collection of all spatial derivatives of order $k$. For a vector field $w=(w_1,\dots,w_N)^{\top}$, we write
	\[
	\nabla w=(\partial_j w_i)_{1\le i,j\le N},
	\qquad
	\operatorname{div}w=\sum_{i=1}^N\partial_i w_i,
	\]
	where $\partial_i=\partial_{x_i}$. If $A=(A_{ij})$ and $B=(B_{ij})$ are $N\times N$ matrices, we set
	\[
	A:B:=\sum_{i,j=1}^N A_{ij}B_{ij},
	\qquad
	|A|:=(A:A)^{1/2},
	\qquad
	A^\top:=(A_{ji})_{1\le i,j\le N},
	\qquad
	\operatorname{tr}A:=\sum_{i=1}^N A_{ii}.
	\]
	For vectors $a,b\in\mathbb{R}^N$, we write
	\[
	a\otimes b:=(a_i b_j)_{1\le i,j\le N}.
	\]
	In particular,
	\[
	|\nabla w|^2=\nabla w:\nabla w,
	\qquad
	\mathbb{D}w=\frac{\nabla w+(\nabla w)^{\top}}{2}.
	\]
	The symbol $|\cdot|$ will also be used for the Euclidean norm of vectors and for the Lebesgue measure of measurable sets; the meaning will always be clear from the context.
	
	If $E\subset\mathbb{R}^N$ is measurable, then $|E|$ denotes its Lebesgue measure, $\mathbf 1_E$ its characteristic function, and $\operatorname{supp}f$ the support of a function $f$. For $s\in\mathbb{R}$, we set
	\[
	s_+:=\max\{s,0\},
	\qquad
	s_-:=\max\{-s,0\}.
	\]
	Whenever no ambiguity is possible, we also write
	\[
	\int f\,dx:=\int_{\mathbb{R}^N}f(x)\,dx.
	\]
	
	Finally, $C$ denotes a generic positive constant, not necessarily the same from line to line.  When it is important to record the dependence explicitly, we write $C=C(\cdots)$. The notation $X\hookrightarrow Y$ means that $X$ is continuously embedded into $Y$.
	\subsection{Reformulation}
Before deriving the a priori estimates, we recast \eqref{eq1} in terms of an effective velocity, following the arguments presented in the appendix of \cite{Gu-Huang-Meng-Zhou}. Set
	\[
	c_\pm=\nu(1\pm\beta)=\nu\pm\sqrt{\nu^2-\varepsilon^2},
	\qquad
	v_\pm=u+c_\pm\alpha\rho^{\alpha-2}\nabla\rho.
	\]
	By \eqref{eq4} and \eqref{eq8}, the constants $c_\pm$ satisfy
	\[
	c_\pm^2-2\nu c_\pm+\varepsilon^2=0.
	\]
	A straightforward computation based on \eqref{eq3} and \eqref{eq1} shows that $(\rho,u)$ solves \eqref{eq1} if and only if $(\rho,v_\pm)$ solves
	\begin{equation}\label{eq11}
		\left\{
		\begin{aligned}
			&\rho_t+\operatorname{div}(\rho v_\pm)-c_\pm\Delta\rho^\alpha=0,\\
			&\rho (v_\pm)_t+\rho\bigl(v_\pm-c_\pm\alpha\rho^{\alpha-2}\nabla\rho\bigr)\cdot\nabla v_\pm+\nabla\rho^\gamma \\
			&\qquad
			=\nu\operatorname{div}(\rho^\alpha\nabla v_\pm)
			+(\nu-c_\pm)\operatorname{div}\bigl(\rho^\alpha(\nabla v_\pm)^\top\bigr)
			+(\alpha-1)(2\nu-c_\pm)\nabla\bigl(\rho^\alpha\operatorname{div}v_\pm\bigr).
		\end{aligned}
		\right.
	\end{equation}
	Moreover, once $(\rho,v_\pm)$ is known, the original velocity is recovered uniquely from
	\begin{equation}\label{eq12}
		u=v_\pm-c_\pm\alpha\rho^{\alpha-2}\nabla\rho.
	\end{equation}
	
	In the rest of the paper we work with the positive branch and write
	\[
	v:=v_+,
	\qquad
	c:=c_+=\nu(1+\beta).
	\]
	Then \eqref{eq11} becomes
	\begin{equation}\label{sys of v}
		\left\{
		\begin{aligned}
			&\rho_t+\operatorname{div}(\rho v)-c\Delta\rho^\alpha=0,\\
			&\rho v_t+\rho\bigl(v-c\alpha\rho^{\alpha-2}\nabla\rho\bigr)\cdot\nabla v+\nabla\rho^\gamma \\
			&\qquad
			=\nu\operatorname{div}(\rho^\alpha\nabla v)
			+(\nu-c)\operatorname{div}\bigl(\rho^\alpha(\nabla v)^\top\bigr)
			+(\alpha-1)(2\nu-c)\nabla\bigl(\rho^\alpha\operatorname{div}v\bigr).
		\end{aligned}
		\right.
	\end{equation}
	Observe that
	\[
	c-\nu=\nu\beta,
	\qquad
	2\nu-c=\nu(1-\beta)>0.
	\]
	
	We next translate the initial conditions for \eqref{eq1} into those for \eqref{sys of v}. Since $N\in\{2,3\}$, the Sobolev embedding
	\[
	H^3(\mathbb{R}^N)\hookrightarrow W^{1,\infty}(\mathbb{R}^N)
	\]
	combined with \eqref{eq7} and the standard composition estimates yields
	\[
	v_0:=u_0+c\alpha\rho_0^{\alpha-2}\nabla\rho_0\in H^2(\mathbb{R}^N).
	\]
	Hence the initial data for \eqref{sys of v} satisfy
	\begin{equation}\label{eq14}
		0<\underline{\rho_0}\le \rho_0\le \overline{\rho_0},
		\qquad
		\rho_0-1\in H^3(\mathbb{R}^N),
		\qquad
		v_0\in H^2(\mathbb{R}^N),
	\end{equation}
	where $\underline{\rho_0}$ and $\overline{\rho_0}$ are the positive constants appearing in \eqref{eq7}. The corresponding far-field condition is
	\begin{equation}\label{eq15}
		\rho(x,t)\to 1,
		\qquad
		v(x,t)\to 0,
		\qquad
		\text{as } |x|\to\infty.
	\end{equation}
	\subsection{Outline of the Proof}
	We now present the outline of the proof for the main results. The strategy is divided into the following key steps:
	
	\begin{itemize}
		\item \textbf{Local Well-posedness:} First, we establish the local-in-time existence of solutions to the Cauchy problem. Since this process is relatively standard for such systems, the details are omitted for brevity.
		
		\item \textbf{Integrability of Momentum:} Second, as a key to improving the integrability of density, we derive the integrability estimates for $\int_{\mathbb{R}^N} \rho |v|^{q+2} \, dx$. Given a fixed value of $\alpha$, obtaining these estimates involves determining the range of the critical index $\beta_N^{+}(\alpha)$ (for $N=2,3$). In particular, the three-dimensional case necessitates certain restrictions on the range of the adiabatic exponent $\gamma$.
		
		\item \textbf{Upper Bound of Density:} Third, equipped with the integrability of $\int_{\mathbb{R}^N} \rho |v|^{q+2} \, dx$ for $q > N-2$, we apply a carefully designed new Nash-Moser iteration to the density equation to establish a uniform upper bound for the density. It is worth noting that as $q$ approaches $N-2$, the permissible lower bound for $\alpha$ becomes broader.
		
		\item \textbf{Lower Bound of Density:} Fourth, by leveraging the newly established density upper bound, we can refine the integrability of $\int_{\mathbb{R}^N} \rho |v|^{q+2} \, dx$. Re-applying another refined Nash-Moser iteration then allows us to derive a strictly positive lower bound for the density, thereby ruling out vacuum formation.
		
		\item \textbf{Global Existence:} Finally, with the uniform upper and lower bounds of density, we obtain the remaining higher-order estimates. By employing a contradiction argument regarding the maximal existence time $T^*$, we conclude the existence of a global-in-time strong solution.
	\end{itemize}
	In the periodic case, the methods for estimating the upper and lower bounds of the density rely heavily on the integrability of the density and the finite measure of the domain, thus completely inapplicable to the whole space. The improvements achieved in this work primarily stem from four technical advancements:\\
	
	\noindent \textbf{§2.3.1 Piecewise analysis of $\gamma$:} We implement a piecewise investigation of the adiabatic exponent $\gamma$ to achieve a \textit{$\gamma$-improvement} for the weighted integrability of the effective velocity.
	
	\medskip
	
	\noindent \textbf{§2.3.2 Truncation procedure for the whole space:} We develop a \textit{truncation procedure} specifically adapted to the whole-space setting, which effectively addresses the difficulties arising from the lack of boundary constraints and compactness.
	\medskip
	
	\noindent \textbf{§2.3.3 Refined Nash-Moser iteration for the lower bound of density:} 
	Establishing the density lower bound is more technically demanding than obtaining the upper bound. For the lower bound analysis, we must construct suitable reverse Hölder inequalities for the equation governing the density's lower bound and guarantee that the iteration converges.
	
	\medskip
	\noindent \textbf{§2.3.4 Decoupling of the strong coupling in the first-level higher-order estimates:} This decoupling method is particularly evident in the case $N=3$. Using the dissipation generated by the estimate of $\|\nabla \rho\|_{L^4}$ as an auxiliary source of damping, we establish a large-scale Gronwall inequality involving $\|\nabla \rho\|_{L^4}$, $\|\nabla^2 \rho\|_{L^2}$, and $\|\nabla v\|_{L^2}$.
	
	\subsubsection{Piecewise analysis of $\gamma$ when $N=3$}
	Prior to establishing the upper bound of the density, it is necessary to obtain the density-weighted estimates for the effective velocity. Testing the momentum equation in \eqref{sys of v} by $|v|^qv$ yields 
	\[
	\frac{1}{q+2}\frac{d}{dt}\int_{\mathbb{R}^3}\rho|v|^{q+2}\,dx
	+\delta_q\int_{\mathbb{R}^3}\rho^\alpha|v|^q|\nabla v|^2\,dx
	\le
	C\left|\int_{\mathbb{R}^3}\nabla\rho^\gamma\cdot |v|^qv\,dx\right|.
	\]
	For different ranges of $\gamma$, we employ distinct sets of known integrability conditions to establish a closed-loop weighted estimate for the effective velocity. Under Assumption \ref{ass1}, the following two cases are sufficient to cover the entire range of the adiabatic exponent $\gamma$.
	\begin{itemize}
		\item \textbf{Case I: $\gamma \in [\frac32\alpha-\frac16,\frac{15\alpha-7}{3})$}. In this case, we choose $1 < q < 2$ such that the following condition is satisfied:$$\frac{3}{2}\alpha - \frac{1}{6} \le \gamma < 3\alpha - 1 + \frac{6\alpha - 4}{q + 2}.$$
		In this step, we exploit the integrability condition 
		$$\left\| \rho^{\frac{3}{2}\alpha-1}-1 \right\|_{L^2(0,T;H^2(\mathbb{R}^3))} \le C,$$
		and by splitting the pressure term according to the domain decomposition $\mathbb{R}^3 = \Omega_1(t) \cup \Omega_2(t)$, we have
		\begin{equation*}
			\left|\int_{\mathbb{R}^3}\nabla\rho^\gamma\cdot |v|^qv\,dx\right|
			\le
			\left|\int_{\Omega_1(t)}\nabla\rho^\gamma\cdot |v|^qv\,dx\right|
			+
			\left|\int_{\Omega_2(t)}\nabla\rho^\gamma\cdot |v|^qv\,dx\right|,
		\end{equation*}
		where the disjoint subsets $\Omega_1(t)$ and $\Omega_2(t)$ are defined as
		\begin{equation} \label{eq:domain-split}
			\Omega_1(t) := \{x \in \mathbb{R}^3 : \rho(x,t) \le 4\}, \quad 
			\Omega_2(t) := \{x \in \mathbb{R}^3 : \rho(x,t) > 4\}.
		\end{equation}
		We establish the density-weighted integrability of the effective velocity within this $\gamma$-interval.
		\item  \textbf{Case II: $\gamma \in [1, \frac{9}{2}\alpha-2)$}. In this case, the main challenge arises from the high-density region $\Omega_2(t)$. By defining $\zeta' := \frac{(5\alpha-3)q+22\alpha-12}{2}$ and employing interpolation techniques, we utilize an additional integrability condition
		\begin{equation*}
			\int_0^T \left( \int_{\Omega_2(t)} \rho^{\zeta'} dx \right)^{\frac{2\lambda}{(q+2)\zeta'}} dt \le C,
		\end{equation*}
		to establish the density-weighted integrability of the effective velocity within this $\gamma$-range.
	\end{itemize}
	\subsubsection{Truncation procedure for the whole space}
	Since $\rho\to1$ at infinity, one cannot iterate directly on $\rho$ or $\rho^{-1}$. Therefore, we introduce suitable truncations of the large-density and low-density parts and performs a modified Nash--Moser iteration only on these truncated quantities. The relative entropy estimate guarantees that their supports have finite measure, which makes the iteration possible. For example, when establishing the upper bound of the density, we choose $\eta\in C^\infty([0,\infty))$ such that
	\[
	0\le \eta\le 1,
	\qquad
	\eta(s)=0 \ \text{for } 0\le s\le 2,
	\qquad
	\eta(s)=1 \ \text{for } s\ge 4,
	\]
	and define
	\[
	f:=\eta(\rho)\rho.\]
	By testing against appropriately chosen smooth-truncated density functions, we focus on improving the $L^p$ integrability of the density within the high-density region, which is also a domain of finite measure. Through a refined Nash-Moser iteration technique, we successfully establish the upper bound for the density. 
	\subsubsection{Refined Nash-Moser iteration for the lower bound of density}
	The lower bound of the density is the most delicate part of the whole-space analysis and the point at which the argument departs decisively from the periodic theory. The difficulty is not simply to iterate on the inverse density. Once one passes to $\tau:=\rho^{-1}$, the low-density regime becomes singular, while the transport generated by the effective velocity remains fully nonlinear. On $\mathbb{R}^N$, where no background integrability for $\tau-1$ is available, a direct iteration on $\rho^{-1}$ cannot be closed.
	
	The argument begins with a truncation of the inverse density, \(f:=\eta(\tau)\tau\), 
	which localizes the analysis to the vacuum region. Since $\rho\to1$ at infinity, the relevant obstruction is the concentration of $\tau$ near vacuum rather than its behavior at spatial infinity. The relative entropy bound then yields the crucial support control: the truncated inverse density is supported on a set of uniformly finite measure in time.
	
	This geometric input must be combined with the refined weighted estimate for the effective velocity, available only after the upper bound of the density has been established. The lower bound is therefore not a formal consequence of Nash--Moser iteration alone. What closes the argument is a carefully designed new Nash--Moser scheme for the truncated inverse density, stable in a singular regime where diffusion and transport remain strongly coupled.
	
	The outcome is a strictly positive lower bound for the density on $\mathbb{R}^N\times[0,T]$. This is the decisive step in the whole-space theory: it excludes vacuum formation and makes the global higher-order analysis possible.
	\subsubsection{Decoupling of the strong coupling in the first-level higher-order estimates}
	To establish the first-level higher-order estimates, we construct a large-scale Gronwall inequality of the form:
	\begin{equation}
		\begin{aligned}
			&\frac{d}{dt} \left( \|\nabla^2\rho(t)\|_{L^2}^2 + \|\nabla v(t)\|_{L^2}^2 + \|\nabla\rho(t)\|_{L^4}^4 \right) \\
			&\quad + \left( \|\nabla^2 v\|_{L^2}^2 + \|\nabla\Delta\rho\|_{L^2}^2 + \| |\nabla \rho| |\nabla^2 \rho| \|_{L^2}^2 \right) \le \mathcal{R}(t).
		\end{aligned}
	\end{equation}
	Specifically, the dissipation term $\| |\nabla \rho| |\nabla^2 \rho| \|_{L^2}^2$ generated by the estimate of $\|\nabla \rho(t)\|_{L^4}^4$ plays a crucial role in canceling the nonlinear terms on the right-hand side. The constant $\frac{9\sqrt{3}-4\sqrt{2}}{9\sqrt{3}-2\sqrt{2}}$ in Assumption \ref{ass1} is precisely the threshold required to ensure that $\| |\nabla \rho| |\nabla^2 \rho| \|_{L^2}^2$ provides the necessary positive dissipation. By coupling these three independent estimates, we establish a comprehensive, closed-loop Gronwall inequality. It is worth noting that without the inclusion of this specific dissipation, the system involving only $\|\nabla^2\rho\|_{L^2}$ and $\|\nabla v\|_{L^2}$ remains strongly coupled and cannot be decoupled. This represents the core technical difficulty encountered in establishing the higher-order estimates. It is precisely what prevents the three-dimensional strong coupling from propagating to higher levels and makes the global higher-order theory available.
	
	\section{A priori estimates}\label{sec2}
	\subsection{Preliminaries}

	Throughout the remainder of the paper, we assume that \eqref{eq4}, \eqref{eq5}, \eqref{eq6}, \eqref{eq7}, and Assumption \ref{ass1} are in force. It follows from the standard local well-posedness theory for \eqref{eq1} together with the above change of variables that the reformulated system \eqref{sys of v}, supplemented with \eqref{eq14} and \eqref{eq15}, admits a unique strong solution
	\[
	(\rho,v)\quad \text{on } \mathbb{R}^N\times[0,T^\ast),
	\]
	where $T^\ast\in(0,\infty]$ denotes the maximal existence time. Moreover, for every fixed $T\in(0,T^\ast)$,
	\[
	\left\{
	\begin{aligned}
		&\rho-1\in C([0,T];H^3(\mathbb{R}^N))\cap L^2(0,T;H^4(\mathbb{R}^N)),\\
		&\rho_t\in C([0,T];H^1(\mathbb{R}^N))\cap L^2(0,T;H^2(\mathbb{R}^N)),\\
		&v\in C([0,T];H^2(\mathbb{R}^N))\cap L^2(0,T;H^3(\mathbb{R}^N)),\\
		&v_t\in L^\infty(0,T;L^2(\mathbb{R}^N))\cap L^2(0,T;H^1(\mathbb{R}^N)).
	\end{aligned}
	\right.
	\]
	In particular, all computations in the a priori estimates below are justified on every fixed interval $[0,T]$.
	
	To prove Theorem \ref{thm1}, we argue by contradiction. If the conclusion of Theorem \ref{thm1} were false, then necessarily
	\[
	T^\ast<\infty.
	\]
	Accordingly, we fix an arbitrary
	\[
	T\in(0,T^\ast)
	\]
	and derive a priori estimates for $(\rho,v)$ on $\mathbb{R}^N\times[0,T]$.
	
	We first introduce the relative internal energy density
	\begin{equation}\label{eq16}
		\Pi(\rho)=
		\begin{cases}
			\dfrac{\rho^\gamma-1-\gamma(\rho-1)}{\gamma-1}, & \gamma>1,\\[2mm]
			\rho\log\rho-\rho+1, & \gamma=1.
		\end{cases}
	\end{equation}
	It is immediate that, for every $\gamma\ge 1$,
	\[
	\Pi(1)=0,
	\qquad
	\Pi'(1)=0,
	\]
	and
	\begin{equation}\label{eq17}
		\rho \nabla \Pi'(\rho)=\nabla \rho^\gamma,
	\end{equation}
	where, in the case $\gamma=1$, \eqref{eq17} is understood as
	\[
	\rho\nabla(\log\rho)=\nabla\rho.
	\]
	
	Accordingly, we define the initial total energy by
	\begin{equation}\label{eq18}
		E_0
		=
		\int_{\mathbb{R}^N}
		\left(
		\frac{1}{2}\rho_0|u_0|^2
		+\Pi(\rho_0)
		+\frac{2\varepsilon^2\alpha^2}{(2\alpha-1)^2}
		\left|\nabla \rho_0^{\alpha-\frac12}\right|^2
		\right)\,dx.
	\end{equation}
	
	Moreover, combining \eqref{eq3} with \eqref{eq4}, a direct computation gives
	\begin{equation}\label{eq19}
		\operatorname{div}\mathbb{K}
		=
		\frac{2\varepsilon^2\alpha^2}{2\alpha-1}\,
		\rho\,\nabla\Bigl(\rho^{\alpha-\frac32}\Delta\rho^{\alpha-\frac12}\Bigr).
	\end{equation}
	
	We first establish the basic energy estimate.
	
	\begin{proposition}\label{prop1}
		Let $N\in\{2,3\}$, $\gamma\ge 1$, and assume that
		\[
		\alpha>\frac{N-1}{N}.
		\]
		Let $\Pi(\rho)$ and $E_0$ be defined by \eqref{eq16} and \eqref{eq18}, respectively. Let $(\rho,u)$ be a sufficiently smooth solution to \eqref{eq1} on $\mathbb{R}^N\times[0,T]$ such that $\rho>0$ and the far-field condition \eqref{eq5} holds. Then, for every $t\in[0,T]$,
		\begin{equation}\label{eq20}
			\begin{aligned}
				&\frac{d}{dt}\int_{\mathbb{R}^N}
				\left(
				\frac12 \rho|u|^2
				+\Pi(\rho)
				+\frac{2\varepsilon^2\alpha^2}{(2\alpha-1)^2}
				\left|\nabla \rho^{\alpha-\frac12}\right|^2
				\right)\,dx \\
				&\qquad
				+2\nu\int_{\mathbb{R}^N}\rho^\alpha|\mathbb{D}u|^2\,dx
				+2\nu(\alpha-1)\int_{\mathbb{R}^N}\rho^\alpha(\operatorname{div}u)^2\,dx
				=0.
			\end{aligned}
		\end{equation}
		In particular, there exists a positive constant
		\[
		C=C(N,\alpha,\gamma,\nu,\varepsilon,E_0)
		\]
		such that
		\begin{equation}\label{est: basic}
			\sup_{0\le t\le T}\int_{\mathbb{R}^N}
			\left(
			\rho|u|^2+\Pi(\rho)+\left|\nabla \rho^{\alpha-\frac12}\right|^2
			\right)\,dx
			+
			\int_0^T\int_{\mathbb{R}^N}\rho^\alpha|\mathbb{D}u|^2\,dx\,dt
			\le C.
		\end{equation}
	\end{proposition}
	
	\begin{proof}
		Multiply the momentum equation in \eqref{eq1} by $u$ and integrate over $\mathbb{R}^N$. By the continuity equation,
		\[
		\int_{\mathbb{R}^N}
		\Bigl((\rho u)_t+\operatorname{div}(\rho u\otimes u)\Bigr)\cdot u\,dx
		=
		\frac{d}{dt}\int_{\mathbb{R}^N}\frac12\rho|u|^2\,dx.
		\]
		Using \eqref{eq17} and the continuity equation again, we obtain
		\[
		\int_{\mathbb{R}^N}\nabla\rho^\gamma\cdot u\,dx
		=
		\int_{\mathbb{R}^N}\rho\nabla\Pi'(\rho)\cdot u\,dx
		=
		-\int_{\mathbb{R}^N}\Pi'(\rho)\operatorname{div}(\rho u)\,dx
		=
		\int_{\mathbb{R}^N}\Pi'(\rho)\rho_t\,dx
		=
		\frac{d}{dt}\int_{\mathbb{R}^N}\Pi(\rho)\,dx.
		\]
		Furthermore, by \eqref{eq4} and the identity
		\[
		\mathbb{D}u:\nabla u=|\mathbb{D}u|^2,
		\]
		we have
		\[
		\int_{\mathbb{R}^N}\operatorname{div}\bigl(2\mu(\rho)\mathbb{D}u\bigr)\cdot u\,dx
		=
		-2\nu\int_{\mathbb{R}^N}\rho^\alpha|\mathbb{D}u|^2\,dx,
		\]
		and
		\[
		\int_{\mathbb{R}^N}\nabla\bigl(\lambda(\rho)\operatorname{div}u\bigr)\cdot u\,dx
		=
		-2\nu(\alpha-1)\int_{\mathbb{R}^N}\rho^\alpha(\operatorname{div}u)^2\,dx.
		\]
		
		For the Korteweg term, it follows from \eqref{eq19} and the continuity equation that
		\[
		\begin{aligned}
			\int_{\mathbb{R}^N}\operatorname{div}\mathbb{K}\cdot u\,dx
			&=
			-\frac{2\varepsilon^2\alpha^2}{2\alpha-1}
			\int_{\mathbb{R}^N}
			\rho^{\alpha-\frac32}\Delta\rho^{\alpha-\frac12}\,
			\operatorname{div}(\rho u)\,dx \\
			&=
			\frac{2\varepsilon^2\alpha^2}{2\alpha-1}
			\int_{\mathbb{R}^N}
			\rho^{\alpha-\frac32}\Delta\rho^{\alpha-\frac12}\,\rho_t\,dx.
		\end{aligned}
		\]
		Noting that
		\[
		\partial_t\rho^{\alpha-\frac12}
		=
		\frac{2\alpha-1}{2}\rho^{\alpha-\frac32}\rho_t,
		\]
		we infer that
		\[
		\begin{aligned}
			\int_{\mathbb{R}^N}\operatorname{div}\mathbb{K}\cdot u\,dx
			&=
			\frac{4\varepsilon^2\alpha^2}{(2\alpha-1)^2}
			\int_{\mathbb{R}^N}
			\Delta\rho^{\alpha-\frac12}\,
			\partial_t\rho^{\alpha-\frac12}\,dx \\
			&=
			-\frac{2\varepsilon^2\alpha^2}{(2\alpha-1)^2}
			\frac{d}{dt}
			\int_{\mathbb{R}^N}\left|\nabla\rho^{\alpha-\frac12}\right|^2\,dx.
		\end{aligned}
		\]
		Collecting the above identities yields \eqref{eq20}.
		
		Integrating \eqref{eq20} over $(0,t)$ and using \eqref{eq18}, we obtain
		\[
		\begin{aligned}
			&\int_{\mathbb{R}^N}
			\left(
			\frac12\rho|u|^2+\Pi(\rho)
			+\frac{2\varepsilon^2\alpha^2}{(2\alpha-1)^2}
			\left|\nabla\rho^{\alpha-\frac12}\right|^2
			\right)(x,t)\,dx \\
			&\qquad
			+2\nu\int_0^t\int_{\mathbb{R}^N}\rho^\alpha
			\Bigl(
			|\mathbb{D}u|^2+(\alpha-1)(\operatorname{div}u)^2
			\Bigr)\,dx\,ds
			=E_0.
		\end{aligned}
		\]
		On the other hand, since
		\[
		(\operatorname{div}u)^2=(\operatorname{tr}\mathbb{D}u)^2\le N|\mathbb{D}u|^2,
		\]
		it follows that
		\[
		|\mathbb{D}u|^2+(\alpha-1)(\operatorname{div}u)^2
		\ge (N\alpha-N+1)|\mathbb{D}u|^2.
		\]
		Because $\alpha>\frac{N-1}{N}$, the coefficient $N\alpha-N+1$ is strictly positive. Therefore,
		\[
		\begin{aligned}
			&\sup_{0\le t\le T}\int_{\mathbb{R}^N}
			\left(
			\frac12\rho|u|^2+\Pi(\rho)
			+\frac{2\varepsilon^2\alpha^2}{(2\alpha-1)^2}
			\left|\nabla\rho^{\alpha-\frac12}\right|^2
			\right)\,dx \\
			&\qquad
			+2\nu(N\alpha-N+1)\int_0^T\int_{\mathbb{R}^N}\rho^\alpha|\mathbb{D}u|^2\,dx\,dt
			\le E_0.
		\end{aligned}
		\]
		Absorbing the fixed positive coefficients into the constant gives \eqref{est: basic}.
	\end{proof}
	We next derive the BD entropy estimate in the effective velocity variables. Since we work with the relative internal energy $\Pi(\rho)$, the estimate is fully compatible with the far-field condition \eqref{eq15}.
	
	\begin{proposition}\label{prop2}
		Let $N\in\{2,3\}$, let $\gamma\ge 1$, and assume that
		\[
		\alpha\in\left(\frac{N-1}{N},\,1\right).
		\]
		Let $\Pi(\rho)$ and $E_0$ be defined by \eqref{eq16} and \eqref{eq18}, respectively. Let $(\rho,v)$ be a sufficiently smooth solution to \eqref{sys of v} on $\mathbb{R}^N\times[0,T]$ such that $\rho>0$ and the far-field condition \eqref{eq15} holds. Then, for every $t\in[0,T]$,
		\begin{equation}\label{eq22}
			\begin{aligned}
				&\int_{\mathbb{R}^N}\left(\frac12\rho|v|^2+\Pi(\rho)\right)(x,t)\,dx
				+c\gamma\alpha\int_0^t\int_{\mathbb{R}^N}\rho^{\gamma+\alpha-3}|\nabla\rho|^2\,dx\,ds\\
				&\qquad
				+\mu_{\mathrm{GE}}\int_0^t\int_{\mathbb{R}^N}\rho^\alpha|\nabla v|^2\,dx\,ds
				\le
				\int_{\mathbb{R}^N}\left(\frac12\rho_0|v_0|^2+\Pi(\rho_0)\right)\,dx,
			\end{aligned}
		\end{equation}
		where
		\begin{equation}\label{eq23}
			\mu_{\mathrm{GE}}
			:=
			(2\nu-c)(N\alpha-N+1)
			=
			\nu(1-\beta)(N\alpha-N+1)>0.
		\end{equation}
		Consequently, there exists a positive constant
		\[
		C=C(N,\alpha,\gamma,\nu,\varepsilon,E_0)
		\]
		such that
		\begin{equation}\label{eq24}
			\sup_{0\le t\le T}\int_{\mathbb{R}^N}\rho|v|^2\,dx
			+
			\int_0^T\int_{\mathbb{R}^N}\rho^{\gamma+\alpha-3}|\nabla\rho|^2\,dx\,dt
			+
			\int_0^T\int_{\mathbb{R}^N}\rho^\alpha|\nabla v|^2\,dx\,dt
			\le C.
		\end{equation}
		More precisely, one may take
		\begin{equation}\label{eq25}
			C
			=
			\frac{2}{1-\beta}
			\max\left\{2,\,\frac{1}{c\gamma\alpha},\,\frac{1}{\mu_{\mathrm{GE}}}\right\}E_0.
		\end{equation}
	\end{proposition}
	
	\begin{proof}
		Since
		\begin{equation}\label{eq26}
			\rho_t+\operatorname{div}(\rho u)=0,
		\end{equation}
		we test the momentum equation in \eqref{sys of v} against $v$ and integrate over $\mathbb{R}^N$. Owing to \eqref{eq26},
		\[
		\int_{\mathbb{R}^N}\bigl(\rho v_t+\rho u\cdot\nabla v\bigr)\cdot v\,dx
		=
		\frac{d}{dt}\int_{\mathbb{R}^N}\frac12\rho|v|^2\,dx.
		\]
		On the other hand, by \eqref{eq17} and the first equation in \eqref{sys of v},
		\[
		\begin{aligned}
			\frac{d}{dt}\int_{\mathbb{R}^N}\Pi(\rho)\,dx
			&=
			\int_{\mathbb{R}^N}\Pi'(\rho)\rho_t\,dx\\
			&=
			-\int_{\mathbb{R}^N}\Pi'(\rho)\operatorname{div}(\rho v)\,dx
			+c\int_{\mathbb{R}^N}\Pi'(\rho)\Delta\rho^\alpha\,dx\\
			&=
			\int_{\mathbb{R}^N}\rho v\cdot\nabla\Pi'(\rho)\,dx
			-c\int_{\mathbb{R}^N}\nabla\Pi'(\rho)\cdot\nabla\rho^\alpha\,dx\\
			&=
			\int_{\mathbb{R}^N}v\cdot\nabla\rho^\gamma\,dx
			-c\gamma\alpha\int_{\mathbb{R}^N}\rho^{\gamma+\alpha-3}|\nabla\rho|^2\,dx.
		\end{aligned}
		\]
		Integrating by parts in the viscous terms, we obtain
		\[
		\begin{aligned}
			&\frac{d}{dt}\int_{\mathbb{R}^N}\left(\frac12\rho|v|^2+\Pi(\rho)\right)\,dx
			+c\gamma\alpha\int_{\mathbb{R}^N}\rho^{\gamma+\alpha-3}|\nabla\rho|^2\,dx\\
			&\qquad
			+\int_{\mathbb{R}^N}\rho^\alpha Q(\nabla v)\,dx
			=0,
		\end{aligned}
		\]
		where
		\[
		Q(A):=
		\nu|A|^2+(\nu-c)A:A^\top+(\alpha-1)(2\nu-c)(\operatorname{tr}A)^2.
		\]
		
		We now estimate the quadratic form $Q$. For any $A\in\mathbb{R}^{N\times N}$, set
		\[
		A_s:=\frac{A+A^\top}{2},
		\qquad
		A_a:=\frac{A-A^\top}{2},
		\qquad
		A_0:=A_s-\frac{\operatorname{tr}A}{N}I.
		\]
		Then
		\[
		A=A_a+A_0+\frac{\operatorname{tr}A}{N}I,
		\qquad
		|A|^2=|A_a|^2+|A_0|^2+\frac{(\operatorname{tr}A)^2}{N},
		\]
		and
		\[
		A:A^\top=-|A_a|^2+|A_0|^2+\frac{(\operatorname{tr}A)^2}{N}.
		\]
		Hence
		\[
		\begin{aligned}
			Q(A)
			&=
			c|A_a|^2
			+(2\nu-c)|A_0|^2
			+\frac{(2\nu-c)(N\alpha-N+1)}{N}(\operatorname{tr}A)^2.
		\end{aligned}
		\]
		Since
		\[
		0<N\alpha-N+1<1,
		\qquad
		2\nu-c=\nu(1-\beta)>0,
		\qquad
		c=\nu(1+\beta)\ge 2\nu-c,
		\]
		it follows that
		\[
		Q(A)\ge \mu_{\mathrm{GE}}|A|^2
		\qquad\text{for all }A\in\mathbb{R}^{N\times N},
		\]
		with $\mu_{\mathrm{GE}}$ given by \eqref{eq23}. Taking $A=\nabla v$, we arrive at
		\[
		\frac{d}{dt}\int_{\mathbb{R}^N}\left(\frac12\rho|v|^2+\Pi(\rho)\right)\,dx
		+c\gamma\alpha\int_{\mathbb{R}^N}\rho^{\gamma+\alpha-3}|\nabla\rho|^2\,dx
		+\mu_{\mathrm{GE}}\int_{\mathbb{R}^N}\rho^\alpha|\nabla v|^2\,dx
		\le 0.
		\]
		Integrating this inequality over $(0,t)$ yields \eqref{eq22}.
		
		It remains to estimate the initial term on the right-hand side. Set
		\[
		w_0:=c\alpha\rho_0^{\alpha-2}\nabla\rho_0,
		\qquad
		v_0=u_0+w_0.
		\]
		For any $\theta>0$, Young's inequality gives
		\[
		\frac12\rho_0|v_0|^2
		\le
		(1+\theta)\frac12\rho_0|u_0|^2
		+
		(1+\theta^{-1})\frac12\rho_0|w_0|^2.
		\]
		Moreover,
		\[
		\frac12\rho_0|w_0|^2
		=
		\frac12c^2\alpha^2\rho_0^{2\alpha-3}|\nabla\rho_0|^2
		=
		\frac{c^2}{\varepsilon^2}\,
		\frac{2\varepsilon^2\alpha^2}{(2\alpha-1)^2}
		\left|\nabla\rho_0^{\alpha-\frac12}\right|^2.
		\]
		Choosing
		\[
		\theta=\frac{c^2}{\varepsilon^2},
		\]
		and using
		\[
		\frac{c^2}{\varepsilon^2}
		=
		\frac{\nu^2(1+\beta)^2}{\nu^2(1-\beta^2)}
		=
		\frac{1+\beta}{1-\beta},
		\]
		we infer that
		\[
		1+\theta
		=
		(1+\theta^{-1})\frac{c^2}{\varepsilon^2}
		=
		\frac{2}{1-\beta}.
		\]
		Therefore,
		\[
		\int_{\mathbb{R}^N}\frac12\rho_0|v_0|^2\,dx
		\le
		\frac{2}{1-\beta}
		\int_{\mathbb{R}^N}
		\left(
		\frac12\rho_0|u_0|^2
		+
		\frac{2\varepsilon^2\alpha^2}{(2\alpha-1)^2}
		\left|\nabla\rho_0^{\alpha-\frac12}\right|^2
		\right)\,dx.
		\]
		Since $\Pi(\rho_0)\ge 0$, it follows from \eqref{eq18} that
		\[
		\int_{\mathbb{R}^N}\left(\frac12\rho_0|v_0|^2+\Pi(\rho_0)\right)\,dx
		\le
		\frac{2}{1-\beta}E_0.
		\]
		
		Finally, we derive \eqref{eq24} and \eqref{eq25}. Since $\Pi(\rho)\ge 0$ by \eqref{eq16}, if we set
		\[
		M:=\max\left\{2,\,\frac{1}{c\gamma\alpha},\,\frac{1}{\mu_{\mathrm{GE}}}\right\},
		\]
		then \eqref{eq22} implies
		\[
		\begin{aligned}
			&\int_{\mathbb{R}^N}\rho|v|^2(x,t)\,dx
			+\int_0^t\int_{\mathbb{R}^N}\rho^{\gamma+\alpha-3}|\nabla\rho|^2\,dx\,ds
			+\int_0^t\int_{\mathbb{R}^N}\rho^\alpha|\nabla v|^2\,dx\,ds\\
			&\qquad
			\le
			M\int_{\mathbb{R}^N}\left(\frac12\rho_0|v_0|^2+\Pi(\rho_0)\right)\,dx
			\le
			\frac{2M}{1-\beta}E_0.
		\end{aligned}
		\]
		Taking the supremum over $t\in[0,T]$ completes the proof.
	\end{proof}
	
	For later use, we record two elementary estimates for the relative internal energy density $\Pi(\rho)$.
	
	\begin{lemma}\label{lem1}
		Let $\Pi(\rho)$ be defined by \eqref{eq16}. Then there exist positive constants
		\[
		c_{\mathrm{loc}}(\gamma)>0,
		\qquad
		C_{\mathrm{loc}}(\gamma)>0,
		\]
		depending only on $\gamma$, such that for every $\rho\in[0,4]$,
		\begin{equation}\label{eq27}
			c_{\mathrm{loc}}(\gamma)(\rho-1)^2
			\le
			\Pi(\rho)
			\le
			C_{\mathrm{loc}}(\gamma)(\rho-1)^2.
		\end{equation}
	\end{lemma}
	
	\begin{proof}
		Define
		\[
		H_\gamma(\rho):=
		\begin{cases}
			\dfrac{\Pi(\rho)}{(\rho-1)^2}, & \rho\neq 1,\\[2mm]
			\dfrac12\Pi''(1), & \rho=1.
		\end{cases}
		\]
		Clearly, $H_\gamma$ is continuous on $[0,4]\setminus\{1\}$. Moreover, by Taylor's theorem,
		\[
		\Pi(\rho)=\frac12\Pi''(1)(\rho-1)^2+o\bigl((\rho-1)^2\bigr)
		\qquad (\rho\to1),
		\]
		and hence $H_\gamma$ is also continuous at $\rho=1$. Therefore,
		\[
		H_\gamma\in C([0,4]).
		\]
		
		On the other hand,
		\[
		\Pi''(\rho)=
		\begin{cases}
			\gamma\rho^{\gamma-2}, & \gamma>1,\\[1mm]
			\dfrac1\rho, & \gamma=1,
		\end{cases}
		\qquad \rho>0,
		\]
		so $\Pi$ is strictly convex on $(0,\infty)$. Since
		\[
		\Pi(1)=0,
		\qquad
		\Pi'(1)=0,
		\]
		it follows that $\rho=1$ is the unique minimizer of $\Pi$ on $(0,\infty)$. Noting also that
		\[
		\Pi(0)=1
		\]
		(with the usual convention $0\log 0=0$ when $\gamma=1$), we obtain
		\[
		\Pi(\rho)>0
		\qquad
		\text{for all }\rho\in[0,4]\setminus\{1\}.
		\]
		Since
		\[
		H_\gamma(1)=\frac12\Pi''(1)>0,
		\]
		we conclude that
		\[
		H_\gamma(\rho)>0
		\qquad
		\text{for all }\rho\in[0,4].
		\]
		By the compactness of $[0,4]$, we may define
		\[
		c_{\mathrm{loc}}(\gamma):=\min_{0\le \rho\le 4}H_\gamma(\rho)>0,
		\qquad
		C_{\mathrm{loc}}(\gamma):=\max_{0\le \rho\le 4}H_\gamma(\rho)<\infty.
		\]
		This gives \eqref{eq27}.
	\end{proof}
	
	\begin{lemma}\label{lem2}
		Let $\Pi(\rho)$ be defined by \eqref{eq16}. Then, for every $\rho\ge 4$,
		\begin{equation}\label{eq28}
			\Pi(\rho)\ge c_{\mathrm{high}}(\gamma)\rho^\gamma,
		\end{equation}
		where
		\begin{equation}\label{eq29}
			c_{\mathrm{high}}(\gamma)
			:=
			\frac{\Pi(4)}{4^\gamma}
			=
			\begin{cases}
				\log 4-\dfrac34, & \gamma=1,\\[2mm]
				\dfrac{4^\gamma-1-3\gamma}{(\gamma-1)4^\gamma}, & \gamma>1.
			\end{cases}
		\end{equation}
		In particular, for every measurable set
		\[
		A\subset\{x\in\mathbb{R}^N:\rho(x)\ge 4\},
		\]
		one has
		\begin{equation}\label{eq30}
			\int_A \rho^\gamma\,dx
			\le
			c_{\mathrm{high}}(\gamma)^{-1}\int_A \Pi(\rho)\,dx.
		\end{equation}
	\end{lemma}
	
	\begin{proof}
		For $\rho>0$, define
		\[
		G_\gamma(\rho):=\frac{\Pi(\rho)}{\rho^\gamma},
		\]
		where, in the case $\gamma=1$, this simply means $G_1(\rho)=\Pi(\rho)/\rho$. A direct computation from \eqref{eq16} shows that, in both cases $\gamma=1$ and $\gamma>1$,
		\[
		G_\gamma'(\rho)=\gamma\rho^{-\gamma-1}(\rho-1),
		\qquad \rho>0.
		\]
		Hence $G_\gamma$ is strictly increasing on $[1,\infty)$. Therefore, for every $\rho\ge 4$,
		\[
		\frac{\Pi(\rho)}{\rho^\gamma}
		=
		G_\gamma(\rho)
		\ge
		G_\gamma(4)
		=
		\frac{\Pi(4)}{4^\gamma}
		=
		c_{\mathrm{high}}(\gamma),
		\]
		which proves \eqref{eq28}. Formula \eqref{eq29} follows immediately from \eqref{eq16}, and moreover
		\[
		c_{\mathrm{high}}(\gamma)=G_\gamma(4)>0.
		\]
		Finally, integrating \eqref{eq28} over any measurable set
		\[
		A\subset\{x\in\mathbb{R}^N:\rho(x)\ge 4\}
		\]
		yields \eqref{eq30}.
	\end{proof}
	
	The previous two lemmas show that the relative internal energy is coercive both near the far-field state and in the high-density regime. Combined with the gradient control furnished by Proposition \ref{prop1}, this yields the following uniform $H^1$-estimate for $\rho^{\alpha-\frac12}-1$, which will be used repeatedly in the sequel.
	
	\begin{proposition}\label{prop3}
		Let $N\in\{2,3\}$. Assume that \eqref{eq4}, \eqref{eq5}, \eqref{eq6}, \eqref{eq7}, and Assumption \ref{ass1} hold. Set
		\[
		m:=\alpha-\frac12\in\left(0,\frac12\right).
		\]
		Then, for every $T>0$, there exists a positive constant
		\[
		C=C(N,\alpha,\gamma,\nu,\beta)
		\]
		such that
		\begin{equation}\label{eq31}
			\left\|\rho^m-1\right\|_{L^\infty(0,T;H^1(\mathbb{R}^N))}
			\le
			CE_0^{1/2}.
		\end{equation}
	\end{proposition}
	
	\begin{proof}
		By the energy identity \eqref{eq20} in Proposition \ref{prop1},
		\begin{equation}\label{eq32}
			\sup_{0\le t\le T}\int_{\mathbb{R}^N}\Pi(\rho)(x,t)\,dx
			\le
			E_0.
		\end{equation}
		Moreover,
		\begin{equation}\label{eq33}
			\sup_{0\le t\le T}\int_{\mathbb{R}^N}\left|\nabla\rho^m\right|^2\,dx
			\le
			\frac{(2\alpha-1)^2}{2\varepsilon^2\alpha^2}E_0
			=
			\frac{(2\alpha-1)^2}{2\nu^2(1-\beta^2)\alpha^2}E_0,
		\end{equation}
		where the last identity follows from \eqref{eq8}.
		
		Fix $t\in[0,T]$, and define
		\[
		\Omega_1(t):=\{x\in\mathbb{R}^N:\rho(x,t)\le 4\},
		\qquad
		\Omega_2(t):=\{x\in\mathbb{R}^N:\rho(x,t)>4\}.
		\]
		Since the function
		\[
		z\longmapsto \frac{z^m-1}{z-1}
		\]
		extends continuously to $[0,4]$, there exists a constant $C>0$ such that
		\[
		|\rho^m-1|\le C|\rho-1|
		\qquad\text{on }\Omega_1(t).
		\]
		Hence, by Lemma \ref{lem1} and \eqref{eq32},
		\[
		\int_{\Omega_1(t)}|\rho^m-1|^2\,dx
		\le
		C\int_{\Omega_1(t)}|\rho-1|^2\,dx
		\le
		C\int_{\Omega_1(t)}\Pi(\rho)\,dx
		\le
		CE_0.
		\]
		
		On the other hand, on $\Omega_2(t)$ we have $\rho\ge 4$. Since
		\[
		2m=2\alpha-1<1\le \gamma,
		\]
		it follows that
		\[
		|\rho^m-1|^2
		\le
		\rho^{2m}
		\le
		4^{\,2m-\gamma}\rho^\gamma.
		\]
		Therefore, by Lemma \ref{lem2} and \eqref{eq32},
		\[
		\int_{\Omega_2(t)}|\rho^m-1|^2\,dx
		\le
		C\int_{\Omega_2(t)}\rho^\gamma\,dx
		\le
		C\int_{\Omega_2(t)}\Pi(\rho)\,dx
		\le
		CE_0.
		\]
		Combining the estimates on $\Omega_1(t)$ and $\Omega_2(t)$, we obtain
		\[
		\sup_{0\le t\le T}\|\rho^m-1\|_{L^2(\mathbb{R}^N)}^2
		\le
		CE_0.
		\]
		Together with \eqref{eq33}, this yields
		\[
		\|\rho^m-1\|_{L^\infty(0,T;H^1(\mathbb{R}^N))}
		\le
		CE_0^{1/2},
		\]
		which proves \eqref{eq31}.
		
		For later use, we also record that $\Pi$ is increasing on $[1,\infty)$, and hence, on $\Omega_2(t)$,
		\[
		\Pi(\rho)\ge \Pi(4).
		\]
		Therefore,
		\begin{equation}\label{eq34}
			|\Omega_2(t)|
			\le
			\Pi(4)^{-1}\int_{\Omega_2(t)}\Pi(\rho)\,dx
			\le
			\Pi(4)^{-1}E_0.
		\end{equation}
	\end{proof}
	In the three-dimensional case, we next derive a higher-order spacetime estimate for suitable powers of the density. Combined later with Proposition \ref{prop3}, this bound will provide the regularity input needed for the derivation of the uniform upper bound in the next subsection.
	\begin{proposition}\label{prop4}
		Assume that $N=3$, and that \eqref{eq4}, \eqref{eq5}, \eqref{eq6}, \eqref{eq7}, and Assumption \ref{ass1} hold.
		Let $(\rho,u)$ be a sufficiently smooth solution to \eqref{eq1} on $\mathbb{R}^3\times[0,T]$ such that $\rho>0$, and let $v$ be defined by \eqref{eq12}.
		Then there exists a positive constant
		\[
		C=C(\alpha,\nu,\varepsilon,E_0)
		\]
		such that
		\begin{equation}\label{eq35}
			\int_0^T\int_{\mathbb{R}^3}
			\left|
			\nabla \rho^{\frac{3\alpha-2}{4}}
			\right|^4\,dx\,dt
			+
			\int_0^T\int_{\mathbb{R}^3}
			\left|
			\nabla^2 \rho^{\frac32\alpha-1}
			\right|^2\,dx\,dt
			\le C.
		\end{equation}
	\end{proposition}
	
	\begin{proof}
		All integrations by parts below are justified by the smoothness of the solution together with the far-field conditions \eqref{eq5} and \eqref{eq15}. Throughout the proof, $C$ denotes a positive constant depending only on $\alpha$, $\nu$, $\varepsilon$, and $E_0$.
		
		Set
		\[
		\phi:=\rho^{\alpha-1},
		\qquad
		\theta:=\frac{c\alpha}{1-\alpha}.
		\]
		By \eqref{eq12},
		\[
		u=v+\theta\nabla\phi,
		\qquad
		\nabla v=\nabla u-\theta\nabla^2\phi.
		\]
		Since $\nabla^2\phi$ is symmetric, we have
		\[
		\nabla u:\nabla^2\phi=\mathbb{D}u:\nabla^2\phi.
		\]
		Hence
		\begin{equation}\label{est: nabla2 phi weighted}
			\begin{aligned}
				\rho^\alpha|\nabla v|^2
				&=
				\rho^\alpha|\nabla u|^2
				+\theta^2\rho^\alpha|\nabla^2\phi|^2
				-2\theta\rho^\alpha \mathbb{D}u:\nabla^2\phi \\
				&\ge
				\frac{\theta^2}{2}\rho^\alpha|\nabla^2\phi|^2
				-2\rho^\alpha|\mathbb{D}u|^2,
			\end{aligned}
		\end{equation}

		where we used Young's inequality in the last step. Integrating over $\mathbb{R}^3\times(0,T)$, we obtain
		\[
		\int_0^T\int_{\mathbb{R}^3}\rho^\alpha|\nabla^2\phi|^2\,dx\,dt
		\le
		C\int_0^T\int_{\mathbb{R}^3}\rho^\alpha|\nabla v|^2\,dx\,dt
		+
		C\int_0^T\int_{\mathbb{R}^3}\rho^\alpha|\mathbb{D}u|^2\,dx\,dt.
		\]
		By \eqref{est: basic} and \eqref{eq24},
		\[
		\int_0^T\int_{\mathbb{R}^3}\rho^\alpha|\mathbb{D}u|^2\,dx\,dt
		+
		\int_0^T\int_{\mathbb{R}^3}\rho^\alpha|\nabla v|^2\,dx\,dt
		\le C.
		\]
		Therefore,
		\[
		M:=
		\int_0^T\int_{\mathbb{R}^3}\rho^\alpha|\nabla^2\phi|^2\,dx\,dt
		\le C.
		\]
		
		We now convert this weighted Hessian bound into \eqref{eq35}. Set
		\[
		w:=\rho^{\frac32\alpha-1},
		\qquad
		z:=\rho^{\frac{3\alpha-2}{4}}.
		\]
		Since
		\[
		\rho^{-\frac{\alpha}{2}}\nabla w
		=
		\frac{3\alpha-2}{2(\alpha-1)}\nabla\phi,
		\]
		we infer that
		\[
		\left(\frac{3\alpha-2}{2(1-\alpha)}\right)^2 M
		=
		\int_0^T\int_{\mathbb{R}^3}
		\rho^\alpha
		\left|
		\nabla\!\left(\rho^{-\frac{\alpha}{2}}\nabla w\right)
		\right|^2\,dx\,dt.
		\]
		Expanding the square gives
		\[
		\left(\frac{3\alpha-2}{2(1-\alpha)}\right)^2 M
		=
		\int_0^T\int_{\mathbb{R}^3}|\nabla^2 w|^2\,dx\,dt
		+
		A
		+
		I,
		\]
		where
		\[
		A:=
		\int_0^T\int_{\mathbb{R}^3}
		\rho^\alpha
		\left|
		\nabla\rho^{-\frac{\alpha}{2}}
		\right|^2
		|\nabla w|^2\,dx\,dt,
		\]
		and
		\[
		I:=
		2\int_0^T\int_{\mathbb{R}^3}
		\rho^{\frac{\alpha}{2}}
		\nabla\rho^{-\frac{\alpha}{2}}
		\cdot\nabla^2 w\cdot\nabla w\,dx\,dt.
		\]
		A direct computation yields
		\[
		A
		=
		\frac{16\alpha^2}{(3\alpha-2)^2}
		\int_0^T\int_{\mathbb{R}^3}|\nabla z|^4\,dx\,dt.
		\]
		
		It remains to estimate the cross term $I$. Since
		\[
		\rho^{\frac{\alpha}{2}}\nabla\rho^{-\frac{\alpha}{2}}
		=
		-\frac{\alpha}{2}\rho^{-1}\nabla\rho,
		\]
		an integration by parts gives
		\[
		I
		=
		\frac{\alpha}{2}\int_0^T\int_{\mathbb{R}^3}
		\operatorname{div}(\rho^{-1}\nabla\rho)\,|\nabla w|^2\,dx\,dt
		=
		\frac{\alpha}{2(\alpha-1)}
		\int_0^T\int_{\mathbb{R}^3}
		\operatorname{div}(\rho^{1-\alpha}\nabla\phi)\,|\nabla w|^2\,dx\,dt.
		\]
		Moreover,
		\[
		\nabla\rho^{1-\alpha}\cdot\nabla\phi
		=
		-(\alpha-1)^2\rho^{-2}|\nabla\rho|^2
		\le 0.
		\]
		Since $\alpha<1$, the contribution of this term is nonnegative, and therefore
		\[
		I
		\ge
		\frac{\alpha}{2(\alpha-1)}
		\int_0^T\int_{\mathbb{R}^3}
		\rho^{1-\alpha}\Delta\phi\,|\nabla w|^2\,dx\,dt.
		\]
		Hence, for any $\delta>0$, Young's inequality yields
		\[
		I
		\ge
		-C_\alpha
		\int_0^T\int_{\mathbb{R}^3}
		\rho^{1-\alpha}|\Delta\phi|\,|\nabla w|^2\,dx\,dt
		\]
		and
		\[
		\rho^{1-\alpha}|\Delta\phi|\,|\nabla w|^2
		\le
		\delta\,\rho^\alpha|\Delta\phi|^2
		+
		\frac{1}{4\delta}\rho^{2-3\alpha}|\nabla w|^4.
		\]
		In dimension three,
		\[
		|\Delta\phi|^2\le 3|\nabla^2\phi|^2,
		\]
		while
		\[
		\rho^{2-3\alpha}|\nabla w|^4
		=
		\frac{(3\alpha-2)^2}{\alpha^2}
		\rho^\alpha
		\left|
		\nabla\rho^{-\frac{\alpha}{2}}
		\right|^2
		|\nabla w|^2.
		\]
		Consequently,
		\[
		I\ge -C_\alpha\delta\,M-\frac{C_\alpha}{\delta}\,A,
		\]
		where $C_\alpha>0$ depends only on $\alpha$. Choosing $\delta=\delta(\alpha)$ so that
		\[
		\frac{C_\alpha}{\delta}=\frac12,
		\]
		we obtain
		\[
		I\ge -C_\alpha M-\frac12 A.
		\]
		
		Substituting this estimate into the preceding expansion, we arrive at
		\[
		\int_0^T\int_{\mathbb{R}^3}|\nabla^2 w|^2\,dx\,dt
		+\frac12 A
		\le
		C_\alpha M
		\le C.
		\]
		Recalling the expression for $A$ and the definitions
		\[
		w=\rho^{\frac32\alpha-1},
		\qquad
		z=\rho^{\frac{3\alpha-2}{4}},
		\]
		we conclude that
		\[
		\int_0^T\int_{\mathbb{R}^3}
		\left|
		\nabla \rho^{\frac{3\alpha-2}{4}}
		\right|^4\,dx\,dt
		+
		\int_0^T\int_{\mathbb{R}^3}
		\left|
		\nabla^2 \rho^{\frac32\alpha-1}
		\right|^2\,dx\,dt
		\le C,
		\]
		which is exactly \eqref{eq35}. The proof is complete.
	\end{proof}
	As a direct consequence of Propositions \ref{prop3} and \ref{prop4}, we obtain the following spacetime $H^2$-estimate for $\rho^{\frac32\alpha-1}-1$ in dimension three.
	\begin{corollary}\label{cor1}
		Assume that $N=3$, and that \eqref{eq4}, \eqref{eq5}, \eqref{eq6}, \eqref{eq7}, and Assumption \ref{ass1} hold.
		Let $(\rho,u)$ be a sufficiently smooth strictly positive local strong solution to \eqref{eq1} on $\mathbb{R}^3\times[0,T^\ast)$, and let $v$ be defined by \eqref{eq12}.
		Then, for every $T\in(0,T^\ast)$, there exists a positive constant
		\[
		C=C(\alpha,\gamma,\nu,\beta,E_0,T^\ast)>0
		\]
		such that
		\begin{equation}\label{eq36}
			\left\|
			\rho^{\frac32\alpha-1}-1
			\right\|_{L^2(0,T;H^2(\mathbb{R}^3))}
			\le C.
		\end{equation}
	\end{corollary}
	
	\begin{proof}
		Set
		\[
		m:=\alpha-\frac12,
		\qquad
		\sigma:=\frac32\alpha-1.
		\]
		Since $N=3$, Assumption \ref{ass1} yields
		\[
		\alpha\in\left(\frac23,1\right),
		\]
		and hence
		\[
		0<\sigma<m.
		\]
		Let
		\[
		\theta:=\frac{\sigma}{m}\in(0,1).
		\]
		Then $\rho^\sigma=(\rho^m)^\theta$, and for every $y\ge0$ one has
		\[
		|y^\theta-1|\le |y-1|.
		\]
		Therefore,
		\[
		|\rho^\sigma-1|
		\le
		|\rho^m-1|
		\qquad\text{a.e. in }\mathbb{R}^3\times(0,T).
		\]
		Combining this with \eqref{eq31} in Proposition \ref{prop3}, we obtain
		\[
		\|\rho^\sigma-1\|_{L^\infty(0,T;L^2(\mathbb{R}^3))}
		\le
		\|\rho^m-1\|_{L^\infty(0,T;L^2(\mathbb{R}^3))}
		\le C.
		\]
		
		On the other hand, Proposition \ref{prop4} gives
		\[
		\|\nabla^2(\rho^\sigma-1)\|_{L^2(0,T;L^2(\mathbb{R}^3))}
		=
		\|\nabla^2\rho^\sigma\|_{L^2(0,T;L^2(\mathbb{R}^3))}
		\le C.
		\]
		
		Set
		\[
		f:=\rho^\sigma-1.
		\]
		Then, for almost every $t\in(0,T)$, Plancherel's theorem yields the standard interpolation inequality
		\[
		\|\nabla f(t)\|_{L^2(\mathbb{R}^3)}^2
		\le
		\|f(t)\|_{L^2(\mathbb{R}^3)}
		\|\nabla^2 f(t)\|_{L^2(\mathbb{R}^3)}.
		\]
		Integrating in time and using the Cauchy--Schwarz inequality, we infer that
		\begin{align*}
			\|\nabla f\|_{L^2(0,T;L^2(\mathbb{R}^3))}^2
			&\le
			\|f\|_{L^\infty(0,T;L^2(\mathbb{R}^3))}
			\int_0^T \|\nabla^2 f(t)\|_{L^2(\mathbb{R}^3)}\,dt \\
			&\le
			(T^\ast)^{1/2}
			\|f\|_{L^\infty(0,T;L^2(\mathbb{R}^3))}
			\|\nabla^2 f\|_{L^2(0,T;L^2(\mathbb{R}^3))} \\
			&\le C.
		\end{align*}
		Moreover,
		\[
		\|f\|_{L^2(0,T;L^2(\mathbb{R}^3))}
		\le
		T^{1/2}\|f\|_{L^\infty(0,T;L^2(\mathbb{R}^3))}
		\le
		(T^\ast)^{1/2}\|f\|_{L^\infty(0,T;L^2(\mathbb{R}^3))}
		\le C.
		\]
		Collecting the above bounds, we conclude that
		\[
		\|f\|_{L^2(0,T;H^2(\mathbb{R}^3))}\le C,
		\]
		where $C$ depends only on $\alpha,\gamma,\nu,\beta,E_0$, and $T^\ast$.
		Recalling that $f=\rho^\sigma-1$ and $\sigma=\frac32\alpha-1$, we obtain \eqref{eq36}.
		By \eqref{eq8}, the same dependence may equivalently be expressed in terms of $\alpha,\gamma,\nu,\varepsilon,E_0$, and $T^\ast$.
	\end{proof}
	\subsection{Upper bound for the density}\label{subsec1}
	
	To initiate the proof of the upper bound for the density, we first derive a weighted $L^{q+2}$-estimate for the effective velocity $v$. This estimate will be a key ingredient in the argument below. For the sake of narrative convenience, let
	\begin{equation}\label{eq37}
		q_{N,\ast}
		:=
		\frac{4(1-\beta)(N\alpha-N+1)}
		{\bigl(\beta+\sqrt{N}(1-\alpha)(1-\beta)\bigr)^2}.
	\end{equation}
	In fact, under Assumption \ref{ass1}, one has the bound
	\[
	q_{N,\ast}>2.
	\]
	Indeed, regarding $q_{N,\ast}$ as a function of $\beta$, we write
	\[
	q_{N,\ast}(\beta)
	=
	\frac{4(1-\beta)(N\alpha-N+1)}
	{\bigl(\beta+\sqrt{N}(1-\alpha)(1-\beta)\bigr)^2}
	=
	\frac{4(1-\beta)(N\alpha-N+1)}
	{\bigl(\sqrt{N}(1-\alpha)+(1-\sqrt{N}(1-\alpha))\beta\bigr)^2}.
	\]
	Since Assumption \ref{ass1} implies $\alpha>\frac{N-1}{N}$, we have
	\[
	N\alpha-N+1>0.
	\]
	Moreover, Assumption \ref{ass1} also yields $\alpha>1-\frac1{\sqrt N}$, and therefore
	\[
	\sqrt{N}(1-\alpha)<1.
	\]
	A direct differentiation gives
	\[
	\frac{d}{d\beta}q_{N,\ast}(\beta)
	=
	-\frac{4(N\alpha-N+1)\bigl(2-\beta-\sqrt{N}(1-\alpha)(1-\beta)\bigr)}
	{\bigl(\beta+\sqrt{N}(1-\alpha)(1-\beta)\bigr)^3}
	<0
	\qquad \text{for all } \beta\in[0,1).
	\]
	Hence $q_{N,\ast}(\beta)$ is strictly decreasing with respect to $\beta$.
	
	Next, let
	\[
	P_N(\beta;\alpha)
	:=
	\beta^2
	+
	\frac{2(1-\alpha)(N\alpha^2+\sqrt{N}\alpha+1-N)}
	{\alpha(1-\sqrt{N}(1-\alpha))^2}\,\beta
	+
	\frac{(1-\alpha)(-N\alpha^2-N\alpha+2N-2)}
	{\alpha(1-\sqrt{N}(1-\alpha))^2}.
	\]
	Then, by Assumption \ref{ass1}, $\beta_N^+(\alpha)$ is the unique positive root of
	\[
	P_N(\beta;\alpha)=0,
	\]
	and the admissible parameter range is precisely $0\le \beta<\beta_N^+(\alpha)$. Furthermore,
	\[
	P_N(0;\alpha)
	=
	\frac{(1-\alpha)(-N\alpha^2-N\alpha+2N-2)}
	{\alpha(1-\sqrt{N}(1-\alpha))^2}
	<0.
	\]
	Indeed, for $N=2$, the condition $\alpha>\frac{\sqrt5-1}{2}$ gives
	\[
	2\alpha^2+2\alpha-2>0,
	\]
	while for $N=3$, the condition
	\[
	\alpha>\frac{9\sqrt3-4\sqrt2}{9\sqrt3-2\sqrt2}
	\]
	implies
	\[
	3\alpha^2+3\alpha-4>0.
	\]
	Consequently, since $P_N(\beta;\alpha)$ has leading coefficient $1$, has a unique positive root $\beta_N^+(\alpha)$, and satisfies $P_N(0;\alpha)<0$, we infer that
	\[
	P_N(\beta;\alpha)<0
	\qquad \text{for all } \beta\in[0,\beta_N^+(\alpha)).
	\]
	
	We now use the identity
	\[
	\frac{\alpha}{2}
	\bigl(\beta+\sqrt{N}(1-\alpha)(1-\beta)\bigr)^2
	\bigl(q_{N,\ast}-2\bigr)
	=
	2(2\alpha-1)(1-\beta)(N\alpha-N+1)
	-
	\alpha(1-\sqrt{N}(1-\alpha))^2 P_N(\beta;\alpha),
	\]
	which follows from a straightforward algebraic computation. Under Assumption \ref{ass1}, we have $\alpha>\frac12$, $1-\beta>0$, $N\alpha-N+1>0$, and $-P_N(\beta;\alpha)>0$. Therefore the right-hand side is strictly positive, and thus
	\[
	q_{N,\ast}-2>0.
	\]
	This proves that
	\[
	q_{N,\ast}>2.
	\]
	\begin{proposition}\label{prop5}
		Assume that $N\in\{2,3\}$, and that \eqref{eq4}, \eqref{eq5}, \eqref{eq6}, \eqref{eq7}, and Assumption \ref{ass1} hold.
		Let $(\rho,u)$ be a sufficiently smooth strictly positive local strong solution to \eqref{eq1} on $\mathbb{R}^N\times[0,T^\ast)$, and let $v$ be defined by \eqref{eq12}.
		
		For $N=2$, let $q$ be any exponent satisfying
		\begin{equation}\label{eq38}
			2\le q<q_{2,\ast}.
		\end{equation}
		For $N=3$, assume in addition that
		\begin{equation}\label{prop5:gamma}
			1\le \gamma<\frac{9}{2}\alpha-2,
		\end{equation}
		and choose $q$ such that
		\begin{equation}\label{eq39}
			2\le q<q_{3,\ast},
			\qquad
			\gamma<3\alpha-1+\frac{6\alpha-4}{q+2}.
		\end{equation}
		
		Then, for every $T\in(0,T^\ast)$, there exists a positive constant
		\[
		C=C\!\left(
		N,\alpha,\gamma,\nu,\beta,E_0,T^\ast,q,
		\left\|\rho_0^{\frac{1}{q+2}}v_0\right\|_{L^{q+2}(\mathbb{R}^N)}
		\right)>0
		\]
		such that
		\begin{equation}\label{eq40}
			\sup_{0\le t\le T}\int_{\mathbb{R}^N}\rho|v|^{q+2}(x,t)\,dx
			\le C.
		\end{equation}
		Equivalently,
		\begin{equation}\label{eq41}
			\sup_{0\le t\le T}
			\left\|
			\rho^{\frac{1}{q+2}}v
			\right\|_{L^{q+2}(\mathbb{R}^N)}
			\le C.
		\end{equation}
	\end{proposition}
	
	\begin{proof}
		It is enough to establish \eqref{eq40}. We shall derive a differential inequality for
		\begin{equation}\label{eq42}
			\int_{\mathbb{R}^N}\rho|v|^{q+2}(x,t)\,dx.
		\end{equation}
		
		Set $\Psi=|v|^qv$. Testing the second equation in \eqref{sys of v} against $\Psi$ and using \eqref{eq26}, we obtain
		\begin{equation}\label{eq43}
			\int_{\mathbb{R}^N}
			\bigl(\rho v_t+\rho u\cdot\nabla v\bigr)\cdot\Psi\,dx
			=
			\frac{1}{q+2}\frac{d}{dt}\int_{\mathbb{R}^N}\rho|v|^{q+2}\,dx.
		\end{equation}
		For the pressure term, an integration by parts together with
		\[
		\bigl|\operatorname{div}(|v|^qv)\bigr|\le C_{N,q}|v|^q|\nabla v|
		\]
		gives
		\begin{equation}\label{eq44}
			\left|
			\int_{\mathbb{R}^N}\nabla\rho^\gamma\cdot\Psi\,dx
			\right|
			\le
			C_{N,q}\int_{\mathbb{R}^N}\rho^\gamma|v|^q|\nabla v|\,dx.
		\end{equation}
		
		For the viscous contribution, write $e=v/|v|$ and $\xi=\nabla|v|$ on $\{v\neq0\}$. A direct computation yields
		\begin{align}
			&\nu\int_{\mathbb{R}^N}\rho^\alpha\nabla v:\nabla\Psi\,dx
			+(\nu-c)\int_{\mathbb{R}^N}\rho^\alpha(\nabla v)^\top:\nabla\Psi\,dx
			\notag\\
			&\qquad
			+(\alpha-1)(2\nu-c)\int_{\mathbb{R}^N}\rho^\alpha(\operatorname{div}v)\operatorname{div}\Psi\,dx
			\notag\\
			&=
			\int_{\mathbb{R}^N}\rho^\alpha|v|^q
			\Bigl(
			\nu|\nabla v|^2
			+(\nu-c)\nabla v:(\nabla v)^\top
			+(\alpha-1)(2\nu-c)(\operatorname{div}v)^2
			\Bigr)\,dx
			\notag\\
			&\quad
			+
			q\int_{\mathbb{R}^N}\rho^\alpha|v|^q
			\Bigl(
			\nu|\xi|^2
			+(\nu-c)(\nabla v\,e)\cdot\xi
			+(\alpha-1)(2\nu-c)(\operatorname{div}v)\,e\cdot\xi
			\Bigr)\,dx.
			\label{eq45}
		\end{align}
		By the argument already used in Proposition \ref{prop2}, the quadratic form in the first integral on the right-hand side of \eqref{eq45} is bounded from below by $\mu_{\mathrm{GE}}|\nabla v|^2$. For the second integral we use
		\[
		|(\nabla v\,e)\cdot\xi|\le |\nabla v|\,|\xi|,
		\qquad
		|\operatorname{div}v|\le \sqrt{N}\,|\nabla v|,
		\]
		together with $c-\nu=\nu\beta$ and $2\nu-c=\nu(1-\beta)$, and complete the square. This gives
		\[
		\nu|\xi|^2
		+(\nu-c)(\nabla v\,e)\cdot\xi
		+(\alpha-1)(2\nu-c)(\operatorname{div}v)\,e\cdot\xi
		\ge
		-\frac{\bigl(c-\nu+\sqrt{N}(1-\alpha)(2\nu-c)\bigr)^2}{4\nu}
		|\nabla v|^2.
		\]
		Hence
		\begin{equation}\label{eq46}
			\begin{aligned}
				&\nu\int_{\mathbb{R}^N}\rho^\alpha\nabla v:\nabla\Psi\,dx
				+(\nu-c)\int_{\mathbb{R}^N}\rho^\alpha(\nabla v)^\top:\nabla\Psi\,dx\\
				&\qquad
				+(\alpha-1)(2\nu-c)\int_{\mathbb{R}^N}\rho^\alpha(\operatorname{div}v)\operatorname{div}\Psi\,dx\\
				&\ge
				\delta_q\int_{\mathbb{R}^N}\rho^\alpha|v|^q|\nabla v|^2\,dx,
			\end{aligned}
		\end{equation}
		where
		\begin{equation}\label{eq47}
			\delta_q
			:=
			\nu
			\left[
			(1-\beta)(N\alpha-N+1)
			-
			\frac{q}{4}
			\bigl(\beta+\sqrt{N}(1-\alpha)(1-\beta)\bigr)^2
			\right].
		\end{equation}
		Since \(q<q_{N,\ast}\), by \eqref{eq37} we have \(\delta_q>0\).
		Combining \eqref{eq43}, \eqref{eq44}, and \eqref{eq46}, and then applying Young's inequality to the pressure term, we arrive at
		\begin{equation}\label{eq48}
			\frac{1}{q+2}\frac{d}{dt}\int_{\mathbb{R}^N}\rho|v|^{q+2}\,dx
			+
			\frac{\delta_q}{2}\int_{\mathbb{R}^N}\rho^\alpha|v|^q|\nabla v|^2\,dx
			\le
			C\int_{\mathbb{R}^N}\rho^{2\gamma-\alpha}|v|^q\,dx.
		\end{equation}
		
		Fix $t\in[0,T]$ and decompose
		\[
		\Omega_1(t):=\{x\in\mathbb{R}^N:\rho(x,t)\le 4\},
		\qquad
		\Omega_2(t):=\{x\in\mathbb{R}^N:\rho(x,t)>4\}.
		\]
		Write
		\begin{equation}\label{eq49}
			\int_{\mathbb{R}^N}\rho^{2\gamma-\alpha}|v|^q\,dx
			=
			I_1(t)+I_2(t),
		\end{equation}
		where
		\[
		I_1(t):=\int_{\Omega_1(t)}\rho^{2\gamma-\alpha}|v|^q\,dx,
		\qquad
		I_2(t):=\int_{\Omega_2(t)}\rho^{2\gamma-\alpha}|v|^q\,dx.
		\]
		
		On $\Omega_1(t)$ one has $\rho^{2\gamma-\alpha}\le C\rho$, since $2\gamma-\alpha>1$. Therefore,
		\[
		I_1(t)\le C\int_{\mathbb{R}^N}\rho|v|^q\,dx.
		\]
		By Hölder's inequality,
		\[
		\int_{\mathbb{R}^N}\rho|v|^q\,dx
		\le
		\left(\int_{\mathbb{R}^N}\rho|v|^2\,dx\right)^{\frac{2}{q}}
		\left(\int_{\mathbb{R}^N}\rho|v|^{q+2}\,dx\right)^{\frac{q-2}{q}},
		\]
		with the usual convention that the second factor is $1$ when $q=2$. Using \eqref{eq24}, we infer
		\begin{equation}\label{eq50}
			I_1(t)\le
			C\left(1+\int_{\mathbb{R}^N}\rho|v|^{q+2}\,dx\right).
		\end{equation}
		
		We first consider the case $N=2$. Let
		\[
		m:=\alpha-\frac12\in\left(0,\frac12\right),
		\qquad
		\lambda:=\frac{(2\gamma-\alpha)(q+2)-q}{2}.
		\]
		By Hölder's inequality,
		\[
		I_2(t)
		\le
		\left(\int_{\mathbb{R}^2}\rho|v|^{q+2}\,dx\right)^{\frac{q}{q+2}}
		\left(
		\int_{\Omega_2(t)}\rho^\lambda\,dx
		\right)^{\frac{2}{q+2}}.
		\]
		On $\Omega_2(t)$ we have $\rho^m\le C(\rho^m-1)$, hence
		\[
		\int_{\Omega_2(t)}\rho^\lambda\,dx
		\le
		C\|\rho^m-1\|_{L^{\lambda/m}(\mathbb{R}^2)}^{\lambda/m}.
		\]
		Moreover,
		\[
		\frac{\lambda}{m}
		=
		\frac{(2\gamma-\alpha)(q+2)-q}{2\alpha-1}
		=
		\frac{(2\gamma-\alpha-1)q+4\gamma-2\alpha}{2\alpha-1}
		\ge
		\frac{4-2\alpha}{2\alpha-1}
		>
		2.
		\]
		Thus \eqref{eq31} and the Sobolev embedding $H^1(\mathbb{R}^2)\hookrightarrow L^r(\mathbb{R}^2)$ for all finite $r\ge2$ yield
		\[
		\sup_{0\le t\le T}\int_{\Omega_2(t)}\rho^\lambda\,dx\le C.
		\]
		Consequently,
		\begin{equation}\label{eq51}
			I_2(t)\le
			C\left(\int_{\mathbb{R}^2}\rho|v|^{q+2}\,dx\right)^{\frac{q}{q+2}}
			\le
			C\left(1+\int_{\mathbb{R}^2}\rho|v|^{q+2}\,dx\right).
		\end{equation}
		Combining \eqref{eq48}, \eqref{eq50}, and \eqref{eq51}, we obtain
		\[
		\frac{d}{dt}\int_{\mathbb{R}^2}\rho|v|^{q+2}\,dx
		\le
		C\left(1+\int_{\mathbb{R}^2}\rho|v|^{q+2}\,dx\right).
		\]
		Gronwall's inequality gives \eqref{eq40} in dimension two.
		
		It remains to treat the case $N=3$. Set
		\[
		m:=\alpha-\frac12\in\left(\frac16,\frac12\right),
		\qquad
		\sigma:=\frac32\alpha-1\in\left(0,\frac12\right),
		\qquad
		\lambda:=\frac{(2\gamma-\alpha)(q+2)-q}{2}.
		\]
		As above,
		\begin{equation}\label{eq52}
			I_2(t)
			\le
			\left(\int_{\mathbb{R}^3}\rho|v|^{q+2}\,dx\right)^{\frac{q}{q+2}}
			\left(
			\int_{\Omega_2(t)}\rho^\lambda\,dx
			\right)^{\frac{2}{q+2}}.
		\end{equation}
		
		Choose $\eta\in C^\infty([0,\infty))$ such that
		\[
		0\le \eta\le 1,
		\qquad
		\eta(s)=0 \ \text{for } 0\le s\le 2,
		\qquad
		\eta(s)=1 \ \text{for } s\ge 4,
		\]
		and define
		\[
		g:=\eta(\rho)\rho^\sigma.
		\]
		Then $g=\rho^\sigma$ on $\Omega_2(t)$. Let also
		\[
		\widetilde{\Omega}_2(t):=\{x\in\mathbb{R}^3:\rho(x,t)>2\}.
		\]
		Since $\Pi$ is increasing on $[1,\infty)$,
		\begin{equation}\label{eq53}
			|\widetilde{\Omega}_2(t)|
			\le
			\Pi(2)^{-1}\int_{\widetilde{\Omega}_2(t)}\Pi(\rho)\,dx
			\le
			\Pi(2)^{-1}E_0.
		\end{equation}
		In particular, $\{x\in\mathbb{R}^3:\ g(x,t)\neq0\}\subset \widetilde{\Omega}_2(t)$ for every $t$.
		
		Now let
		\[
		p_0:=\frac{6m}{\sigma}
		=
		\frac{12\alpha-6}{3\alpha-2}>6.
		\]
		On $\{\rho>2\}$ we have $\rho^m\le C(\rho^m-1)$, and therefore
		\[
		|g|^{p_0}\le C|\rho^m-1|^6.
		\]
		Hence, by \eqref{eq31} and the embedding $H^1(\mathbb{R}^3)\hookrightarrow L^6(\mathbb{R}^3)$,
		\begin{equation}\label{eq54}
			\|g\|_{L^\infty(0,T;L^{p_0}(\mathbb{R}^3))}\le C.
		\end{equation}
		Further, \eqref{eq53} and $p_0>6$ imply
		\[
		\|g(t)\|_{L^6(\mathbb{R}^3)}
		\le
		|\widetilde{\Omega}_2(t)|^{\frac16-\frac1{p_0}}
		\|g(t)\|_{L^{p_0}(\mathbb{R}^3)},
		\]
		so $\|g\|_{L^2(0,T;L^6(\mathbb{R}^3))}\le C$. On the other hand,
		\[
		\nabla g
		=
		\eta'(\rho)\rho^\sigma\nabla\rho+\eta(\rho)\nabla\rho^\sigma.
		\]
		Since \(\eta'(\rho)\neq0\) only when \(2\le \rho\le4\), where $\rho$ is uniformly bounded above and below, we have $|\nabla g|\le C|\nabla\rho^\sigma|$. Therefore, by \eqref{eq36} and the embedding $H^1(\mathbb{R}^3)\hookrightarrow L^6(\mathbb{R}^3)$,
		\[
		\|\nabla g\|_{L^2(0,T;L^6(\mathbb{R}^3))}\le C.
		\]
		Thus
		\begin{equation}\label{eq55}
			\|g\|_{L^2(0,T;W^{1,6}(\mathbb{R}^3))}\le C.
		\end{equation}
		
		We next interpolate between \eqref{eq54} and \eqref{eq55}. Choose $\theta\in(0,1)$ and $\zeta>1$ so that
		\begin{equation}\label{eq56}
			\frac{1}{\zeta}
			=
			\frac{\theta}{p_0}
			-
			\frac{1-\theta}{6}.
		\end{equation}
		Then the Gagliardo--Nirenberg inequality gives
		\begin{equation}\label{eq57}
			\|g\|_{L^\zeta(\mathbb{R}^3)}
			\le
			C
			\|g\|_{L^{p_0}(\mathbb{R}^3)}^\theta
			\|g\|_{W^{1,6}(\mathbb{R}^3)}^{1-\theta}.
		\end{equation}
		We now impose
		\begin{equation}\label{eq58}
			\frac{2}{q+2}\zeta=\frac{2}{1-\theta}.
		\end{equation}
		With this choice, \eqref{eq54}, \eqref{eq55}, \eqref{eq57}, and Hölder's inequality in time imply
		\begin{equation}\label{eq59}
			\|g\|_{L^{\frac{2}{q+2}\zeta}(0,T;L^\zeta(\mathbb{R}^3))}\le C.
		\end{equation}
		Solving \eqref{eq56} and \eqref{eq58}, we find
		\[
		\zeta
		=
		\frac{(5\alpha-3)q+22\alpha-12}{3\alpha-2}.
		\]
		Set
		\begin{equation}\label{eq60}
			\zeta'
			:=
			\sigma\zeta
			=
			\frac{(5\alpha-3)q+22\alpha-12}{2}.
		\end{equation}
		Since $g=\rho^\sigma$ on $\Omega_2(t)$, \eqref{eq59} yields
		\begin{equation}\label{eq61}
			\int_0^T
			\left(
			\int_{\Omega_2(t)}\rho^{\zeta'}\,dx
			\right)^{\frac{2}{q+2}}
			dt
			\le C.
		\end{equation}
		
		A short computation shows that \eqref{eq39} is equivalent to $\lambda<\zeta'$. Using \eqref{eq34} and Hölder's inequality on the set $\Omega_2(t)$, we obtain
		\[
		\left(
		\int_{\Omega_2(t)}\rho^\lambda\,dx
		\right)^{\frac{2}{q+2}}
		\le
		C
		\left(
		\int_{\Omega_2(t)}\rho^{\zeta'}\,dx
		\right)^{\frac{2\lambda}{(q+2)\zeta'}}.
		\]
		Substituting this into \eqref{eq52}, we obtain
		\begin{equation}\label{eq62}
			I_2(t)
			\le
			C
			\left(\int_{\mathbb{R}^3}\rho|v|^{q+2}\,dx\right)^{\frac{q}{q+2}}
			\left(
			\int_{\Omega_2(t)}\rho^{\zeta'}\,dx
			\right)^{\frac{2\lambda}{(q+2)\zeta'}}.
		\end{equation}
		
		Combining \eqref{eq48}, \eqref{eq50}, and \eqref{eq62}, we infer
		\[
		\frac{d}{dt}\left(1+\int_{\mathbb{R}^3}\rho|v|^{q+2}\,dx\right)
		\le
		C\left(1+\int_{\mathbb{R}^3}\rho|v|^{q+2}\,dx\right)
		+
		C\left(1+\int_{\mathbb{R}^3}\rho|v|^{q+2}\,dx\right)^{\frac{q}{q+2}}
		\left(
		\int_{\Omega_2(t)}\rho^{\zeta'}\,dx
		\right)^{\frac{2\lambda}{(q+2)\zeta'}}.
		\]
		Equivalently,
		\begin{equation}\label{eq63}
			\frac{d}{dt}\left(1+\int_{\mathbb{R}^3}\rho|v|^{q+2}\,dx\right)^{\frac{2}{q+2}}
			\le
			C\left(1+\int_{\mathbb{R}^3}\rho|v|^{q+2}\,dx\right)^{\frac{2}{q+2}}
			+
			C
			\left(
			\int_{\Omega_2(t)}\rho^{\zeta'}\,dx
			\right)^{\frac{2\lambda}{(q+2)\zeta'}}.
		\end{equation}
		Since $\lambda<\zeta'$, the exponent on the right is strictly smaller than $2/(q+2)$; therefore \eqref{eq61} implies
		\[
		\int_0^T
		\left(
		\int_{\Omega_2(t)}\rho^{\zeta'}\,dx
		\right)^{\frac{2\lambda}{(q+2)\zeta'}}
		dt
		\le C.
		\]
		Gronwall's inequality applied to \eqref{eq63} now yields \eqref{eq40} in dimension three as well. The proof is complete.
	\end{proof}
	The previous proposition treats the case $q\ge 2$. To cover the remaining range $1<q<2$ in dimension three, we next establish the corresponding weighted $L^{q+2}$-estimate for the effective velocity $v$, using the additional regularity provided by Corollary \ref{cor1}.
	\begin{proposition}\label{prop6}
		Assume that $N=3$, and that \eqref{eq4}, \eqref{eq5}, \eqref{eq6}, \eqref{eq7}, and Assumption \ref{ass1} hold.
		Let $(\rho,u)$ be a sufficiently smooth strictly positive local strong solution to \eqref{eq1} on $\mathbb{R}^3\times[0,T^\ast)$, and let $v$ be defined by \eqref{eq12}.
		
		Assume in addition that
		\begin{equation}\label{prop6:gamma}
			\frac32\alpha-\frac16\le \gamma<5\alpha-\frac73.
		\end{equation}
		Choose $q$ such that
		\begin{equation}\label{eq64}
			1<q<2,
			\qquad
			q<q_{3,\ast},
			\qquad
			\gamma<3\alpha-1+\frac{6\alpha-4}{q+2}.
		\end{equation}
		
		Then, for every $T\in(0,T^\ast)$, there exists a positive constant
		\[
		C=C\!\left(
		\alpha,\gamma,\nu,\varepsilon,E_0,T^\ast,q,
		\left\|\rho_0^{\frac{1}{q+2}}v_0\right\|_{L^{q+2}(\mathbb{R}^3)}
		\right)>0
		\]
		such that
		\begin{equation}\label{eq65}
			\sup_{0\le t\le T}\int_{\mathbb{R}^3}\rho|v|^{q+2}(x,t)\,dx\le C.
		\end{equation}
		Equivalently,
		\begin{equation}\label{eq66}
			\sup_{0\le t\le T}
			\left\|
			\rho^{\frac{1}{q+2}}v
			\right\|_{L^{q+2}(\mathbb{R}^3)}
			\le C.
		\end{equation}
	\end{proposition}
	
	\begin{proof}
		It is enough to prove \eqref{eq65}. Arguing exactly as in the proof of Proposition \ref{prop5}, namely by testing the second equation in \eqref{sys of v} against $|v|^qv$, one finds that
		\[
		\frac{1}{q+2}\frac{d}{dt}\int_{\mathbb{R}^3}\rho|v|^{q+2}\,dx
		+\delta_q\int_{\mathbb{R}^3}\rho^\alpha|v|^q|\nabla v|^2\,dx
		\le
		C\left|\int_{\mathbb{R}^3}\nabla\rho^\gamma\cdot |v|^qv\,dx\right|,
		\]
		where $\delta_q$ is given by \eqref{eq47}. Since $q<q_{3,\ast}$, \eqref{eq37} yields $\delta_q>0$.
		
		Let
		\[
		\Omega_1(t):=\{x\in\mathbb{R}^3:\rho(x,t)\le4\},
		\qquad
		\Omega_2(t):=\{x\in\mathbb{R}^3:\rho(x,t)>4\},
		\qquad
		\sigma:=\frac32\alpha-1.
		\]
		By Assumption \ref{ass1}, in particular $\alpha\in(\frac23,1)$, hence $\sigma\in(0,\frac12)$. The lower bound in \eqref{prop6:gamma} gives  $\gamma\ge \sigma+\frac56$, and therefore $\frac65(\gamma-\sigma)\ge1$. Splitting the pressure term according to $\Omega_1(t)\cup\Omega_2(t)$, we obtain
		\[
		\left|\int_{\mathbb{R}^3}\nabla\rho^\gamma\cdot |v|^qv\,dx\right|
		\le
		\left|\int_{\Omega_1(t)}\nabla\rho^\gamma\cdot |v|^qv\,dx\right|
		+
		\left|\int_{\Omega_2(t)}\nabla\rho^\gamma\cdot |v|^qv\,dx\right|.
		\]
		
		On $\Omega_1(t)$, using $\nabla\rho^\gamma=\frac{\gamma}{\sigma}\rho^{\gamma-\sigma}\nabla\rho^\sigma$ and Hölder's inequality,
		\[
		\left|\int_{\Omega_1(t)}\nabla\rho^\gamma\cdot |v|^qv\,dx\right|
		\le
		C\|\nabla\rho^\sigma(t)\|_{L^6}
		\left(
		\int_{\Omega_1(t)}
		\rho^{\frac65(\gamma-\sigma)}|v|^{\frac65(q+1)}\,dx
		\right)^{\frac56}.
		\]
		Since $\rho\le4$ on $\Omega_1(t)$ and $\frac65(\gamma-\sigma)\ge1$, the last integral is bounded by
		\[
		C\left(\int_{\mathbb{R}^3}\rho|v|^{\frac65(q+1)}\,dx\right)^{\frac56}.
		\]
		Now $\frac65(q+1)\in(2,q+2)$ because $q\in(1,2)$, and
		\[
		\int_{\mathbb{R}^3}\rho|v|^{\frac65(q+1)}\,dx
		\le
		\left(\int_{\mathbb{R}^3}\rho|v|^2\,dx\right)^{\frac{4-q}{5q}}
		\left(\int_{\mathbb{R}^3}\rho|v|^{q+2}\,dx\right)^{\frac{6q-4}{5q}}.
		\]
		In view of \eqref{eq24}, this yields
		\[
		\left|\int_{\Omega_1(t)}\nabla\rho^\gamma\cdot |v|^qv\,dx\right|
		\le
		C\|\nabla\rho^\sigma(t)\|_{L^6}
		\left(1+\int_{\mathbb{R}^3}\rho|v|^{q+2}\,dx\right)^{1-\frac{2}{3q}}.
		\]
		Moreover, by \eqref{eq36} and the embedding $H^2(\mathbb{R}^3)\hookrightarrow W^{1,6}(\mathbb{R}^3)$,
		\[
		\|\nabla\rho^\sigma\|_{L^1(0,T;L^6(\mathbb{R}^3))}
		\le
		T^{1/2}\|\nabla\rho^\sigma\|_{L^2(0,T;L^6(\mathbb{R}^3))}
		\le C.
		\]
		
		For the contribution of $\Omega_2(t)$, integrating by parts gives
		\[
		-\int_{\Omega_2(t)}\nabla\rho^\gamma\cdot |v|^qv\,dx
		=
		\int_{\mathbb{R}^3}(\rho^\gamma-4^\gamma)_+\operatorname{div}(|v|^qv)\,dx,
		\]
		whence
		\[
		\left|\int_{\Omega_2(t)}\nabla\rho^\gamma\cdot |v|^qv\,dx\right|
		\le
		C\int_{\Omega_2(t)}\rho^\gamma|v|^q|\nabla v|\,dx.
		\]
		By Young's inequality,
		\[
		\left|\int_{\Omega_2(t)}\nabla\rho^\gamma\cdot |v|^qv\,dx\right|
		\le
		\frac{\delta_q}{4}\int_{\mathbb{R}^3}\rho^\alpha|v|^q|\nabla v|^2\,dx
		+
		C\int_{\Omega_2(t)}\rho^{2\gamma-\alpha}|v|^q\,dx.
		\]
		Set
		\[
		\lambda:=\frac{(2\gamma-\alpha)(q+2)-q}{2}.
		\]
		From the lower bound in \eqref{eq64},
		\[
		2\gamma-\alpha\ge2\alpha-\frac13>1,
		\]
		so that $\lambda>0$. Hence Hölder's inequality gives
		\[
		\int_{\Omega_2(t)}\rho^{2\gamma-\alpha}|v|^q\,dx
		\le
		\left(\int_{\mathbb{R}^3}\rho|v|^{q+2}\,dx\right)^{\frac{q}{q+2}}
		\left(\int_{\Omega_2(t)}\rho^\lambda\,dx\right)^{\frac{2}{q+2}}.
		\]
		
		It remains to control the last factor. Let
		\[
		m:=\alpha-\frac12,
		\qquad
		p_0:=\frac{6m}{\sigma}=\frac{12\alpha-6}{3\alpha-2}>6,
		\]
		and choose $\eta\in C^\infty([0,\infty))$ such that
		\[
		0\le\eta\le1,\qquad
		\eta\equiv0\ \text{on }[0,2],\qquad
		\eta\equiv1\ \text{on }[4,\infty).
		\]
		Write
		\[
		g:=\eta(\rho)\rho^\sigma,
		\qquad
		\widetilde\Omega_2(t):=\{x\in\mathbb{R}^3:\rho(x,t)>2\}.
		\]
		Since $\Pi$ is increasing on $[1,\infty)$, \eqref{eq20} implies
		\[
		|\widetilde\Omega_2(t)|
		\le
		\Pi(2)^{-1}\int_{\widetilde\Omega_2(t)}\Pi(\rho)\,dx
		\le
		\Pi(2)^{-1}E_0.
		\]
		On $\widetilde\Omega_2(t)$ one has $\rho^m\le C(\rho^m-1)$, and therefore
		\[
		|g|^{p_0}\le C|\rho^m-1|^6\mathbf 1_{\widetilde\Omega_2(t)}.
		\]
		Using \eqref{eq31} and the embedding $H^1(\mathbb{R}^3)\hookrightarrow L^6(\mathbb{R}^3)$, we infer that
		\[
		\|g\|_{L^\infty(0,T;L^{p_0}(\mathbb{R}^3))}\le C.
		\]
		The uniform bound on $|\widetilde\Omega_2(t)|$ and the fact that $p_0>6$ then imply
		\[
		\|g\|_{L^2(0,T;L^6(\mathbb{R}^3))}\le C.
		\]
		Furthermore,
		\[
		\nabla g=\eta'(\rho)\rho^\sigma\nabla\rho+\eta(\rho)\nabla\rho^\sigma,
		\]
		and since $\eta'$ is supported in $[2,4]$, one has $|\nabla g|\le C|\nabla\rho^\sigma|$. Hence \eqref{eq36} gives
		\[
		\|g\|_{L^2(0,T;W^{1,6}(\mathbb{R}^3))}\le C.
		\]
		
		Choose $\theta\in(0,1)$ and $\zeta>1$ so that
		\[
		\frac1\zeta=\frac{\theta}{p_0}-\frac{1-\theta}{6},
		\qquad
		\frac{\zeta}{q+2}=\frac1{1-\theta}.
		\]
		Then the Gagliardo--Nirenberg inequality yields
		\[
		\|g\|_{L^\zeta(\mathbb{R}^3)}
		\le
		C\|g\|_{L^{p_0}(\mathbb{R}^3)}^\theta
		\|g\|_{W^{1,6}(\mathbb{R}^3)}^{1-\theta},
		\]
		and therefore, by Hölder's inequality in time,
		\[
		\|g\|_{L^{\frac{2\zeta}{q+2}}(0,T;L^\zeta(\mathbb{R}^3))}\le C.
		\]
		A direct computation gives
		\[
		\zeta=\frac{(5\alpha-3)q+22\alpha-12}{3\alpha-2},
		\qquad
		\sigma\zeta=\frac{(5\alpha-3)q+22\alpha-12}{2}.
		\]
		Since $g=\rho^\sigma$ on $\Omega_2(t)$, it follows that
		\[
		\int_0^T
		\left(
		\int_{\Omega_2(t)}\rho^{\frac{(5\alpha-3)q+22\alpha-12}{2}}\,dx
		\right)^{\frac{2}{q+2}}
		dt
		\le C.
		\]
		The upper bound in \eqref{eq64} is equivalent to
		\[
		\lambda<\frac{(5\alpha-3)q+22\alpha-12}{2}.
		\]
		Hence, using again the uniform bound on $|\Omega_2(t)|$ from \eqref{eq34},
		\[
		\left(
		\int_{\Omega_2(t)}\rho^\lambda\,dx
		\right)^{\frac{2}{q+2}}
		\le
		C
		\left(
		\int_{\Omega_2(t)}\rho^{\frac{(5\alpha-3)q+22\alpha-12}{2}}\,dx
		\right)^{
			\frac{2\lambda}{(q+2)\bigl((5\alpha-3)q+22\alpha-12\bigr)/2}
		}.
		\]
		Since
		\[
		0<
		\frac{\lambda}{\frac{(5\alpha-3)q+22\alpha-12}{2}}
		<1,
		\]
		the elementary inequality $a^\vartheta\le1+a$ for $a\ge0$ and $\vartheta\in(0,1)$ gives
		\[
		\int_0^T
		\left(
		\int_{\Omega_2(t)}\rho^\lambda\,dx
		\right)^{\frac{2}{q+2}}
		dt
		\le C.
		\]
		
		Collecting the preceding bounds, we arrive at
		\[
		\frac{d}{dt}\int_{\mathbb{R}^3}\rho|v|^{q+2}\,dx
		\le
		C\|\nabla\rho^\sigma(t)\|_{L^6}
		\left(1+\int_{\mathbb{R}^3}\rho|v|^{q+2}\,dx\right)^{1-\frac{2}{3q}}
		+
		C\left(\int_{\Omega_2(t)}\rho^\lambda\,dx\right)^{\frac{2}{q+2}}
		\left(\int_{\mathbb{R}^3}\rho|v|^{q+2}\,dx\right)^{\frac{q}{q+2}}.
		\]
		Let $s:=\frac{2}{3q}\in(\frac13,\frac23)$. Multiplying by
		\[
		s\left(1+\int_{\mathbb{R}^3}\rho|v|^{q+2}\,dx\right)^{s-1},
		\]
		we obtain
		\[
		\frac{d}{dt}
		\left(
		1+\int_{\mathbb{R}^3}\rho|v|^{q+2}\,dx
		\right)^s
		\le
		C\|\nabla\rho^\sigma(t)\|_{L^6}
		+
		C\left(\int_{\Omega_2(t)}\rho^\lambda\,dx\right)^{\frac{2}{q+2}}
		\left(
		1+\int_{\mathbb{R}^3}\rho|v|^{q+2}\,dx
		\right)^{s-1+\frac{q}{q+2}}.
		\]
		Now
		\[
		s-1+\frac{q}{q+2}
		=
		-\frac{4(q-1)}{3q(q+2)}<0,
		\]
		so the last factor is bounded by $1$. Integrating over $(0,t)$ and using the preceding $L^1(0,T)$ bounds, we conclude that
		\[
		\sup_{0\le t\le T}
		\left(
		1+\int_{\mathbb{R}^3}\rho|v|^{q+2}(x,t)\,dx
		\right)^s
		\le
		C\left(
		1+
		\left\|\rho_0^{\frac{1}{q+2}}v_0\right\|_{L^{q+2}(\mathbb{R}^3)}^{q+2}
		\right).
		\]
		Since $s>0$, this implies \eqref{eq65}. The proof is complete.
	\end{proof}
	Combining Propositions \ref{prop5} and \ref{prop6}, we obtain the following uniform weighted $L^{q+2}$-estimate for the effective velocity $v$ in the whole admissible range $1<q<q_{3,\ast}$ in dimension three.
	\begin{corollary}\label{cor2}
		Assume that $N=3$, and that \eqref{eq4}, \eqref{eq5}, \eqref{eq6}, \eqref{eq7}, and Assumption \ref{ass1} hold.
		Let $(\rho,u)$ be a sufficiently smooth strictly positive local strong solution to \eqref{eq1} on $\mathbb{R}^3\times[0,T^\ast)$, and let $v$ be defined by \eqref{eq12}$.$
		
		Note that
		\[
		q_{3,\ast}>2,
		\qquad
		1\le \gamma<5\alpha-\frac73.
		\]
		Choose $q\in(1,q_{3,\ast})$ according to the following rule:
		\begin{equation}\label{cor2:q-choice}
			\left\{
			\begin{aligned}
				&2\le q<q_{3,\ast},
				\qquad
				\gamma<3\alpha-1+\frac{6\alpha-4}{q+2},
				&&\text{if }1\le \gamma<\frac92\alpha-2,\\[1mm]
				&1<q<2,
				\qquad
				\gamma<3\alpha-1+\frac{6\alpha-4}{q+2},
				&&\text{if }\frac92\alpha-2\le \gamma<5\alpha-\frac73.
			\end{aligned}
			\right.
		\end{equation}
		Such a choice is possible. Indeed, the function
		\[
		f(q):=3\alpha-1+\frac{6\alpha-4}{q+2}
		\]
		is continuous and strictly decreasing on $(1,q_{3,\ast})$, with
		\[
		f(2)=\frac92\alpha-2,
		\qquad
		\lim_{q\to1+}f(q)=5\alpha-\frac73.
		\]
		
		Then, for every $T\in(0,T^\ast)$, there exists a positive constant
		\[
		C=C\!\left(
		\alpha,\gamma,\nu,\varepsilon,E_0,T^\ast,q,
		\left\|\rho_0^{\frac{1}{q+2}}v_0\right\|_{L^{q+2}(\mathbb{R}^3)}
		\right)>0
		\]
		such that
		\[
		\sup_{0\le t\le T}\int_{\mathbb{R}^3}\rho|v|^{q+2}(x,t)\,dx\le C.
		\]
	\end{corollary}
	
	\begin{proof}
		The conclusion follows directly from Proposition \ref{prop5} in the first case of \eqref{cor2:q-choice}, and from Proposition \ref{prop6} in the second one.
	\end{proof}

	We now turn to the two-dimensional case. Since $q_{2,\ast}>2$, Proposition \ref{prop5} applies with $q=2$; interpolating the resulting $L^4$-bound with the basic BD entropy estimate yields the following family of weighted integrability estimates for the effective velocity, which will serve as the two-dimensional input for the upper-bound argument below.
	\begin{lemma}\label{lem3}
		Assume that $N=2$, and that \eqref{eq4}, \eqref{eq5}, \eqref{eq6}, \eqref{eq7}, and Assumption \ref{ass1} hold.
		Let $(\rho,u)$ be a sufficiently smooth strictly positive local strong solution to \eqref{eq1} on $\mathbb{R}^2\times[0,T^\ast)$, and let $v$ be defined by \eqref{eq12}$.$
		
		Then, for every $h\in\left(0,\frac12\right]$ and every $T\in(0,T^\ast)$, there exists a positive constant
		\[
		C_h=C_h\!\left(
		\alpha,\gamma,\nu,\varepsilon,E_0,T^\ast,
		\left\|\rho_0^{1/4}v_0\right\|_{L^4}
		\right)>0
		\]
		such that
		\begin{equation}\label{eq97}
			\sup_{0\le t\le T}
			\int_{\mathbb{R}^2}\rho |v|^{\frac{2}{1-h}}(x,t)\,dx
			\le C_h.
		\end{equation}
		Equivalently,
		\begin{equation}\label{eq98}
			\sup_{0\le t\le T}
			\left\|
			\rho^{\frac{1-h}{2}}v
			\right\|_{L^{\frac{2}{1-h}}(\mathbb{R}^2)}
			\le C_h.
		\end{equation}
	\end{lemma}
	
	\begin{proof}
		By the BD entropy estimate \eqref{eq24}, there exists a positive constant $C_2>0$ such that
		\[
		\sup_{0\le t\le T}\int_{\mathbb{R}^2}\rho|v|^2\,dx\le C_2.
		\]
		Since $q_{2,\ast}>2$, Proposition \ref{prop5} applies with $q=2$. Hence there exists a positive constant $C_4>0$ such that
		\[
		\sup_{0\le t\le T}\int_{\mathbb{R}^2}\rho|v|^4\,dx\le C_4.
		\]
		
		Fix $h\in\left(0,\frac12\right]$.
		
		When $0<h<\frac12$, set
		\[
		\vartheta:=\frac{h}{1-h}\in(0,1).
		\]
		Then
		\[
		\frac{2}{1-h}=2+2\vartheta=2(1-\vartheta)+4\vartheta.
		\]
		Therefore, for each fixed $t\in[0,T]$, Hölder's inequality yields
		\begin{align*}
			\int_{\mathbb{R}^2}\rho|v|^{\frac{2}{1-h}}\,dx
			&=
			\int_{\mathbb{R}^2}
			\bigl(\rho|v|^2\bigr)^{1-\vartheta}
			\bigl(\rho|v|^4\bigr)^{\vartheta}\,dx
			\\
			&\le
			\left(\int_{\mathbb{R}^2}\rho|v|^2\,dx\right)^{1-\vartheta}
			\left(\int_{\mathbb{R}^2}\rho|v|^4\,dx\right)^{\vartheta}
			\\
			&\le
			C_2^{1-\vartheta}C_4^\vartheta.
		\end{align*}
		
		When $h=\frac12$, we have $\frac{2}{1-h}=4$, and thus directly from the above $L^4$ bound,
		\[
		\sup_{0\le t\le T}\int_{\mathbb{R}^2}\rho|v|^{\frac{2}{1-h}}\,dx
		=
		\sup_{0\le t\le T}\int_{\mathbb{R}^2}\rho|v|^4\,dx
		\le C_4.
		\]
		
		Combining the two cases, we conclude that for every $h\in\left(0,\frac12\right]$, there exists a positive constant $C_h>0$, with the dependence stated above, such that
		\[
		\sup_{0\le t\le T}\int_{\mathbb{R}^2}\rho |v|^{\frac{2}{1-h}}\,dx
		\le C_h.
		\]
		This proves \eqref{eq97}.
		
		Finally, since
		\[
		\left\|
		\rho^{\frac{1-h}{2}}v
		\right\|_{L^{\frac{2}{1-h}}(\mathbb{R}^2)}^{\frac{2}{1-h}}
		=
		\int_{\mathbb{R}^2}\rho |v|^{\frac{2}{1-h}}\,dx,
		\]
		estimate \eqref{eq98} follows immediately after relabeling the constant.
	\end{proof}
	We are now in a position to establish the uniform upper bound for the density. The weighted integrability estimates for the effective velocity obtained above---namely, Corollary \ref{cor2} in three dimensions and Lemma \ref{lem3} in two dimensions---provide the key input for the following proposition.
	\begin{proposition}\label{prop7}
		Assume that $N\in\{2,3\}$, and that \eqref{eq4}, \eqref{eq5}, \eqref{eq6}, \eqref{eq7}, and Assumption \ref{ass1} hold.
		Let $(\rho,u)$ be a sufficiently smooth strictly positive local strong solution to \eqref{eq1} on $\mathbb{R}^N\times[0,T^\ast)$, and let $v$ be defined by \eqref{eq12}.\ignore{
			Assume moreover that either
			\begin{equation}\label{eq99}
				N=2,
				\qquad
				\alpha>\frac12,
				\qquad
				q_{2,\ast}>2,
			\end{equation}
			or
			\begin{equation}\label{eq100}
				N=3,
				\qquad
				\alpha\ge \frac79,
				\qquad
				q_{3,\ast}>2,
				\qquad
				\gamma<\frac{15\alpha-7}{3}.
		\end{equation}} Then, for every $T\in(0,T^\ast)$, there exists a positive constant
		\[
		C=C\!\left(
		N,\alpha,\gamma,\nu,\varepsilon,E_0,T^\ast,
		\overline{\rho_0},
		\|\rho_0^{1/4}v_0\|_{L^4}
		\right)>0
		\]
		such that
		\begin{equation}\label{eq101}
			\sup_{(x,t)\in\mathbb{R}^N\times[0,T]}\rho(x,t)\le C.
		\end{equation}
	\end{proposition}
	
	\begin{proof}
		Fix $T\in(0,T^\ast)$. Let $\eta\in C^\infty([0,\infty))$ be nondecreasing, with
		\[
		0\le \eta\le 1,\qquad \eta\equiv 0\ \text{on }[0,2],\qquad \eta\equiv 1\ \text{on }[4,\infty),
		\]
		and set
		\[
		f:=\eta(\rho)\rho,
		\qquad
		\widetilde\Omega_2(t):=\{x\in\mathbb{R}^N:\rho(x,t)>2\}.
		\]
		Exactly as in the proof of Proposition \ref{prop5},
		\[
		|\widetilde\Omega_2(t)|\le \Pi(2)^{-1}E_0=:K_0
		\qquad\text{for all }t\in[0,T].
		\]
		Hence
		\[
		\{x\in\mathbb{R}^N:f(x,t)\neq0\}\subset \widetilde\Omega_2(t),
		\qquad
		|\{x\in\mathbb{R}^N:f(x,t)\neq0\}|\le K_0
		\qquad\text{for all }t\in[0,T].
		\]
		
		We next choose $\theta$. If $N=2$, we set
		\[
		\theta:=\frac{2\alpha-1}{8}\in\Bigl(0,\frac12\Bigr).
		\]
		If $N=3$, we choose $q_\sharp\in(1,2]$ as follows. When $\gamma<\frac{9\alpha-4}{2}$, let $q_\sharp=2$. When
		\[
		\frac{9\alpha-4}{2}\le \gamma<\frac{15\alpha-7}{3},
		\]
		the function
		\[
		q\longmapsto 3\alpha-1+\frac{6\alpha-4}{q+2}
		\]
		is continuous and strictly decreasing on $(1,2]$, with limit $(15\alpha-7)/3$ as $q\to1+$; since $q_{3,\ast}>2$, one may therefore choose $q_\sharp\in(1,2)$ such that
		\[
		\gamma<3\alpha-1+\frac{6\alpha-4}{q_\sharp+2}.
		\]
		We then set
		\[
		\theta:=\frac{q_\sharp}{q_\sharp+2}\in\Bigl(\frac13,\frac12\Bigr].
		\]
		In either case,
		\begin{equation}\label{prop7:eq1}
			\sup_{0\le t\le T}\int_{\mathbb{R}^N}\rho|v|^{\frac{2}{1-\theta}}\,dx\le C_\theta .
		\end{equation}
		Indeed, in two dimensions this is Lemma \ref{lem3}. In three dimensions, if $q_\sharp=2$ this is Proposition \ref{prop5}; if $1<q_\sharp<2$, then by Hölder's inequality, \eqref{eq24} at $t=0$, and $\rho_0\le \overline{\rho_0}$,
		\[
		\int_{\mathbb{R}^3}\rho_0|v_0|^{q_\sharp+2}\,dx
		\le
		\left(\int_{\mathbb{R}^3}\rho_0|v_0|^2\,dx\right)^{\frac{2-q_\sharp}{2}}
		\left(\int_{\mathbb{R}^3}\rho_0|v_0|^4\,dx\right)^{\frac{q_\sharp}{2}}
		\le C,
		\]
		and Corollary \ref{cor2} applies.
		
		For $p\ge0$, define
		\[
		\lambda_p:=p+\alpha+1,
		\qquad
		l:=p+2+\theta-\alpha,
		\qquad
		\Phi_p(s):=\int_0^s \eta(y)^{\lambda_p}y^{p+1}\,dy.
		\]
		Then
		\[
		\Phi_p'(s)=\eta(s)^{\lambda_p}s^{p+1},
		\qquad
		(\eta(s)s)^{p+2}\le (p+2)\Phi_p(s)+4^{p+2}\mathbf 1_{\{s>2\}}.
		\]
		Testing the first equation in \eqref{sys of v},
		\[
		\rho_t+\operatorname{div}(\rho v)-c\Delta\rho^\alpha=0,
		\]
		against $\Phi_p'(\rho)$ gives
		\begin{align*}
			\frac{d}{dt}\int_{\mathbb{R}^N}\Phi_p(\rho)\,dx
			+c\alpha\int_{\mathbb{R}^N}\rho^{\alpha-1}\Phi_p''(\rho)|\nabla\rho|^2\,dx
			=
			\int_{\mathbb{R}^N}\rho v\cdot \nabla\!\bigl(\eta(\rho)^{\lambda_p}\rho^{p+1}\bigr)\,dx.
		\end{align*}
		Since $\eta'\ge0$,
		\[
		\Phi_p''(\rho)\ge (p+1)\eta(\rho)^{\lambda_p}\rho^p.
		\]
		Writing the right-hand side as
		\[
		(p+1)\int \eta(\rho)^{\lambda_p}\rho^{p+1}v\cdot\nabla\rho
		+\lambda_p\int \eta(\rho)^{\lambda_p-1}\eta'(\rho)\rho^{p+2}v\cdot\nabla\rho
		=:I_1+I_2,
		\]
		we estimate $I_1$ by Young's inequality:
		\[
		|I_1|
		\le
		\frac{c\alpha(p+1)}{4}\int \eta(\rho)^{\lambda_p}\rho^{p+\alpha-1}|\nabla\rho|^2
		+C(p+1)\int \eta(\rho)^{\lambda_p}\rho^{p+3-\alpha}|v|^2.
		\]
		Since $p+3-\alpha=l+1-\theta$ and
		\[
		\eta(\rho)^{\lambda_p}\rho^l\le f^l+4^l\mathbf 1_{\{2<\rho<4\}},
		\]
		Hölder's inequality together with \eqref{prop7:eq1} yields
		\begin{align*}
			\int \eta(\rho)^{\lambda_p}\rho^{p+3-\alpha}|v|^2
			&=
			\int \bigl(\rho^{\frac{1-\theta}{2}}|v|\bigr)^2\eta(\rho)^{\lambda_p}\rho^l
			\\
			&\le
			\left(\int \rho|v|^{\frac{2}{1-\theta}}\,dx\right)^{1-\theta}
			\left(\int \bigl(\eta(\rho)^{\lambda_p}\rho^l\bigr)^{1/\theta}\,dx\right)^\theta
			\\
			&\le
			C\left(\int f^{l/\theta}\,dx\right)^\theta+C4^{p+2},
		\end{align*}
		where we also used $l\le p+2$ and $|\widetilde\Omega_2(t)|\le K_0$.
		
		For $I_2$, let $m:=\alpha-\frac12>0$. Since $\operatorname{supp}\eta'\subset[2,4]$ and
		\[
		|\nabla\rho|
		=\frac1m\rho^{1-m}|\nabla(\rho^m-1)|
		\le C|\nabla(\rho^m-1)|
		\qquad\text{on }\{2<\rho<4\},
		\]
		Cauchy--Schwarz, \eqref{eq24}, and Proposition \ref{prop3} give
		\begin{align*}
			|I_2|
			&\le
			C(p+1)4^{p+2}\int_{\{2<\rho<4\}}|v||\nabla\rho|\,dx
			\\
			&\le
			C(p+1)4^{p+2}
			\left(\int \rho|v|^2\,dx\right)^{1/2}
			\left(\int |\nabla(\rho^m-1)|^2\,dx\right)^{1/2}
			\\
			&\le C(p+1)4^{p+2}.
		\end{align*}
		Therefore,
		\begin{align}
			&\sup_{0\le t\le T}\int_{\mathbb{R}^N}\Phi_p(\rho(x,t))\,dx
			+c_1(p+1)\int_0^T\int_{\mathbb{R}^N}\eta(\rho)^{\lambda_p}\rho^{p+\alpha-1}|\nabla\rho|^2\,dxdt
			\notag\\
			&\le
			\int_{\mathbb{R}^N}\Phi_p(\rho_0)\,dx
			+C(p+1)\|f\|_{L^l(0,T;L^{l/\theta}(\mathbb{R}^N))}^l
			+C(p+1)4^{p+2},
			\label{prop7:eq2}
		\end{align}
		for some $c_1>0$ depending only on $\alpha,\nu,\varepsilon$.
		
		The initial term is estimated by
		\[
		\Phi_p(s)\le \frac{s^{p+2}}{p+2}\mathbf 1_{\{s>2\}},
		\qquad
		\int_{\mathbb{R}^N}\Phi_p(\rho_0)\,dx\le C\overline{\rho_0}^{\,p+2}.
		\]
		Moreover,
		\[
		\sup_{0\le t\le T}\int_{\mathbb{R}^N}f^{p+2}(x,t)\,dx
		\le
		C(p+2)\sup_{0\le t\le T}\int_{\mathbb{R}^N}\Phi_p(\rho(x,t))\,dx
		+C4^{p+2},
		\]
		and, since $|\{x\in\mathbb{R}^N:f(x,t)\neq0\}|\le K_0$,
		\[
		\int_0^T\int_{\mathbb{R}^N}f^{p+\alpha+1}\,dxdt
		\le
		CT^\ast
		\left(
		1+\sup_{0\le t\le T}\int_{\mathbb{R}^N}f^{p+2}(x,t)\,dx
		\right).
		\]
		
		If
		\[
		g_p:=f^{\frac{p+\alpha+1}{2}},
		\]
		then, using $\operatorname{supp}\eta'\subset\{2<\rho<4\}$, 
		\[
		|\nabla g_p|
		\le
		C(p+2)\eta(\rho)^{\frac{p+\alpha+1}{2}}\rho^{\frac{p+\alpha-1}{2}}|\nabla\rho|
		+C(p+2)4^{\frac{p+\alpha+1}{2}}\mathbf 1_{\{2<\rho<4\}}|\nabla\rho|.
		\]
		Hence,
		\[
		\int_0^T\int_{\mathbb{R}^N}|\nabla g_p|^2\,dxdt
		\le
		C(p+2)^2\int_0^T\int_{\mathbb{R}^N}\eta(\rho)^{\lambda_p}\rho^{p+\alpha-1}|\nabla\rho|^2\,dxdt
		+C(p+2)^24^{p+2},
		\]
		where Proposition \ref{prop3} was used again to control the integral over $\{2<\rho<4\}$. Combining these estimates with \eqref{prop7:eq2}, and enlarging the constants if necessary, we obtain: there exist $C_0\ge1$ and $\Lambda\ge 8+\overline{\rho_0}$, depending only on
		\[
		N,\alpha,\gamma,\nu,\varepsilon,E_0,T^\ast,\overline{\rho_0},
		\|\rho_0^{1/4}v_0\|_{L^4},
		\]
		such that, for every $p\ge0$,
		\begin{equation}\label{prop7:eq3}
			\left\|f^{\frac{p+2}{2}}\right\|_{L^\infty(0,T;L^2(\mathbb{R}^N))}^2
			+
			\left\|f^{\frac{p+\alpha+1}{2}}\right\|_{L^2(0,T;H^1(\mathbb{R}^N))}^2
			\le
			C_0(p+2)^2\|f\|_{L^l(0,T;L^{l/\theta}(\mathbb{R}^N))}^l
			+
			C_0\Lambda^{p+2}.
		\end{equation}
		
		We now derive the reverse Hölder inequality. If $N=2$, let
		\[
		r:=\frac{8}{\theta},
		\qquad
		a:=\theta-\frac2r=\frac{3\theta}{4}>0.
		\]
		Since $\theta=\frac{2\alpha-1}{8}<\alpha$,
		\[
		\alpha(2+a)-1-\theta(1+a)
		=
		2\alpha-1-\theta+a(\alpha-\theta)>0.
		\]
		If $N=3$, let
		\[
		r:=6,
		\qquad
		a:=\theta-\frac13=\frac{2(q_\sharp-1)}{3(q_\sharp+2)}>0.
		\]
		A direct computation gives
		\[
		\alpha(2+a)-1-\theta(1+a)
		=
		\frac{2\bigl(4\alpha q_\sharp^2+13\alpha q_\sharp+10\alpha-4q_\sharp^2-8q_\sharp-6\bigr)}
		{3(q_\sharp+2)^2}.
		\]
		The numerator is concave in $q_\sharp$ and, since $\alpha\ge\frac79$, its values at $q_\sharp=1$ and $q_\sharp=2$ are
		\[
		27\alpha-18>0,
		\qquad
		52\alpha-38>0.
		\]
		Hence
		\begin{equation}\label{prop7:eq4}
			a=\theta-\frac2r>0,
			\qquad
			\alpha(2+a)-1-\theta(1+a)>0.
		\end{equation}
		
		For $p\ge0$, define
		\[
		s:=(1+a)p+\alpha+1+2a,
		\qquad
		\vartheta:=\frac2s,
		\qquad
		\xi:=\frac{2a}{s}.
		\]
		Then
		\[
		\frac{\xi(p+2)}{2}+\frac{\vartheta(p+\alpha+1)}{2}=1,
		\qquad
		\frac{\vartheta s}{2}=1,
		\qquad
		\frac{\xi}{2}+\frac{\vartheta}{r}=\frac{\theta}{s}.
		\]
		Therefore, Hölder's inequality in time and space yields
		\[
		\|f\|_{L^s(0,T;L^{s/\theta}(\mathbb{R}^N))}
		\le
		\left\|f^{\frac{p+2}{2}}\right\|_{L^\infty(0,T;L^2)}^\xi
		\left\|f^{\frac{p+\alpha+1}{2}}\right\|_{L^2(0,T;L^r)}^\vartheta.
		\]
		By the Sobolev embedding $H^1(\mathbb{R}^N)\hookrightarrow L^r(\mathbb{R}^N)$ and \eqref{prop7:eq3},
		\begin{equation}\label{prop7:eq5}
			\|f\|_{L^s(0,T;L^{s/\theta}(\mathbb{R}^N))}
			\le
			C_r^\vartheta
			\left(
			C_0(p+2)^2\|f\|_{L^l(0,T;L^{l/\theta}(\mathbb{R}^N))}^l
			+
			C_0\Lambda^{p+2}
			\right)^{\frac{\xi+\vartheta}{2}} .
		\end{equation}
		Moreover,
		\[
		s-(1+a)l=\alpha(2+a)-1-\theta(1+a)>0,
		\]
		hence
		\[
		\frac{\xi+\vartheta}{2}\,l=\frac{(1+a)l}{s}\le 1.
		\]
		Since
		\[
		\frac{l}{\theta}-(p+2)=\frac{(1-\theta)(p+2)+\theta-\alpha}{\theta}>0,
		\]
		we also have $p+2\le l/\theta$, and therefore $\Lambda^{p+2}\le (\Lambda^{1/\theta})^l$. Define
		\[
		\Upsilon(\lambda):=
		\max\left\{
		\|f\|_{L^\lambda(0,T;L^{\lambda/\theta}(\mathbb{R}^N))},
		\Lambda^{1/\theta}
		\right\},
		\qquad
		\lambda\ge l_0:=2+\theta-\alpha.
		\]
		Then \eqref{prop7:eq5} implies
		\begin{equation}\label{prop7:eq6}
			\Upsilon(s)
			\le
			(2C_r^2C_0)^{\frac{1+a}{s}}
			(p+2)^{\frac{2(1+a)}{s}}
			\Upsilon(l).
		\end{equation}
		
		We now iterate. Set $l_0:=2+\theta-\alpha$, and define
		\[
		p_k:=l_k-2-\theta+\alpha,
		\qquad
		l_{k+1}:=(1+a)p_k+\alpha+1+2a
		\qquad (k\ge0).
		\]
		Then
		\[
		l_{k+1}
		=
		(1+a)l_k+\alpha(2+a)-1-\theta(1+a),
		\]
		so that, by \eqref{prop7:eq4},
		\[
		l_{k+1}\ge (1+a)l_k,
		\qquad
		l_k\ge (1+a)^k l_0,
		\qquad
		l_k\to\infty.
		\]
		Applying \eqref{prop7:eq6} with $p=p_k$ and $l=l_k$ gives
		\[
		\Upsilon(l_{k+1})
		\le
		(2C_r^2C_0)^{\frac{1+a}{l_{k+1}}}
		(p_k+2)^{\frac{2(1+a)}{l_{k+1}}}
		\Upsilon(l_k).
		\]
		Since
		\[
		\sum_{k=0}^\infty \frac1{l_k}<\infty,
		\qquad
		p_k+2=l_k+\alpha-\theta\le 2l_k,
		\]
		the corresponding infinite product converges. It remains to show that $\Upsilon(l_0)$ is finite.
		
		If $N=2$, let again $m=\alpha-\frac12>0$. On $\widetilde\Omega_2(t)$ one has $\rho^m\le C(\rho^m-1)$, hence
		\[
		f\le C(\rho^m-1)^{1/m}\mathbf 1_{\widetilde\Omega_2(t)}.
		\]
		Since $l_0/(\theta m)<\infty$, the embedding $H^1(\mathbb{R}^2)\hookrightarrow L^{l_0/(\theta m)}(\mathbb{R}^2)$ and Proposition \ref{prop3} imply
		\[
		\|f(t)\|_{L^{l_0/\theta}(\mathbb{R}^2)}
		\le
		C\|\rho^m-1\|_{L^{l_0/(\theta m)}(\mathbb{R}^2)}^{1/m}
		\le
		C\|\rho^m-1\|_{H^1(\mathbb{R}^2)}^{1/m}
		\le C
		\]
		uniformly for $t\in[0,T]$. Therefore
		\[
		\|f\|_{L^{l_0}(0,T;L^{l_0/\theta}(\mathbb{R}^2))}
		\le
		(T^\ast)^{1/l_0}\|f\|_{L^\infty(0,T;L^{l_0/\theta}(\mathbb{R}^2))}
		\le C,
		\]
		and thus $\Upsilon(l_0)\le C$.
		
		If $N=3$, define
		\[
		p_\star:=\frac{l_0}{\theta}-2>0,
		\qquad
		l_\star:=p_\star+2+\theta-\alpha,
		\qquad
		F_\star:=f^{\frac{p_\star+\alpha+1}{2}}.
		\]
		Applying \eqref{prop7:eq3} with $p=p_\star$,
		\begin{equation}\label{prop7:eq7}
			\left\|f^{\frac{p_\star+2}{2}}\right\|_{L^\infty(0,T;L^2(\mathbb{R}^3))}^2
			+
			\|F_\star\|_{L^2(0,T;H^1(\mathbb{R}^3))}^2
			\le
			C_0(p_\star+2)^2\|f\|_{L^{l_\star}(0,T;L^{l_\star/\theta}(\mathbb{R}^3))}^{l_\star}
			+
			C_0\Lambda^{p_\star+2}.
		\end{equation}
		Moreover,
		\[
		\|f(t)\|_{L^{l_\star/\theta}(\mathbb{R}^3)}^{l_\star}
		=
		\|F_\star(t)\|_{L^{s_\star}(\mathbb{R}^3)}^{\mu_\star},
		\qquad
		s_\star:=\frac{2l_\star}{\theta(p_\star+\alpha+1)},
		\qquad
		\mu_\star:=\frac{2l_\star}{p_\star+\alpha+1}.
		\]
		A direct computation yields
		\[
		\mu_\star<2
		\iff
		\alpha>\frac{q_\sharp+1}{q_\sharp+2},
		\]
		and the right-hand side is at most $3/4<7/9\le\alpha$, so $0<\mu_\star<2$. Likewise,
		\[
		s_\star\ge2
		\iff
		(1-\alpha)q_\sharp^2+(3-2\alpha)q_\sharp+(4-2\alpha)\ge0,
		\]
		while
		\[
		s_\star\le6
		\iff
		(\alpha+1)q_\sharp^2+q_\sharp+2\alpha-4\ge0.
		\]
		The latter quadratic is strictly increasing for $q_\sharp>0$, and its value at $q_\sharp=1$ is $3\alpha-2>0$; hence
		\[
		2\le s_\star\le6,
		\qquad
		0<\mu_\star<2.
		\]
		Therefore, by $H^1(\mathbb{R}^3)\hookrightarrow L^{s_\star}(\mathbb{R}^3)$,
		\[
		\|f\|_{L^{l_\star}(0,T;L^{l_\star/\theta}(\mathbb{R}^3))}^{l_\star}
		\le
		C\int_0^T\|F_\star(t)\|_{H^1(\mathbb{R}^3)}^{\mu_\star}\,dt
		\le
		\delta\|F_\star\|_{L^2(0,T;H^1(\mathbb{R}^3))}^2+C_\delta
		\]
		for every $\delta>0$. Choosing
		\[
		\delta=\frac{1}{2C_0(p_\star+2)^2}
		\]
		and returning to \eqref{prop7:eq7}, we obtain
		\[
		\left\|f^{\frac{p_\star+2}{2}}\right\|_{L^\infty(0,T;L^2(\mathbb{R}^3))}^2
		+
		\|F_\star\|_{L^2(0,T;H^1(\mathbb{R}^3))}^2
		\le C.
		\]
		Since $p_\star+2=l_0/\theta$, this yields
		\[
		\|f\|_{L^\infty(0,T;L^{l_0/\theta}(\mathbb{R}^3))}\le C,
		\]
		and hence
		\[
		\|f\|_{L^{l_0}(0,T;L^{l_0/\theta}(\mathbb{R}^3))}
		\le
		(T^\ast)^{1/l_0}\|f\|_{L^\infty(0,T;L^{l_0/\theta}(\mathbb{R}^3))}
		\le C.
		\]
		Thus in both dimensions
		\[
		\Upsilon(l_0)\le C,
		\]
		and consequently
		\[
		\sup_{k\ge0}\Upsilon(l_k)\le C.
		\]
		
		Let $Q_T:=\mathbb{R}^N\times(0,T)$. Since $|\{x\in\mathbb{R}^N:f(x,t)\neq0\}|\le K_0$,
		\[
		\|f\|_{L^{l_k}(Q_T)}
		\le
		K_0^{\frac{1-\theta}{l_k}}
		\|f\|_{L^{l_k}(0,T;L^{l_k/\theta}(\mathbb{R}^N))}
		\le C
		\qquad\text{for all }k\ge0.
		\]
		Fix $\lambda>C$. Then Chebyshev's inequality gives
		\[
		\lambda^{l_k}|\{(x,t)\in Q_T:f(x,t)>\lambda\}|
		\le
		\|f\|_{L^{l_k}(Q_T)}^{l_k}
		\le
		C^{\,l_k}.
		\]
		Letting $k\to\infty$ and using $l_k\to\infty$, we infer that
		\[
		|\{(x,t)\in Q_T:f(x,t)>\lambda\}|=0.
		\]
		Hence $f\in L^\infty(Q_T)$, with
		\[
		\|f\|_{L^\infty(Q_T)}\le C.
		\]
		Since $f=\rho$ on $\{\rho\ge4\}$ and $\rho\le4$ on $\{\rho<4\}$, we have
		\[
		\rho\le 4+f \qquad \text{on }Q_T,
		\]
		and therefore
		\[
		\sup_{(x,t)\in\mathbb{R}^N\times(0,T)}\rho(x,t)\le C.
		\]
		Finally, $\rho_0\le \overline{\rho_0}$, so enlarging $C$ if necessary we obtain
		\[
		\sup_{(x,t)\in\mathbb{R}^N\times[0,T]}\rho(x,t)\le C.
		\]
		This proves \eqref{eq101}.
	\end{proof}

	\subsection{Positive lower bound for the density}\label{subsec2}
	In this subsection, building on the upper bound \eqref{eq101} established in Proposition \ref{prop7}, we first recast the density-weighted \(L^{q+2}\) estimate for the effective velocity obtained in the previous subsection in a form that no longer requires the additional three-dimensional restriction on \(\gamma\). We then use this refined estimate to derive a positive lower bound for the density. We continue to use the notation \eqref{eq37}.

	\begin{proposition}\label{prop8}
		Assume that \(N\in\{2,3\}\), and that \eqref{eq4}, \eqref{eq5}, \eqref{eq6}, \eqref{eq7}, and Assumption \ref{ass1} hold.
		Let \((\rho,u)\) be a sufficiently smooth strictly positive local strong solution to \eqref{eq1} on \(\mathbb{R}^N\times[0,T^\ast)\), and let \(v\) be given by \eqref{eq12}.
		
		Assume in addition that the hypotheses of Proposition \ref{prop7} are satisfied, so that \eqref{eq101} holds for every \(T\in(0,T^\ast)\).
		Fix any \(q\ge 2\) with \(q<q_{N,\ast}\), Then, for every \(T\in(0,T^\ast)\), there exists a positive constant
		\[
		C=C\!\left(
		N,\alpha,\gamma,\nu,\varepsilon,E_0,T^\ast,
		\overline{\rho_0},\|\rho_0^{1/4}v_0\|_{L^4},
		q,\left\|\rho_0^{\frac{1}{q+2}}v_0\right\|_{L^{q+2}}
		\right)
		\]
		such that
		\begin{equation}\label{eq164}
			\sup_{0\le t\le T}\int_{\mathbb{R}^N}\rho|v|^{q+2}(x,t)\,dx\le C.
		\end{equation}
	\end{proposition}
	
	\begin{proof}
		Fix \(T\in(0,T^\ast)\). All integrals below are taken over \(\mathbb{R}^N\), and \(C\) denotes a positive constant depending only on the quantities displayed in the statement.
		
		Repeating the argument leading to \eqref{eq48} in the proof of Proposition \ref{prop5}, we obtain
		\begin{equation}\label{eq165}
			\frac{1}{q+2}\frac{d}{dt}\int \rho|v|^{q+2}\,dx
			+\frac{\delta_q}{2}\int \rho^\alpha |v|^q|\nabla v|^2\,dx
			\le C\int \rho^{2\gamma-\alpha}|v|^q\,dx.
		\end{equation}
		Since \(q<q_{N,\ast}\), \eqref{eq37} yields \(\delta_q>0\).
		
		By Proposition \ref{prop7}, there exists \(\overline C_T>0\) such that
		\[
		0<\rho(x,t)\le \overline C_T
		\qquad\text{for all }(x,t)\in\mathbb{R}^N\times[0,T].
		\]
		Moreover, Assumption \ref{ass1} gives \(\gamma\ge 1\) and \(\alpha<1\), hence \(2\gamma-\alpha>1\). Therefore
		\[
		\rho^{2\gamma-\alpha}(x,t)\le \overline C_T^{\,2\gamma-\alpha-1}\rho(x,t)
		\qquad\text{for all }(x,t)\in\mathbb{R}^N\times[0,T],
		\]
		and \eqref{eq165} reduces to
		\begin{equation}\label{eq166}
			\frac{1}{q+2}\frac{d}{dt}\int \rho|v|^{q+2}\,dx
			\le C\int \rho|v|^q\,dx.
		\end{equation}
		
		If \(q=2\), then \eqref{eq24} gives
		\[
		\sup_{0\le t\le T}\int \rho|v|^2\,dx\le C.
		\]
		If \(q>2\), Hölder's inequality together with \eqref{eq24} yields
		\begin{equation}\label{eq167}
			\int \rho|v|^q\,dx
			\le
			\left(\int \rho|v|^2\,dx\right)^{\frac2q}
			\left(\int \rho|v|^{q+2}\,dx\right)^{\frac{q-2}{q}}
			\le
			C\left(\int \rho|v|^{q+2}\,dx\right)^{\frac{q-2}{q}}
			\le
			C\left(1+\int \rho|v|^{q+2}\,dx\right),
		\end{equation}
		since \(\frac{q-2}{q}\in(0,1)\). In the case \(q=2\), the same bound is immediate from \eqref{eq24}. Consequently, \eqref{eq166} implies
		\[
		\frac{d}{dt}\int \rho|v|^{q+2}\,dx
		\le
		C\left(1+\int \rho|v|^{q+2}\,dx\right)
		\qquad\text{for a.e. }t\in(0,T).
		\]
		Gronwall's lemma therefore gives
		\[
		\sup_{0\le t\le T}\int \rho|v|^{q+2}(x,t)\,dx
		\le
		C\left(1+\int \rho_0|v_0|^{q+2}\,dx\right),
		\]
		this proves \eqref{eq164}.
	\end{proof}
	Combining the upper bound \eqref{eq101} from Proposition \ref{prop7} with the refined
	density-weighted $L^{q+2}$-estimate for the effective velocity obtained in Proposition
	\ref{prop8}, we next derive a positive lower bound for the density by estimating its
	reciprocal. Before doing so, we record that the exponent required in the argument below
	can indeed be chosen within the admissible range.
	
	Since Assumption \ref{ass1} implies $\alpha>\frac12$, the condition
	\[
	\alpha<\frac{q}{q+2}
	\]
	is equivalent to
	\[
	q>\frac{2\alpha}{1-\alpha},
	\]
	which is stronger than $q\ge2$. On the other hand, using the definition of
	$q_{N,\ast}$ in \eqref{eq37}, and denoting by $P_N(\beta;\alpha)$ the quadratic
	polynomial introduced above whose positive zero is $\beta_N^+(\alpha)$, we have the
	identity
	\begin{equation}\label{eq:prop9-qgap}
		\begin{aligned}
			&\frac{1-\alpha}{2}
			\bigl(\beta+\sqrt{N}(1-\alpha)(1-\beta)\bigr)^2
			\left(
			q_{N,\ast}-\frac{2\alpha}{1-\alpha}
			\right)  \\
			&\qquad\qquad
			=
			-\alpha\bigl(1-\sqrt{N}(1-\alpha)\bigr)^2 P_N(\beta;\alpha).
		\end{aligned}
	\end{equation}
	By Assumption \ref{ass1}, one has
	\[
	1-\sqrt{N}(1-\alpha)>0,
	\qquad
	P_N(\beta;\alpha)<0
	\qquad
	\text{for } \beta\in[0,\beta_N^+(\alpha)).
	\]
	Therefore \eqref{eq:prop9-qgap} gives
	\[
	q_{N,\ast}>\frac{2\alpha}{1-\alpha}.
	\]
	Consequently, the interval
	\[
	\left(\frac{2\alpha}{1-\alpha},\,q_{N,\ast}\right)
	\]
	is nonempty, and we may choose once and for all an exponent $q$ satisfying
	\begin{equation}\label{eq168}
		\frac{2\alpha}{1-\alpha}<q<q_{N,\ast}.
	\end{equation}
	For such a choice of $q$, both $q\ge2$ and $\alpha<q/(q+2)$ hold automatically.
	With this admissible exponent fixed, we prove the following $L^\infty$-bound for
	$\rho^{-1}$.
	
	\begin{proposition}\label{prop9}
		Assume that $N\in\{2,3\}$, and that \eqref{eq4}, \eqref{eq5}, \eqref{eq6},
		\eqref{eq7}, and Assumption \ref{ass1} hold.
		Let $(\rho,u)$ be a sufficiently smooth strictly positive local strong solution to
		\eqref{eq1} on $\mathbb{R}^N\times[0,T^\ast)$, and let $v$ be defined by
		\eqref{eq12}.
		
		Assume moreover that the hypotheses of Proposition \ref{prop7} hold, so that
		\eqref{eq101} is available for every $T\in(0,T^\ast)$.
		Let $q$ be any exponent satisfying \eqref{eq168}. Then, for every
		$T\in(0,T^\ast)$, there exists a positive constant
		\[
		C=C\!\left(
		N,\alpha,\gamma,\nu,\varepsilon,E_0,T^\ast,
		\overline{\rho_0},\|\rho_0^{1/4}v_0\|_{L^4},
		q,\left\|\rho_0^{\frac{1}{q+2}}v_0\right\|_{L^{q+2}},
		\left\|\rho_0^{-1}\right\|_{L^\infty}
		\right)>0
		\]
		such that
		\begin{equation}\label{eq169}
			\sup_{(x,t)\in\mathbb{R}^N\times[0,T]}\rho(x,t)^{-1}\le C.
		\end{equation}
	\end{proposition}
	
	\begin{proof}
		Fix $T\in(0,T^\ast)$, and write
		\[
		\tau:=\rho^{-1},
		\qquad
		Q_T:=\mathbb{R}^N\times(0,T),
		\qquad
		\Theta_T:=\|\tau\|_{L^\infty(Q_T)}+1.
		\]
		All constants below depend only on the quantities displayed in the statement. From the first equation in \eqref{sys of v},
		\[
		\partial_t\tau
		-c\alpha\operatorname{div} \!\bigl(\tau^{1-\alpha}\nabla\tau\bigr)
		+2c\alpha\tau^{-\alpha}|\nabla\tau|^2
		+v\cdot\nabla\tau
		-\tau\operatorname{div} v
		=0.
		\]
		
		Choose $\eta\in C^\infty([0,\infty))$ nondecreasing such that
		\[
		0\le \eta\le 1,
		\qquad
		\eta\equiv 0 \ \text{on } [0,2],
		\qquad
		\eta\equiv 1 \ \text{on } [4,\infty),
		\]
		and set
		\[
		f:=\eta(\tau)\tau.
		\]
		By \eqref{eq27} and the basic energy inequality,
		\[
		\bigl|\{\tau(\cdot,t)>2\}\bigr|
		=
		\bigl|\{\rho(\cdot,t)<1/2\}\bigr|
		\le C\int_{\mathbb{R}^N}\Pi(\rho(x,t))\,dx
		\le C E_0
		=:K_-
		\qquad\text{for all }t\in[0,T].
		\]
		Hence
		\[
		\{x\in\mathbb{R}^N:f(x,t)\ne0\}\subset\{\tau(\cdot,t)>2\},
		\qquad
		|\{x\in\mathbb{R}^N:f(x,t)\ne0\}|\le K_-,
		\qquad
		\tau\le f+4.
		\]
		We also introduce
		\[
		H(t):=\int_{\{2<\tau<4\}}\bigl(|v|^2+|\nabla\tau|^2\bigr)(x,t)\,dx.
		\]
		Since $\{2<\tau<4\}=\{1/4<\rho<1/2\}$, one has
		\[
		|\nabla\tau|=\rho^{-2}|\nabla\rho|
		\le C\,|\nabla(\rho^{\alpha-\frac12}-1)|
		\qquad\text{on }\{2<\tau<4\},
		\]
		and therefore, by \eqref{eq31} and Proposition \ref{prop8},
		\[
		\int_0^T H(t)\,dt\le C.
		\]
		
		Let
		\[
		\delta:=\frac{2}{q+2},
		\qquad
		a:=\frac{q+2}{q},
		\]
		and fix $p\ge0$. We write
		\[
		\Lambda_p:=p+3-\alpha,
		\qquad
		l:=p+1+\alpha+\delta,
		\qquad
		\Phi_p(s):=\int_0^s \eta(\sigma)^{\Lambda_p}\sigma^{p+1}\,d\sigma.
		\]
		Testing the equation for $\tau$ against $\Phi_p'(\tau)=\eta(\tau)^{\Lambda_p}\tau^{p+1}$ and integrating by parts, we obtain
		\[
		\frac{d}{dt}\int_{\mathbb{R}^N}\Phi_p(\tau)\,dx
		+c\alpha(p+3)\int_{\mathbb{R}^N}\eta(\tau)^{\Lambda_p}\tau^{p+1-\alpha}|\nabla\tau|^2\,dx
		\le I_1+I_2,
		\]
		where
		\[
		I_1:=(p+3)\left|
		\int_{\mathbb{R}^N}\eta(\tau)^{\Lambda_p}\tau^{p+1}v\cdot\nabla\tau\,dx
		\right|,
		\]
		and
		\[
		I_2:=\Lambda_p\left|
		\int_{\mathbb{R}^N}\eta(\tau)^{\Lambda_p-1}\eta'(\tau)\tau^{p+2}v\cdot\nabla\tau\,dx
		\right|.
		\]
		Then, after one application of Young's inequality,
		\[
		I_1\le \frac{c\alpha(p+1)}{4}\int_{\mathbb{R}^N}\eta(\tau)^{\Lambda_p}\tau^{p+1-\alpha}|\nabla\tau|^2\,dx
		+C(p+3)\int_{\mathbb{R}^N}\eta(\tau)^{\Lambda_p}\tau^{p+1+\alpha}|v|^2\,dx,
		\]
		and, since $\operatorname{supp}\eta'\subset[2,4]$,
		\[
		I_2\le C(p+3)4^{p+2}H(t).
		\]
		Now
		\[
		\tau^{p+1+\alpha}|v|^2
		=
		\bigl(\rho|v|^{q+2}\bigr)^{\frac{2}{q+2}}\tau^l,
		\qquad
		\Lambda_p-l=2-2\alpha-\delta>0
		\]
		by \eqref{eq168}. Hence, using $0\le\eta\le1$ and Proposition \ref{prop8},
		\[
		\int_{\mathbb{R}^N}\eta(\tau)^{\Lambda_p}\tau^{p+1+\alpha}|v|^2\,dx
		\le
		\left(\int_{\mathbb{R}^N}\rho|v|^{q+2}\,dx\right)^{\frac{2}{q+2}}
		\left(\int_{\mathbb{R}^N}\eta(\tau)^{a\Lambda_p}\tau^{la}\,dx\right)^{\frac1a}
		\le C\|f(t)\|_{L^{la}}^l.
		\]
		Consequently,
		\begin{equation}\label{eq:prop9-basic}
			\frac{d}{dt}\int_{\mathbb{R}^N}\Phi_p(\tau)\,dx
			+\frac{c\alpha(p+1)}{2}\int_{\mathbb{R}^N}\eta(\tau)^{\Lambda_p}\tau^{p+1-\alpha}|\nabla\tau|^2\,dx
			\le
			C(p+3)\|f(t)\|_{L^{la}}^l
			+
			C(p+3)4^{p+2}H(t).
		\end{equation}
		
		On the other hand,
		\[
		f^{p+2}\le C\Bigl((p+2)\Phi_p(\tau)+4^{p+2}\mathbf 1_{\{\tau>2\}}\Bigr).
		\]
		Integrating \eqref{eq:prop9-basic} over $(0,T)$ and using the preceding bound, the estimate on $\int_0^T H$, and the fact that $\tau_0\in L^\infty$, we find a constant $C_0\ge1$, depending only on $K_-$ and $\|\tau_0\|_{L^\infty}$, such that
		\begin{equation}\label{eq:prop9-caccioppoli}
			\sup_{0\le t\le T}\int_{\mathbb{R}^N} f^{p+2}(x,t)\,dx
			+\int_0^T\!\!\int_{\mathbb{R}^N}\eta(\tau)^{\Lambda_p}\tau^{p+1-\alpha}|\nabla\tau|^2\,dxdt
			\le
			C(p+3)^2\Bigl(\|f\|_{L^l(0,T;L^{la})}^l+C_0^{p+2}\Bigr).
		\end{equation}
		
		Set
		\[
		g_p:=f^{\frac{p+3-\alpha}{2}}.
		\]
		Since $\operatorname{supp}\eta'\subset[2,4]$, the difference between $|\nabla g_p|$ and \(\eta(\tau)^{\Lambda_p/2}\tau^{\frac{p+1-\alpha}{2}}|\nabla\tau|\) is supported in $\{2<\tau<4\}$ and is controlled by $C(p+3)H(t)^{1/2}$. It follows from \eqref{eq:prop9-caccioppoli} that
		\[
		\int_0^T\!\!\int_{\mathbb{R}^N}|\nabla g_p|^2\,dxdt
		\le
		C(p+3)^2\Bigl(\|f\|_{L^l(0,T;L^{la})}^l+C_0^{p+2}\Bigr).
		\]
		Moreover,
		\[
		\int_0^T\!\!\int_{\mathbb{R}^N}g_p^2\,dxdt
		=
		\int_0^T\!\!\int_{\mathbb{R}^N}f^{p+3-\alpha}\,dxdt
		\le
		T^\ast\Theta_T^{1-\alpha}
		\sup_{0\le t\le T}\int_{\mathbb{R}^N}f^{p+2}(x,t)\,dx.
		\]
		Hence
		\begin{equation}\label{eq:prop9-H1}
			\left\|f^{\frac{p+2}{2}}\right\|_{L^\infty(0,T;L^2)}^2
			+
			\left\|f^{\frac{p+3-\alpha}{2}}\right\|_{L^2(0,T;H^1)}^2
			\le
			C\Theta_T^{1-\alpha}(p+3)^2
			\Bigl(\|f\|_{L^l(0,T;L^{la})}^l+C_0^{p+2}\Bigr).
		\end{equation}
		
		We now choose $r$ so that
		\[
		H^1(\mathbb{R}^N)\hookrightarrow L^r(\mathbb{R}^N),
		\qquad
		\theta:=\frac{q}{q+2}-\frac{2}{r}>\frac1q.
		\]
		For $N=2$ this is immediate with $r<\infty$ large. For $N=3$, Assumption \ref{ass1} gives $\alpha>\frac23$; since $\alpha<\frac{q}{q+2}$, we have $q>\frac{2\alpha}{1-\alpha}>4$, and then one may choose
		\[
		r\in\left(\frac{2q(q+2)}{q^2-q-2},\,6\right].
		\]
		Now define
		\[
		s:=(p+2)\theta+p+3-\alpha,
		\qquad
		\xi:=\frac{2\theta}{s},
		\qquad
		\zeta:=\frac{2}{s}.
		\]
		Then
		\[
		\frac{\xi}{2}(p+2)+\frac{\zeta}{2}(p+3-\alpha)=1,
		\qquad
		\frac{1}{as}=\frac{\xi}{2}+\frac{\zeta}{r},
		\]
		so that Hölder's inequality and Sobolev embedding yield
		\[
		\|f\|_{L^s(0,T;L^{as})}
		\le
		C_r^\zeta
		\left\|f^{\frac{p+2}{2}}\right\|_{L^\infty(0,T;L^2)}^\xi
		\left\|f^{\frac{p+3-\alpha}{2}}\right\|_{L^2(0,T;H^1)}^\zeta.
		\]
		Combining this with \eqref{eq:prop9-H1}, we get
		\[
		\|f\|_{L^s(0,T;L^{as})}
		\le
		C^{\frac{1+\theta}{s}}
		\Theta_T^{(1-\alpha)\frac{1+\theta}{s}}
		(p+3)^{\frac{2(1+\theta)}{s}}
		\Bigl(\|f\|_{L^l(0,T;L^{al})}^l+C_0^{p+2}\Bigr)^{\frac{1+\theta}{s}}.
		\]
		By \eqref{eq168},
		\[
		1-\alpha-\delta>0,
		\qquad
		2-2\alpha-\delta>0,
		\]
		and therefore
		\[
		s-(1+\theta)l
		=
		2-2\alpha-\delta+\theta(1-\alpha-\delta)
		>0.
		\]
		In particular,
		\[
		l\frac{1+\theta}{s}\le1.
		\]
		Since also $p+2\le C_\ast l$ for some constant $C_\ast=C_\ast(\alpha,q)$, we may rewrite the preceding estimate as
		\begin{equation}\label{eq:prop9-iteration}
			\max\Bigl\{\|f\|_{L^s(0,T;L^{as})},\,C_0^{C_\ast}\Bigr\}
			\le
			C^{\frac{1+\theta}{s}}
			\Theta_T^{(1-\alpha)\frac{1+\theta}{s}}
			(p+3)^{\frac{2(1+\theta)}{s}}
			\max\Bigl\{\|f\|_{L^l(0,T;L^{al})},\,C_0^{C_\ast}\Bigr\}.
		\end{equation}
		
		Set
		\[
		l_0:=q,
		\qquad
		p_k:=l_k-1-\alpha-\delta,
		\qquad
		l_{k+1}:=(p_k+2)\theta+p_k+3-\alpha.
		\]
		Then
		\[
		l_{k+1}=(1+\theta)l_k+A,
		\qquad
		A:=2-2\alpha-\delta+\theta(1-\alpha-\delta)>0.
		\]
		Hence $l_k\to\infty$, more precisely
		\[
		l_{k+1}\ge (1+\theta)^{k+1}q,
		\qquad
		p_k+3=l_k+2-\alpha-\delta\le C(1+\theta)^k.
		\]
		Iterating \eqref{eq:prop9-iteration} with $p=p_k$, and using
		\[
		\sum_{k=0}^\infty \frac{1+\theta}{l_{k+1}}<\infty,
		\qquad
		\sum_{k=0}^\infty \frac{\log(p_k+3)}{l_{k+1}}<\infty,
		\]
		we infer that
		\begin{equation}\label{eq:prop9-iterated}
			\max\Bigl\{\|f\|_{L^{l_k}(0,T;L^{al_k})},\,C_0^{C_\ast}\Bigr\}
			\le
			C\,\Theta_T^\sigma
			\max\Bigl\{\|f\|_{L^q(0,T;L^{q+2})},\,C_0^{C_\ast}\Bigr\}
		\end{equation}
		for all $k\ge0$, where
		\[
		\sigma:=(1-\alpha)\sum_{k=0}^\infty \frac{1+\theta}{l_{k+1}}.
		\]
		Moreover,
		\[
		0<\sigma
		\le
		(1-\alpha)\frac{1+\theta}{q\theta}
		<
		(1-\alpha)\Bigl(1+\frac1q\Bigr)
		<1.
		\]
		Indeed, $\theta>\frac1q$, $q\ge2$, and Assumption \ref{ass1} implies $\alpha>\frac{\sqrt5-1}{2}>\frac12$.
		
		Since \(|\{x\in\mathbb{R}^N:f(x,t)\neq0\}|\le K_-\), \eqref{eq:prop9-iterated} gives
		\[
		\|f\|_{L^{l_k}(Q_T)}
		\le
		C\|f\|_{L^{l_k}(0,T;L^{al_k})}
		\le
		C\,\Theta_T^\sigma
		\max\Bigl\{\|f\|_{L^q(0,T;L^{q+2})},\,C_0^{C_\ast}\Bigr\}.
		\]
		As $l_k\to\infty$ and $|\{(x,t)\in Q_T:f(x,t)\neq0\}|\le K_-T<\infty$, letting $k\to\infty$ yields
		\begin{equation}\label{eq:prop9-Linfty}
			\|f\|_{L^\infty(Q_T)}
			\le
			C\,\Theta_T^\sigma
			\max\Bigl\{\|f\|_{L^q(0,T;L^{q+2})},\,C_0^{C_\ast}\Bigr\}.
		\end{equation}
		
		It remains to control the initial norm on the right-hand side. To this end we return to \eqref{eq:prop9-basic} with $p=q$. Set
		\[
		l_q:=q+1+\alpha+\delta<q+2,
		\qquad
		g_q:=f^{\frac{q+3-\alpha}{2}},
		\]
		and
		\[
		m_q:=\frac{2(q+2)l_q}{q(q+3-\alpha)},
		\qquad
		r_q:=\frac{2l_q}{q+3-\alpha},
		\qquad
		b_q:=\frac{2(q+2)}{q+3-\alpha}.
		\]
		Then
		\[
		\|f\|_{L^{al_q}}^{l_q}=\|g_q\|_{L^{m_q}}^{r_q},
		\qquad
		\|g_q\|_{L^{b_q}}^{b_q}
		=
		\int_{\mathbb{R}^N}f^{q+2}\,dx.
		\]
		Since $r_q<2$, and $m_q<\infty$ for $N=2$ while $m_q<6$ for $N=3$, the Gagliardo--Nirenberg inequality and Young's inequality  give
		\[
		\|f\|_{L^{al_q}}^{l_q}
		\le
		\eta\|\nabla g_q\|_{L^2}^2
		+
		C_\eta
		\left(
		\int_{\mathbb{R}^N}f^{q+2}\,dx
		\right)^{\nu_q},
		\]
		where
		\[
		\nu_q
		=
		\frac{
			q^2+\bigl(3-(N-1)\alpha\bigr)q+4-(N-2)\alpha-2N
		}{
			q^2+\bigl(4-N\alpha\bigr)q+4-N-2N\alpha
		}.
		\]
		A direct computation yields
		\[
		1-\nu_q
		=
		\begin{cases}
			\dfrac{q(1-\alpha)+2-4\alpha}{q^2+(4-2\alpha)q+2-4\alpha}, & N=2,\\[3mm]
			\dfrac{q(1-\alpha)+3-5\alpha}{q^2+(4-3\alpha)q+1-6\alpha}, & N=3,
		\end{cases}
		\]
		so that $0<\nu_q<1$ because \eqref{eq168} implies $q>\dfrac{2\alpha}{1-\alpha}$.
		Using again the transition-layer estimate,
		\[
		\|\nabla g_q(t)\|_{L^2}^2
		\le
		C\int_{\mathbb{R}^N}\eta(\tau)^{q+3-\alpha}\tau^{q+1-\alpha}|\nabla\tau|^2\,dx
		+
		C H(t).
		\]
		Substituting these bounds into \eqref{eq:prop9-basic} with $p=q$ and taking $\eta$ sufficiently small, we arrive at
		\[
		\frac{d}{dt}\int_{\mathbb{R}^N}\Phi_q(\tau)\,dx
		\le
		C\left(
		\int_{\mathbb{R}^N}f^{q+2}\,dx
		\right)^{\nu_q}
		+
		C H(t).
		\]
		Since
		\[
		f^{q+2}\le C\,\Phi_q(\tau)+C\,\mathbf 1_{\{\tau>2\}},
		\qquad
		|\{\tau(\cdot,t)>2\}|\le K_-,
		\]
		we get
		\[
		\frac{d}{dt}\int_{\mathbb{R}^N}\Phi_q(\tau)\,dx
		\le
		C\left(1+\left(\int_{\mathbb{R}^N}\Phi_q(\tau)\,dx\right)^{\nu_q}\right)
		+
		C H(t).
		\]
		The nonlinear Gronwall inequality, together with $\int_0^T H(t)\,dt\le C$, now gives
		\[
		\sup_{0\le t\le T}\int_{\mathbb{R}^N}\Phi_q(\tau(x,t))\,dx\le C,
		\]
		and therefore
		\[
		\sup_{0\le t\le T}\int_{\mathbb{R}^N}f^{q+2}(x,t)\,dx\le C.
		\]
		In particular,
		\[
		\|f\|_{L^q(0,T;L^{q+2})}
		\le
		(T^\ast)^{1/q}\sup_{0\le t\le T}\|f(t)\|_{L^{q+2}}
		\le C.
		\]
		Returning to \eqref{eq:prop9-Linfty}, we conclude that
		\[
		\|f\|_{L^\infty(Q_T)}\le C\,\Theta_T^\sigma,
		\qquad 0<\sigma<1.
		\]
		Since $\tau\le f+4$, it follows that
		\[
		\Theta_T
		\le
		\|f\|_{L^\infty(Q_T)}+5
		\le
		C\,\Theta_T^\sigma+5.
		\]
		Because $\sigma<1$, Young's inequality implies $\Theta_T\le C$. Equivalently,
		\[
		\|\tau\|_{L^\infty(Q_T)}\le C.
		\]
		This is exactly \eqref{eq169}.
	\end{proof}
	\subsection{Higher-order estimates}\label{subsec3}

	As a preparation for the higher-order estimates, we first derive an auxiliary bound on $\nabla^3 \rho$.
	\begin{lemma}
		For $N \in \{2, 3\}$, suppose $0 < \underline{\rho} \leq \rho(x, t) \leq \bar{\rho} < \infty$ and Assumption \ref{ass1} hold. Then there exists a constant $C=C(\alpha, \underline{\rho}, \bar{\rho},N)>0$ such that
		\begin{equation}\label{est: rho H_dot3}
			\|\nabla^3\rho\|_{L^2}^2
			\le
			C\|\nabla\Delta\rho^\alpha\|_{L^2}^2
			+
			C\int_{\mathbb{R}^N}|\nabla\rho|^2|\nabla^2\rho|^2\,dx
			+
			C\int_{\mathbb{R}^N}|\nabla\rho|^6\,dx.
		\end{equation}
		Moreover, for $N=3$, it holds that
		\begin{equation}\label{est: rho H_dot3 3D}
			\|\nabla^3\rho\|_{L^2(\mathbb{R}^3)}^2
			\le
			C\|\nabla\Delta\rho^\alpha\|_{L^2(\mathbb{R}^{3})}^2
			+
			C\int_{\mathbb{R}^3}\rho^{\alpha-1}|\nabla\rho|^2|\nabla^2\rho|^2\,dx.
		\end{equation}
		For $N=2,$ assuming that $\sup_{0 \le s \le T}\|\nabla \rho^\alpha(s)\|_{L^2}$ is bounded, then the estimate simplifies to  \begin{equation}\label{est: rho H_dot3 2D}
			\|\nabla^3\rho\|_{L^2(\mathbb{R}^2)}^2
			\le
			C\|\nabla\Delta\rho^\alpha\|_{L^2(\mathbb{R}^{2})}^2
			.
		\end{equation}
	\end{lemma}
	\begin{proof}
		A direct calculation leads to
		\begin{equation}
			\begin{aligned}
				\nabla\Delta \rho
				=& \nabla \left[ \alpha^{-1}\rho^{1-\alpha}\left( \Delta(\rho^{\alpha})-\alpha(\alpha-1)\rho^{\alpha-2}\left| \nabla \rho \right| ^{2} \right)  \right] \\
				\leq&C\left( \left| \nabla \rho \right| \left| \nabla^{2}\rho \right|  +\left| \nabla\Delta \rho^{\alpha} \right| +\left| \nabla\rho\right| ^{3}\right) .
			\end{aligned}
		\end{equation}
		By Plancherel's theorem, we obtain
		\begin{equation}\label{est: ho pre general}
			\|\nabla^3\rho\|_{L^2}^2\le C\|\nabla\Delta\rho^\alpha\|_{L^2}^2+C\left\lVert \left| \nabla \rho \right| \left| \nabla^{2}\rho \right|  \right\rVert ^{2}_{L^{2}}+C\left\lVert \nabla \rho \right\rVert_{L^{6}}^{6} .
		\end{equation}
		For $N=3$ with $\alpha < 1$, it follows from integration by parts and Young's inequality that
		\begin{equation}\label{est: nabla rho L6 3D with weight}
			\begin{aligned}
				\int_{\mathbb{R}^{3}}\rho^{\alpha-3}\left| \nabla \rho \right|^{6}dx=& \frac{1}{\alpha-2}\sum^{3}_{i=1}\int_{\mathbb{R}^{3}}\partial_{i}(\rho^{\alpha-2})\partial_{i}\rho
				|\nabla \rho|^{4}dx\\
				=& \frac{1}{2-\alpha}\int_{\mathbb{R}^{3}}\rho^{\alpha-2}\left( \Delta\rho|\nabla \rho|^{4}+\sum_{i,j=1}^{3}4\left| \nabla \rho \right| ^{2} \partial_{i}\rho\partial_{j}\rho\partial_{ij}\rho \right)dx \\
				\leq& \frac{1}{2-\alpha}\int_{\mathbb{R}^{3}}\rho^{\alpha-2}|\nabla \rho|^{2}\left| \left( |\nabla \rho| ^{2}I+4\nabla \rho \otimes \nabla \rho\right) :\nabla^{2}\rho \right| dx\\
				\leq& \frac{3\sqrt{ 3 }}{2-\alpha}\int_{\mathbb{R}^{3}}\rho^{\alpha-2}|\nabla \rho|^{4}|\nabla^{2}\rho|dx\\
				\leq& \frac{1}{2} \int_{\mathbb{R}^{3}}\rho^{\alpha-3}|\nabla \rho|^{6}dx+ \frac{27}{2(2-\alpha)^{2}}\int_{\mathbb{R}^{3}}\rho^{\alpha-1}|\nabla \rho|^{2}\left| \nabla^{2}\rho \right| ^{2}dx.
			\end{aligned}
		\end{equation}
		From the uniform boundedness of $\rho$, we have
		\begin{equation}\label{est: nabla rho L6 3D}
			\begin{aligned}
				\left\lVert \nabla \rho \right\rVert _{L^{6}(\mathbb{R}^3)}^6\leq& \overline{\rho}^{3-\alpha}\int_{\mathbb{R}^{3}}\rho^{\alpha-3} \left| \nabla \rho \right|^6 dx\leq\frac{27\,\overline\rho^{3-\alpha}}{(2-\alpha)^2}\int_{\mathbb{R}^3}\rho^{\alpha-1}|\nabla\rho|^2|\nabla^2\rho|^2\,dx.
			\end{aligned}
		\end{equation}
		
		For $N=2$, by \eqref{est: basic}, integration by parts, H\"{o}lder's inequality, Plancherel's theorem, together with the boundedness of $\rho$, we deduce that
		\begin{equation}\label{nabla rho L6 2D}
			\begin{aligned}
				\left\lVert \nabla \rho \right\rVert _{L^{6}}^{6}\leq& C\left\lVert \nabla \rho^{\alpha} \right\rVert _{L^{6}}^{6}\leq C\left\lVert \nabla ^{2}\rho^{\alpha} \right\rVert _{L^{2}}^{4}\left\lVert \nabla \rho^{\alpha} \right\rVert _{L^{2}}^{2}\\
				\leq&C\left\lVert \nabla^{2}\rho^{\alpha} \right\rVert _{L^{2}}^{4}\leq C\left( \int_{\mathbb{R}^{2}} \left| \nabla \rho ^{\alpha}\right|\left| \nabla^{3}\rho^{\alpha} \right| dx \right)^{2} \\
				\leq& C\left\lVert \nabla \rho^{\alpha}\right\rVert _{L^{2}}^{2}\left\lVert \nabla^{3}\rho^{\alpha} \right\rVert _{L^{2}}^{2}\\
				\leq& C\left\lVert \nabla\Delta \rho^{\alpha} \right\rVert _{L^{2}}^{2}.
			\end{aligned}
		\end{equation}
		Using \eqref{nabla rho L6 2D}, Gagliardo–Nirenberg inequality, we have
		\begin{equation}\label{est: ho pre 2D tmp}
			\begin{aligned}
				\left\lVert \left| \nabla \rho \right| \left| \nabla^{2}\rho \right|  \right\rVert _{L^{2}}^{2}\leq&C\left\lVert  \left| \nabla \rho^{\alpha} \right| \left| \nabla^{2}\rho^{\alpha} \right| \right\rVert ^{2}_{L^{2}}+C\left\lVert \nabla \rho^{\alpha} \right\rVert ^{6}_{L^{6}}\\
				\leq&C\left\lVert \nabla \rho^{\alpha} \right\rVert ^{2}_{L^{4}}\left\lVert \nabla^{2}\rho^{\alpha} \right\rVert _{L^{4}}^{2}+C\left\lVert \nabla \rho^{\alpha} \right\rVert ^{6}_{L^{6}}\\
				\leq &C\left\lVert \nabla \rho^{\alpha} \right\rVert _{L^{2}}\left\lVert \nabla^{2}\rho^\alpha \right\rVert _{L^{2}}^{2}\left\lVert \nabla^{3}\rho^{\alpha} \right\rVert _{L^{2}}+C\left\lVert \nabla \rho^{\alpha} \right\rVert ^{6}_{L^{6}}\\
				\leq&C\left\lVert \nabla^{2} \rho^{\alpha} \right\rVert _{L^{2}}^{4}+C\left\lVert \nabla^{3}\rho^{\alpha} \right\rVert _{L^{2}}^{2}+C\left\lVert \nabla \rho^{\alpha} \right\rVert ^{6}_{L^{6}}\\
				\leq&C\left\lVert \nabla \rho^{\alpha}\right\rVert _{L^{2}}^{2}\left\lVert \nabla^{3}\rho^{\alpha} \right\rVert _{L^{2}}^{2}+C\left\lVert \nabla^{3}\rho^{\alpha} \right\rVert ^{2}_{L^{2}}+C\left\lVert \nabla \rho^{\alpha} \right\rVert ^{6}_{L^{6}}\\
				\leq&C\left\lVert \nabla\Delta \rho^{\alpha} \right\rVert _{L^{2}}^{2}.
			\end{aligned}
		\end{equation}
		It follows from \eqref{est: ho pre general}, \eqref{nabla rho L6 2D} and \eqref{est: ho pre 2D tmp} that \eqref{est: rho H_dot3 2D} holds, which completes the proof.
	\end{proof}
	The preceding lemmas, combined with the uniform positive upper and lower bounds on the density, allow us to establish the higher-order estimates for the solution.
	\begin{proposition}\label{prop: ho1}
		Let $N \in \{2, 3\}$, and assume that \eqref{eq4}-\eqref{eq7} along with Assumption \ref{ass1} hold. Let $(\rho,u)$ be a strong solution to the system \eqref{eq1} on $\mathbb{R}^N \times [0, T^*)$, and let $v$ be defined by \eqref{eq12}.
		
		Assume further that the conditions of Proposition \ref{prop7} and Proposition \ref{prop9} are satisfied on the interval $[0, T]$. Fix the exponent $q \ge 2$ chosen in Proposition \ref{prop9}, and denote $r := q + 2.$ Then $r > N$, and there exists a constant$$C = C\Bigl(N, \alpha, \gamma, \nu, \varepsilon, E_0, T^*, q, \overline{\rho_0}, \|\rho_0^{-1}\|_{L^\infty}, \|\rho_0^{1/r}v_0\|_{L^r}, \|\rho_0^{1/4}v_0\|_{L^4}, \|\rho_0-1\|_{H^3}, \|u_0\|_{H^2}\Bigr) $$such that
		\begin{equation}\label{est: ho1.1}
			\begin{aligned}
				&\sup_{0\le t\le T}
				\Bigl(
				\|\nabla^2\rho(t)\|_{L^2}^2
				+\|\nabla v(t)\|_{L^2}^2
				+\|\rho_t(t)\|_{L^2}^2
				\Bigr)\\
				&\qquad
				+\int_0^T
				\Bigl(
				\|\nabla^3\rho(t)\|_{L^2}^2
				+\|\nabla\rho_t(t)\|_{L^2}^2
				+\|v_t(t)\|_{L^2}^2
				+\|\nabla^2 v(t)\|_{L^2}^2
				\Bigr)\,dt
				\le C.
			\end{aligned}
		\end{equation}
		Furthermore, when $N=3$, it holds that
		\begin{equation}\label{est: ho1.2}
			\sup_{0\le t\le T}\|\nabla\rho(t)\|_{L^4(\mathbb{R}^3)}^4\le C.
		\end{equation}
		
	\end{proposition}
	\begin{proof}
		\noindent\textbf{Step1: Some lower-order estimates}

		Combining \eqref{eq41} with the uniform boundedness of $\rho$ yields
		\begin{equation}\label{est: v Lr}
			\sup_{0\le t\le T}\|v(t)\|_{L^r(\mathbb{R}^N)}\le C.
		\end{equation}
		Combining the restriction $\alpha < \frac{q}{q+2}$ in Proposition \ref{prop9} with Assumption \ref{ass1} (where $\alpha > 1/2$ for $N=2$ and $\alpha > 2/3$ for $N=3$), we obtain $r > N$. 
		Direct calculation yields
		\begin{equation}
			\begin{aligned}
				\int_0^T \|\nabla^2\rho\|_{L^2}^2 dt =&\int^{T}_{0}\left\| \frac{1}{\alpha-1}\rho^{2-\alpha}\nabla^2\rho^{\alpha-1} - \frac{\alpha-2}{(\alpha-1)^2}\rho^{3-2\alpha}\nabla\rho^{\alpha-1} \otimes \nabla\rho^{\alpha-1} \right\|_{L^{2}}^{2} dt\\
				\le& C \int_0^T \|\nabla^2(\rho^{\alpha-1})\|_{L^2}^2 dt + C \int_0^T \|\nabla(\rho^{\alpha-1})\|_{L^4}^4 dt.
			\end{aligned}
		\end{equation}
		Here \eqref{est: basic}, \eqref{eq24} and \eqref{est: nabla2 phi weighted} implies that
		\begin{equation}\label{est: ho1 step1 tmp1}
			\begin{aligned}
				\int_0^T \|\nabla^2(\rho^{\alpha-1})\|_{L^2}^2 dt\le &C\int_0^T \int_{\mathbb{R}^N} \rho^\alpha|\nabla^2(\rho^{\alpha-1})|^2 dx dt\\ \le& C \int_0^T \int_{\mathbb{R}^N} \rho^\alpha|\nabla v|^2 dx dt + C \int_0^T \int_{\mathbb{R}^N} \rho^\alpha|\mathbb{D}u|^2 dx dt\leq C. 
			\end{aligned}
		\end{equation}
		Furthermore, for $N=2$, by \eqref{est: basic}, \eqref{est: ho1 step1 tmp1} and Gagliardo–Nirenberg inequality, we obtain
		\begin{equation}
			\int_0^T \|\nabla(\rho^{\alpha-1})\|_{L^4}^4 dt \le C \left( \sup_{0 \le t \le T} \|\nabla(\rho^{\alpha-1})(t)\|_{L^2}^2 \right) \int_0^T \|\nabla^2(\rho^{\alpha-1})(t)\|_{L^2}^2 dt\leq C,
		\end{equation}
		whereas for $N=3$, by \eqref{eq35} and uniform boundness of $\rho,$ we deduce $ \int_0^T \|\nabla(\rho^{\alpha-1})\|_{L^4}^4 dt \le C.$ Thus, we have established that  \begin{equation}\label{est: rho L2t H2_hot}
			\int^{T}_{0}\left\lVert \nabla^{2}\rho(t) \right\rVert ^{2}_{L^{2}}dt\leq C.
		\end{equation}

		\noindent\textbf{Step2: Second-order bounds on the density}

		The continuity equation $\eqref{sys of v}_1$ can be recast into the form
		\begin{equation}
			\rho_{t}-c\alpha \rho^{\alpha-1}\Delta \rho=-\text{div}\left(\rho v\right)+c\alpha(\alpha-1)\rho^{\alpha-2}|\nabla \rho|^{2}.
		\end{equation}
		Multiplying by $\Delta \rho_{t}$, and integrating by parts yields
		\begin{equation}\label{eq: ho1tmp1}
			\begin{aligned}
				&\frac{c\alpha}{2} \frac{d}{dt}\int \rho^{\alpha-1}(\Delta \rho)^2 dx + \int | \nabla \rho_t |^2 dx \\
				=& \frac{c\alpha(\alpha-1)}{2}\int \rho^{\alpha-2}\rho_t(\Delta \rho)^2 dx  -  \int \nabla(\text{div}(\rho v))\cdot \nabla \rho_t dx  + c\alpha(\alpha-1)\int \nabla(\rho^{\alpha-2}|\nabla \rho|^2)\cdot\nabla \rho_t dx \\
				=&-\frac{c\alpha(\alpha-1)}{2}\int \rho^{\alpha-2}\text{div}\left(\rho v\right)(\Delta \rho)^{2}dx+\frac{c^{2}\alpha(\alpha-1)}{2}\int \rho^{\alpha-2}\Delta(\rho^{\alpha})(\Delta \rho)^{2}dx\\
				&-  \int \nabla(\text{div}(\rho v))\cdot \nabla \rho_t dx  + c\alpha(\alpha-1)\int \nabla(\rho^{\alpha-2}|\nabla \rho|^2)\cdot\nabla \rho_t dx\\
				=& K_{1}+K_{2}+K_{3}+K_{4}.
			\end{aligned}
		\end{equation}
		Integration by parts, combined with boundness of $\rho$, Young's inequality and \eqref{est: rho H_dot3}, we have
		\begin{equation}\label{est: ho1K1}
			\begin{aligned}
				\left| K_{1} \right|=&  \left| \frac{c\alpha(\alpha-1)(\alpha-2)}{2}\int \rho^{\alpha-2} (v \cdot \nabla \rho) |\Delta \rho|^{2} dx + c\alpha(\alpha-1)\int \rho^{\alpha-1} v \cdot \nabla\Delta \rho \Delta \rho dx \right|\\
				\le&C\int|v|\,|\nabla\rho|\,|\nabla^2\rho|^2\,dx+C\int|v|\,|\nabla^2\rho|\,|\nabla^3\rho|\,dx \\
				\leq&	\frac{c^2}{32}\|\nabla\Delta\rho^\alpha\|_{L^2}^2+C\int	\Bigl(|v|^2|\nabla^2\rho|^2+|\nabla\rho|^2|\nabla^2\rho|^2+|\nabla\rho|^6\Bigr)\,dx,
			\end{aligned}
		\end{equation}
		and 
		\begin{equation}
			\begin{aligned}
				|K_{2}|=&\left| \frac{c^{2}\alpha^{2}(\alpha-1)(\alpha-2)}{2}\int \rho^{2\alpha-4}|\nabla \rho|^{2}|\Delta \rho|^{2}dx+c^{2}\alpha^{2}(\alpha-1)\int \rho^{2\alpha-3}\Delta \rho \nabla\Delta \rho \cdot \nabla \rho dx \right| \\
				\leq&\frac{c^2}{32}\|\nabla\Delta\rho^\alpha\|_{L^2}^2+C\int\Bigl(|\nabla\rho|^2|\nabla^2\rho|^2+|\nabla\rho|^6\Bigr)\,dx.
			\end{aligned}
		\end{equation}
		Moreover, using Young's inequality and boundness of the density, it directly follows that
		\begin{equation}
			\begin{aligned}
				|K_3| \le& \frac{1}{4} \int |\nabla \rho_t|^2 dx + C \int |\nabla(\text{div}(\rho v))|^2 dx \\
				\leq&\frac{1}{4} \left\lVert \nabla \rho_{t} \right\rVert _{L^{2}}^2+C\int \left| \nabla v \right| ^{2}\left| \nabla \rho \right| ^{2}+\left| v \right|^{2} \left| \nabla^{2}\rho \right|^{2}+\left| \nabla^{2}v \right|^{2} dx
			\end{aligned}
		\end{equation}
		and
		\begin{equation}\label{eq: ho1K4}
			\begin{aligned}
				|K_4| \le& \frac{1}{4} \int |\nabla \rho_t|^2 dx + C\int \left| \nabla(\rho^{\alpha-2}|\nabla \rho|^2) \right|^2 dx \\
				\leq&\frac{1}{4} \left\lVert \nabla \rho_{t} \right\rVert _{L^{2}}^2+C\int \left| \nabla \rho \right| ^{6}+\left| \nabla \rho \right| ^{2}\left| \nabla^{2}\rho \right| ^{2}dx.
			\end{aligned}
		\end{equation}
		Taking the gradient of the continuity equation $\eqref{sys of v}_1$ followed by the $L^2$-norm gives 
		\begin{equation}\label{est: ho1tmp1}
			\frac{c^2}{8}\|\nabla\Delta\rho^\alpha\|_{L^2}^2\le\frac14\|\nabla\rho_t\|_{L^2}^2+\frac14\|\nabla\operatorname{div}(\rho v)\|_{L^2}^2.
		\end{equation}

		Substituting \eqref{est: ho1K1}-\eqref{eq: ho1K4} into \eqref{eq: ho1tmp1} and combining \eqref{est: ho1tmp1} we deduce 
		\begin{equation}\label{est: ho1step2}
			\begin{aligned}
				&\frac{c\alpha}{2}\frac{d}{dt}\int \rho^{\alpha-1}|\Delta\rho|^2\,dx+\frac14\|\nabla\rho_t\|_{L^2}^2+\frac{c^2}{16}\|\nabla\Delta\rho^\alpha\|_{L^2}^2\\
				\le&C_{\rho,v}\|\nabla^2 v\|_{L^2}^2+C\int \Bigl(|v|^2|\nabla^2\rho|^2	+|\nabla\rho|^2|\nabla v|^2
				+|\nabla\rho|^2|\nabla^2\rho|^2
				+|\nabla\rho|^6
				\Bigr)\,dx,\end{aligned}
		\end{equation}
		where $C_{\rho,v}>0$ depends only on the equation parameters and the density bounds.

		\noindent\textbf{Step3: First-order bounds on the effective velocity}

		Denote
		\begin{equation}
			\mathcal Q_+(\nabla v)
			:=
			\nu|\nabla v|^2
			+(\nu-c)\nabla v:(\nabla v)^\top
			+(\alpha-1)(2\nu-c)(\operatorname{div}v)^2.
		\end{equation}
		We have the following equivalence: 
		\begin{equation}\label{est: Q_v down}
			\mathcal Q_+(\nabla v)
			\ge
			(2\nu-c)(N\alpha-N+1)|\nabla v|^2,
		\end{equation}
		and 
		\begin{equation}\label{est: Q_v}
			|\mathcal Q_+(\nabla v)|
			\le
			\bigl(\nu+|\nu-c|+N(1-\alpha)(2\nu-c)\bigr)|\nabla v|^2,
		\end{equation}
		which follows from $|\nabla v:(\nabla v)^\top|\le |\nabla v|^2$ and $(\operatorname{div}v)^2\le N|\nabla v|^2.$ Testing $\eqref{sys of v}_2$ against $v_t$, integrating by parts, and utilizing $\eqref{sys of v}_1,$ we obtain
		\begin{equation}\label{eq: ho1tmp2}
			\begin{aligned}
				& \frac{1}{2}\frac{d}{dt}\int \rho^\alpha\mathcal Q_+(\nabla v)\,dx
				+\int \rho|v_t|^2\,dx\\
				=&-\int \rho u\cdot\nabla v\cdot v_t\,dx
				-\int \nabla\rho^\gamma\cdot v_t\,dx
				+ \frac{1}{2}\int (\rho^\alpha)_t\mathcal Q_+(\nabla v)\,dx.
				\\
				=&-\int \rho u\cdot\nabla v\cdot v_t\,dx
				-\int \nabla\rho^\gamma\cdot v_t\,dx
				\\&+\frac{c\alpha}{2}\int \rho^{\alpha-1}\Delta\rho^\alpha\mathcal Q_+(\nabla v)\,dx-\frac{\alpha}{2}\int \rho^{\alpha-1}\operatorname{div}(\rho v)\mathcal Q_+(\nabla v)\,dx\\
				=&L_{1}+L_{2}+L_{3}+L_{4}.
			\end{aligned}
		\end{equation}
		By $u=v-c\alpha\rho^{\alpha-2}\nabla\rho$ and the boundedness of $\rho$, we have
		\begin{equation}\label{est: ho1L1}
			\begin{aligned}
				\left| L_{1} \right| &= \left|\int \rho u\cdot\nabla v\cdot v_t\,dx\right|
				\\&\le
				\frac14\int \rho|v_t|^2\,dx
				+
				C\int \bigl(|v|^2+|\nabla\rho|^2\bigr)|\nabla v|^2\,dx,
			\end{aligned}
		\end{equation}
		and
		\begin{equation}
			\begin{aligned}
				\left| L_{2} \right| &= \left|\int \nabla\rho^\gamma\cdot v_t\,dx\right|\le\frac14\int \rho|v_t|^2\,dx
				+
				C\int |\nabla\rho|^2\,dx.
			\end{aligned}
		\end{equation}
		Before estimating $L_3,$ $L_4,$ it is necessary to first bound the term $\|\nabla^2 v(t)\|_{L^2}.$ Rewriting $\eqref{sys of v}_2$ in the form of
		\begin{equation}
			\nu\Delta v+a_0\nabla\operatorname{div}v=F,
		\end{equation}
		where $a_0:=(\nu-c)+(\alpha-1)(2\nu-c),$ and 
		\begin{equation}
			\begin{aligned}
				F=&\rho^{-\alpha}\Bigl(
				\rho v_t+\rho u\cdot\nabla v+\nabla\rho^\gamma\\
				&-\nu\nabla\rho^\alpha\cdot\nabla v
				-(\nu-c)\nabla\rho^\alpha\cdot(\nabla v)^\top
				-(\alpha-1)(2\nu-c)\nabla\rho^\alpha\operatorname{div}v
				\Bigr).
			\end{aligned}
		\end{equation}
		Noting that 
		\begin{equation}
			m_0:=\min\{\nu,\nu+a_0\}=\min\{\nu,\alpha(2\nu-c)\}>0,
		\end{equation}
		thus by integration by parts, we obtain
		\begin{equation}
			\begin{aligned}
				\left\lVert \nabla^{2}v \right\rVert ^{2}_{L^{2}}\leq& \frac{1}{m_{0}^{2}}\left\lVert F \right\rVert _{L^{2}}^{2}\\
				\leq& \frac{C}{m_{0}^{2}}\left( \left\lVert \sqrt{ \rho } v_{t}\right\rVert ^{2}_{L^{2}} +\left\lVert \left| v \right| \left| \nabla v \right|  \right\rVert _{L^{2}}^{2}+\left\lVert \left| \nabla \rho \right|  \left| \nabla v \right| \right\rVert^{2}_{L^{2}}+\left\lVert \nabla \rho \right\rVert^{2}_{L^{2}}  \right) ,
			\end{aligned}
		\end{equation}
		where $C=C(\alpha, \gamma, \nu, \varepsilon, N, \underline{\rho}, \bar{\rho}) > 0.$ Choosing $\varepsilon_1 = \frac{m_0^2}{4C},$ it holds that 
		\begin{equation}\label{est: v_H2_dot}
			\varepsilon_1\|\nabla^2 v(t)\|_{L^2}^2
			\le
			\frac14\int_{\mathbb{R}^N}\rho|v_t|^2\,dx
			+\frac14\int_{\mathbb{R}^N}
			\Bigl(
			|v|^2|\nabla v|^2
			+|\nabla\rho|^2|\nabla v|^2
			+|\nabla\rho|^2
			\Bigr)\,dx.
		\end{equation}
		By integration by parts, the boundedness of $\rho$, \eqref{est: Q_v}, and noting that $|\nabla\mathcal Q_+(\nabla v)|\le C|\nabla v||\nabla^2 v|$, an application of Young's inequality yields
		\begin{equation}
			\begin{aligned}
				\left| L_3 \right| 
				=&\left|  \frac{c\alpha}{2}\int \nabla(\rho^{\alpha-1})\cdot\nabla\rho^\alpha\,\mathcal Q_+(\nabla v)\,dx
				+\frac{c\alpha}{2}\int \rho^{\alpha-1}\nabla\rho^\alpha\cdot\nabla\mathcal Q_+(\nabla v)\,dx \right|\\
				\leq&C\int |\nabla\rho|^2|\nabla v|^2\,dx
				+
				C\int |\nabla\rho||\nabla v||\nabla^2 v|\,dx\\
				\leq&\frac{\varepsilon_1}{4}\|\nabla^2 v\|_{L^2}^2
				+C\int |\nabla\rho|^2|\nabla v|^2\,dx,
			\end{aligned}
		\end{equation}
		and
		\begin{equation}\label{est: ho1L4}
			\begin{aligned}
				\left| L_4 \right| =&\left| \frac{\alpha}{2}\int \rho v\cdot\nabla\bigl(\rho^{\alpha-1}\mathcal Q_+(\nabla v)\bigr)\,dx \right| \\
				\leq&C\int |v||\nabla\rho||\nabla v|^2\,dx
				+
				C\int |v||\nabla v||\nabla^2 v|\,dx\\
				\leq&\frac{\varepsilon_1}{4}\|\nabla^2 v\|_{L^2}^2
				+
				C\int 
				\Bigl(|v|^2|\nabla v|^2+|\nabla\rho|^2|\nabla v|^2\Bigr)\,dx.
			\end{aligned}
		\end{equation}
		Inserting \eqref{est: ho1L1}-\eqref{est: ho1L4} into \eqref{eq: ho1tmp2} and combining 
		\eqref{est: v_H2_dot}, it holds that 
		\begin{equation}\label{est: ho1step3}\begin{aligned}
				&\frac12\frac{d}{dt}\int \rho^\alpha\mathcal Q_+(\nabla v)\,dx	+\frac14\int \rho|v_t|^2\,dx
				+\frac{\varepsilon_1}{2}\|\nabla^2 v\|_{L^2}^2\\
				\le&C\int \Bigl(|v|^2|\nabla v|^2
				+|\nabla\rho|^2|\nabla v|^2
				+|\nabla\rho|^2
				\Bigr)\,dx.
			\end{aligned}
		\end{equation}

		\noindent\textbf{Step4: Closing the higher-order energy estimates}
		
		Defining
		\begin{equation}
			\mathcal E(t):=\frac{c\alpha}{2}\int \rho^{\alpha-1}|\Delta\rho|^2\,dx+\frac{2C_{\rho,v}}{\varepsilon_1}\int \rho^\alpha\mathcal Q_+(\nabla v)\,dx.
		\end{equation}
		We have 
		\begin{equation}\label{est: ho1 E}
			c_0\bigl(\|\nabla^2\rho\|_{L^2}^2+\|\nabla v\|_{L^2}^2\bigr)\le\mathcal E(t)\le C_0\bigl(\|\nabla^2\rho\|_{L^2}^2+\|\nabla v\|_{L^2}^2\bigr),
		\end{equation} for some $c_0,C_0>0,$ due to \eqref{est: Q_v down} and \eqref{est: Q_v}. By adding $\frac{4C_{\rho,v}}{\varepsilon_1}$ times \eqref{est: ho1step3} to \eqref{est: ho1step2}, we have
		\begin{equation}\label{est: ho1Step4}
			\begin{aligned}
				&\frac{d}{dt}\mathcal E(t)+\frac14\|\nabla\rho_t\|_{L^2}^2+\frac{c^2}{16}\|\nabla\Delta\rho^\alpha\|_{L^2}^2+\frac{C_{\rho,v}}{\varepsilon_1}\int_{\mathbb{R}^N}\rho|v_t|^2\,dx+C_{\rho,v}\|\nabla^2 v\|_{L^2}^2\\
				\le& C\int_{\mathbb{R}^N}\Bigl(|v|^2|\nabla^2\rho|^2+|v|^2|\nabla v|^2+|\nabla\rho|^2|\nabla v|^2+|\nabla\rho|^2|\nabla^2\rho|^2+|\nabla\rho|^6+|\nabla\rho|^2\Bigr)\,dx.
			\end{aligned}
		\end{equation}
		Now we proceed to estimate the right-hand side of \eqref{est: ho1Step4}.

		\noindent\textbf{Case1: $N=2$.}

		For some $r>2,$ by H\"{o}lder's inequality, Gagliardo–Nirenberg inequality, Young's inequality, \eqref{est: v Lr} and \eqref{est: ho1 E}, we obtain 
		\begin{equation}\label{est: ho1 2D tmp1}
			\begin{aligned}
				C\left\lVert \left| v \right| \left| \nabla^{2}\rho \right|  \right\rVert ^{2}_{L^{2}}\leq&C\left\lVert  v\right\rVert _{L^{r}}^{2}\left\lVert \nabla^{2}\rho \right\rVert ^{2}_{L^{\frac{2r}{r-2}}}\\
				\leq&C\left\lVert v \right\rVert ^{2}_{L^{r}}\left\lVert \nabla^{2}\rho \right\rVert ^{2- \frac{4}{r}}_{L^{2}}\left\lVert \nabla^{3}\rho \right\rVert ^{\frac{4}{r}}_{L^{2}}\\
				\leq& \eta \left\lVert \nabla^{3}\rho \right\rVert ^{2}_{L^{2}}+C_{\eta}\left\lVert \nabla^{2}\rho \right\rVert ^{2}_{L^{2}}\\
				\leq&\eta \left\lVert \nabla\Delta \rho^{\alpha} \right\rVert ^{2}_{L^{2}}+C_{\eta}\mathcal{E}(t).
			\end{aligned}
		\end{equation}
		Similarly, it holds that 
		\begin{equation}\label{est: ho1 2Dtmp2}
			\begin{aligned}
				C\left\lVert \left| v \right| \left| \nabla v \right|  \right\rVert_{L^{2}}^{2}\leq& C\left\lVert v \right\rVert_{L^{r}}^{2}\left\lVert \nabla v \right\rVert ^{2}_{L^{\frac{2r}{r-2}}} \\
				\leq&C\left\lVert v \right\rVert ^{2}_{L^{r}}\left\lVert \nabla v \right\rVert _{L^{2}}^{2- \frac{4}{r}}\left\lVert \nabla^{2}v \right\rVert^{\frac{4}{r}}_{L^{2}} \\
				\leq& \eta \left\lVert \nabla^{2}v \right\rVert _{L^2}^{2}+C_{\eta}\left\lVert \nabla v \right\rVert ^{2}_{L^{2}}\\
				\leq& \eta \left\lVert \nabla^{2}v \right\rVert _{L^2}^{2}+C_{\eta}\mathcal{E}(t).
			\end{aligned}
		\end{equation}
		It follows from \eqref{est: basic} and boundness of $\rho$ that 
		\begin{equation}\label{est: rho_H1_dot}
			\int_{\mathbb{R}^2}|\nabla\rho|^2\,dx\le C.
		\end{equation}
		Next, we deal with $\int|\nabla\rho|^2|\nabla v|^2dx.$ Integrating by parts yields 
		\begin{equation}\label{est: ho1 2d diff}
			\begin{aligned}
				C\int_{\mathbb{R}^2} \left| \nabla \rho \right| ^{2}\left| \nabla v \right| ^{2}dx=&-C\int _{\mathbb{R}^2}v\partial_{j}\left[ (\partial_{j}v)(\partial_{i}\rho) ^{2}\right] dx\\
				=& -C\int _{\mathbb{R}^2} v \partial_{j}v (2\partial_{i}\rho \partial_{i}\partial_{j}\rho)dx-C\int_{\mathbb{R}^2} v\partial_{j}^{2}v (\partial_{i}\rho)^{2}dx\\
				\leq&C\int _{\mathbb{R}^2}\left| v \right| \left| \nabla v \right| \left| \nabla \rho \right| \left| \nabla^{2}\rho \right| dx+C\int_{\mathbb{R}^2} \left| v \right| \left| \nabla^{2}v \right| \left| \nabla \rho \right| ^{2}dx\\
				=&I_{1}+I_{2}.
			\end{aligned}
		\end{equation}
		Applying \eqref{est: v Lr}, H\"{o}lder’s inequality, Gagliardo–Nirenberg inequality and Sobolev embedding, we have
		\begin{equation}\label{est: ho1I2}
			\begin{aligned}
				I_{2}=&C\int_{\mathbb{R}^2} \left| v \right| \left| \nabla^{2}v \right| \left| \nabla \rho \right| ^{2}dx\\
				\leq&\left\lVert v \right\rVert _{L^{r}}\left\lVert \nabla^{2}v \right\rVert _{L^{2}}\left\lVert \nabla \rho^{\alpha} \right\rVert ^{2}_{L^{\frac{4r}{r-2}}}\\
				\leq&C\left\lVert v \right\rVert _{L^{r}}\left\lVert \nabla^{2}v \right\rVert _{L^{2}}\left\lVert \nabla \rho^{\alpha} \right\rVert ^{\frac{r-2}{r}}_{L^{2}}\left\lVert \nabla^{2} \rho^{\alpha} \right\rVert ^{\frac{r+2}{r}}_{L^{2}}\\
				\leq& \eta\left\lVert \nabla^{2}v \right\rVert ^{2}_{L^{2}}+ C_{\eta}\left\lVert \nabla^{2} \rho^{\alpha} \right\rVert _{L^{2}}^{\frac{2(r+2)}{r}}\\
				{ \leq } &\eta \left\lVert \nabla^{2}v \right\rVert ^{2}_{L^{2}}+ \eta\left\lVert \nabla ^{2}\rho^{\alpha} \right\rVert ^{4}_{L^{2}}+C_{\eta},
			\end{aligned}
		\end{equation}
		and 
		\begin{equation}\label{est: hoI1}
			\begin{aligned}
				I_{1}\leq &\left\lVert v \right\rVert _{L^{r}}\left\lVert \nabla v \right\rVert _{L^{\frac{1}{\frac{1}{4}- \frac{1}{2r}}}}\left\lVert \left| \nabla \rho  \right| \left| \nabla^{2}\rho \right| \right\rVert _{L^{\frac{1}{\frac{3}{4}- \frac{1}{2r}}}}\\
				\leq &C\left\lVert \nabla v \right\rVert _{H_{1}}\left\lVert \left| \nabla \rho \right| \left| \nabla^{2}\rho \right|  \right\rVert _{L^{\frac{1}{\frac{3}{4}- \frac{1}{2r}}}}\\
				{ \leq }& \eta\left\lVert \nabla v \right\rVert ^{2}_{H^{1}}+C\left\lVert \left| \nabla \rho \right| \left| \nabla^{2}\rho \right|  \right\rVert _{L^{1}}^{2\left(\frac{1}{2}- \frac{1}{r}\right)}\left\lVert \left| \nabla \rho \right| \left| \nabla^{2}\rho \right| \right\rVert ^{2\left(\frac{1}{2}+ \frac{1}{r}\right)} _{L^{2}}\\
				\leq& \eta\left\lVert \nabla v \right\rVert ^{2}_{H^{1}}+ C\left\lVert \left| \nabla \rho \right|\left| \nabla^{2}\rho \right|   \right\rVert ^{2}_{L^{1}}+\delta \left\lVert \left| \nabla \rho \right| \left| \nabla^{2}\rho \right|  \right\rVert ^{2}_{L^{2}}.
			\end{aligned}
		\end{equation}
		Due to H\"{o}lder's inequality and \eqref{est: rho_H1_dot}, one has
		\begin{equation}\label{est: ho1I1_1}
			C\left\lVert \left| \nabla \rho \right| \left| \nabla^{2}\rho \right|  \right\rVert ^{2}_{L^{1}}\leq C\left\lVert \nabla \rho \right\rVert ^{2}_{L^{2}}\left\lVert \nabla^{2}\rho \right\rVert ^{2}_{L^{2}}\leq C\left\lVert \nabla^{2}\rho \right\rVert ^{2}_{L^{2}}.
		\end{equation}
		A direct calculation yields
		\begin{equation}\label{eq: rho and rho alpha}
			\nabla \rho^{\alpha} =\alpha \rho^{\alpha-1}\nabla \rho ,  \quad\nabla^{2}\rho^{\alpha} = \alpha \rho^{\alpha-1}\nabla^{2}\rho+ \alpha(\alpha-1)\nabla \rho \otimes \nabla \rho,
		\end{equation}
		indicating that
		\begin{equation}
			\left| \nabla \rho \right| \leq C\left| \nabla \rho^{\alpha} \right| ,\left| \nabla^{2}\rho \right| \leq C(\left| \nabla^{2}\rho ^{\alpha}\right| +\left| \nabla \rho ^{\alpha}\right| ^{2}).
		\end{equation}
		Consequently, we have
		\begin{equation}\label{est: ho1tmp4}
			\begin{aligned}
				\delta \left\lVert \left| \nabla \rho \right| \left| \nabla^{2}\rho \right| \right\rVert ^{2}_{L^{2}}&= \delta \int _{\mathbb{R}^2}\left| \nabla \rho \right| ^{2}\left| \nabla^{2}\rho \right| ^{2}dx\\
				&\leq \delta C \int _{\mathbb{R}^2}\left| \nabla \rho^{\alpha} \right| ^{2}\left| \nabla^{2}\rho^{\alpha} \right| ^{2}dx+\delta C\int _{\mathbb{R}^2}\left| \nabla \rho^{\alpha} \right| ^{6}dx.
			\end{aligned}
		\end{equation}
		Using Gagliardo–Nirenberg inequality, boundness of $\rho$ and \eqref{est: basic}, we derive
		\begin{equation}\label{est: ho1tmp2}
			\begin{aligned}
				\int_{\mathbb{R}^2}|\nabla \rho^{\alpha}|^{2}|\nabla^{2}\rho^{\alpha}|^{2}dx&\leq \left\lVert \nabla \rho^{\alpha} \right\rVert ^{2}_{L^{4}}\left\lVert \nabla^{2}\rho^{\alpha} \right\rVert ^{2}_{L^{4}}\\
				&\leq C\left( \left\lVert \nabla \rho^{\alpha} \right\rVert _{L^{2}}\left\lVert \nabla^{2}\rho^{\alpha} \right\rVert _{L^{2}} \right) \left( \left\lVert \nabla^{2} \rho^{\alpha} \right\rVert _{L^{2}}\left\lVert \nabla^{3}\rho^{\alpha} \right\rVert_{L^{2}}  \right) \\
				&\leq C\left\lVert \nabla^{2}\rho^{\alpha} \right\rVert ^{2}_{L^{2}}\left\lVert  \nabla^{3}\rho^{\alpha}\right\rVert _{L^{2}},
			\end{aligned}
		\end{equation}
		and
		\begin{equation}\label{est: ho1tmp3}
			\left\lVert \nabla \rho^{\alpha} \right\rVert _{L^{6}}^{6}\leq C\left\lVert \nabla \rho^{\alpha} \right\rVert ^{2}_{L^{2}}\left\lVert \nabla^{2}\rho^{\alpha} \right\rVert^{4} _{L^{2}}\leq C\left\lVert \nabla^{2}\rho^{\alpha} \right\rVert^{4}_{L^{2}} .
		\end{equation}
		Therefore, substituting \eqref{est: ho1tmp2}, \eqref{est: ho1tmp3} into \eqref{est: ho1tmp4} gives 
		\begin{equation}\label{est: ho1I1_2}
			\begin{aligned} \delta \int _{\mathbb{R}^2}\left| \nabla \rho \right|^2 \left| \nabla^{2}\rho \right|^2 dx \leq & \delta C \int_{\mathbb{R}^2} \left| \nabla^{3}\rho^{\alpha} \right| ^{2}dx + \delta C\left\lVert \nabla^{2}\rho^{\alpha} \right\rVert^{4}_{L^{2}}. \end{aligned}
		\end{equation}
		Inserting \eqref{est: ho1I2}-\eqref{est: ho1I1_2} into \eqref{est: ho1 2d diff}, we obtain 
		\begin{equation}
			\begin{aligned}
				C\int _{\mathbb{R}^2}\left| \nabla \rho \right| ^{2}\left| \nabla v \right| ^{2}dx\le&2 \eta\left\lVert  \nabla^{2}v \right\rVert ^{2}_{L^{2}}+ \delta C\left\lVert \nabla^{3} \rho^{\alpha} \right\rVert ^{2}_{L^{2}}\\
				&+C_{\eta,\delta}(1+\left\lVert \nabla^{2}\rho \right\rVert ^{2}_{L^{2}}+\|\nabla v\|^{2}_{L^2}+\|\nabla^2(\rho^{\alpha})\|_{L^{2}}^4).
			\end{aligned}
		\end{equation}
		Recalling \eqref{eq: rho and rho alpha} and \eqref{est: ho1 E}, along with Gagliardo–Nirenberg inequality, one gets 
		\begin{equation}
			\left\lVert \nabla^{2}(\rho^{\alpha}) \right\rVert _{L^{2}}^{2}\leq C\left\lVert \nabla^{2}\rho \right\rVert ^{2}_{L^{2}}+C\left\lVert \nabla \rho \right\rVert^{4}_{L^{4}}\leq C\left\lVert \nabla^{2}\rho \right\rVert  ^{2}_{L^{2}}+C\left\lVert \nabla \rho \right\rVert ^{2}_{L^{2}}\left\lVert \nabla^{2}\rho \right\rVert ^{2}_{L^{2}}\leq C\mathcal{E}(t),
		\end{equation}
		which implies 
		\begin{equation}\label{est: ho1 2Dtmp3}
			\begin{aligned}
				C\int _{\mathbb{R}^2}\left| \nabla \rho \right| ^{2}\left| \nabla v \right| ^{2}dx\le&2 \eta\left\lVert  \nabla^{2}v \right\rVert ^{2}_{L^{2}}+\delta C\left\lVert \nabla^{3} \rho^{\alpha} \right\rVert ^{2}_{L^{2}}\\
				&+C_{\eta,\delta}(1+\|\nabla^{2}\rho\|^{2}_{L^2})(\mathcal{E}(t)+1).
			\end{aligned}
		\end{equation}
		Repeating the similar procedure used for \eqref{est: ho1tmp4}-\eqref{est: ho1tmp3}, we have 
		\begin{equation}\label{est: ho1 2Dtmp4}
			C\int _{\mathbb{R}^2}(|\nabla\rho|^2|\nabla^2\rho|^2+|\nabla\rho|^6)dx\le \eta\|\nabla^{3}\rho^{\alpha}\|_{L^2}^2+C\|\nabla^{2}\rho^{\alpha}\|^{4}_{L^2}\le \eta\|\nabla^{3}\rho^{\alpha}\|_{L^2}^2+C\|\nabla^{2}\rho\|^{2}_{L^2}\mathcal{E}(t).
		\end{equation}
		Subsituting \eqref{est: ho1 2D tmp1}, \eqref{est: ho1 2Dtmp2}, \eqref{est: ho1 2Dtmp3}, \eqref{est: ho1 2Dtmp4} into \eqref{est: ho1Step4} and choosing $\eta,\delta\ll 1$
		,it follows from Gronwall’s inequality and \eqref{est: rho L2t H2_hot} that 
		\begin{equation}\label{est: ho1 2D final}
			\begin{aligned}
				& \sup_{0\le t\le T} \Bigl( \|\nabla^2\rho(t)\|_{L^2}^2 + \|\nabla v(t)\|_{L^2}^2 \Bigr) \\
				& + \int_0^T \Bigl( \|\nabla\rho_t\|_{L^2}^2 + \|\nabla\Delta(\rho^\alpha)\|_{L^2}^2 + \|v_t\|_{L^2}^2 + \|\nabla^2 v\|_{L^2}^2 \Bigr) \,dt \le C.
			\end{aligned}
		\end{equation}

		\noindent\textbf{Case2: $N=3$.}

		Before proceeding the proof to  it is necessary to obtain an auxiliary $L^4$-estimate for $\nabla \rho$. Taking the inner product of the gradient of $\eqref{sys of v}_1$ with $c^{-1}|\nabla\rho|^2\nabla\rho$, and performing integration by parts, we obtain
		\begin{equation}\label{est: ho1 3D nabla rho L4 tmp1}
			\frac{1}{4c}\frac{d}{dt}\left\lVert \nabla \rho \right\rVert ^{4}_{L^{4}}-\int_{\mathbb{R}^{3}}\nabla\Delta \rho^{\alpha}\left( \left| \nabla \rho \right|^{2}\nabla \rho \right) dx=c^{-1}\int_{\mathbb{R}^{3}}\text{div}\left(\rho v\right)\text{div}\left(\left| \nabla \rho \right|^{2}\nabla \rho \right)dx.
		\end{equation}
		Define 
		\begin{equation}
			A(t):=\int_{\mathbb{R}^3}\rho^{\alpha-1}|\nabla\rho|^2|\nabla^2\rho|^2\,dx.
		\end{equation}
		By integration by parts and \eqref{est: nabla rho L6 3D with weight}
		\begin{equation}\label{est: ho1 3D nabla rho L4 tmp2}
			\begin{aligned} 
				&-\int_{\mathbb{R}^{3}}\nabla\Delta \rho^{\alpha}\left( \left| \nabla \rho \right|^{2}\nabla \rho \right) dx \\
				=&\alpha \int_{\mathbb{R}^3} \rho^{\alpha-1} \Big[ |\nabla\rho|^2 |\nabla^2\rho|^2 + 2 |\nabla^2\rho \cdot \nabla\rho|^2 - \frac{3(1-\alpha)}{\rho}|\nabla\rho|^2 (\nabla\rho \cdot \nabla^2\rho \cdot \nabla\rho) \Big] dx\\
				\geq&\alpha \int_{\mathbb{R}^3} \rho^{\alpha-1} \Big[ |\nabla\rho|^2 |\nabla^2\rho|^2  - \frac{9(1-\alpha)^2}{8\rho^2}|\nabla\rho|^6 \Big] dx\\
				=& \alpha A(t) - \frac{9\alpha(1-\alpha)^2}{8} \int_{\mathbb{R}^3} \rho^{\alpha-3}|\nabla\rho|^6 dx\\
				\geq&\alpha\left(
				1-
				\frac{243(1-\alpha)^2}{8(2-\alpha)^2}
				\right)
				\int_{\mathbb{R}^3}\rho^{\alpha-1}|\nabla\rho|^2|\nabla^2\rho|^2\,dx
				=
				\varepsilon_2 A(t),
			\end{aligned}
		\end{equation}
		where $\varepsilon_2
		:=
		\alpha\left(
		1-
		\frac{243(1-\alpha)^2}{8(2-\alpha)^2}
		\right)>0,$ which necessitates the condition $1>\alpha> \frac{9\sqrt{3}-4\sqrt{2}}{9\sqrt{3}-2\sqrt{2}}$ introduced in Assumption \ref{ass1}. Substituting \eqref{est: ho1 3D nabla rho L4 tmp2} into \eqref{est: ho1 3D nabla rho L4 tmp1}, combining Young's inequality and boundness of $\rho,$ one gets
		\begin{equation}
			\begin{aligned}
				\frac1{4c}\frac{d}{dt}\|\nabla\rho\|_{L^4}^4+\varepsilon_2 A(t)\le&
				C\int_{\mathbb{R}^3}|\nabla(\rho v)|\,|\nabla^2\rho|\,|\nabla\rho|^2\,dx\\
				\le&\frac{\varepsilon_2}{2}\int_{\mathbb{R}^3}\rho^{\alpha-1}|\nabla\rho|^2|\nabla^2\rho|^2\,dx\\
				&+C\int_{\mathbb{R}^3}\rho^{1-\alpha}|\nabla\rho|^2|\nabla(\rho v)|^2\,dx\\
				\le&\frac{\varepsilon_2}{2}A(t)+C\int_{\mathbb{R}^3}|\nabla\rho|^2|\nabla(\rho v)|^2\,dx,
			\end{aligned}
		\end{equation}
		which implies
		\begin{equation}\label{eq135}
			\frac1{4c}\frac{d}{dt}\|\nabla\rho\|_{L^4(\mathbb{R}^3)}^4
			+
			\frac{\varepsilon_2}{2}
			\int_{\mathbb{R}^3}\rho^{\alpha-1}|\nabla\rho|^2|\nabla^2\rho|^2\,dx
			\le
			C\int_{\mathbb{R}^3}|\nabla\rho|^2|\nabla(\rho v)|^2\,dx.
		\end{equation}
		We now return to handle right hand side of \eqref{est: ho1Step4}.
		By Hölder's inequality,\eqref{est: v Lr}, and Sobolev embedding, together with $r > 3$, we deduce that 
		\begin{equation}\label{est: ho1 3D tpm1}
			\begin{aligned}
				\int_{\mathbb{R}^3}|v|^2|\nabla v|^2\,dx
				&\le
				\|v\|_{L^r}^2\|\nabla v\|_{L^{\frac{2r}{r-2}}}^2
				\notag\\
				&\le
				C\|\nabla v\|_{L^2}^{2-\frac6r}\|\nabla^2 v\|_{L^2}^{\frac6r}
				\le
				\eta\|\nabla^2 v\|_{L^2}^2
				+
				C_\eta\|\nabla v\|_{L^2}^2.
			\end{aligned}
		\end{equation}
		Likewise, we deduce that 
		\begin{equation}
			\begin{aligned}
				\int_{\mathbb{R}^3}|v|^2|\nabla^2\rho|^2\,dx
				&\le
				\|v\|_{L^r}^2\|\nabla^2\rho\|_{L^{\frac{2r}{r-2}}}^2
				\notag\\
				&\le
				C\|\nabla^2\rho\|_{L^2}^{2-\frac6r}\|\nabla^3\rho\|_{L^2}^{\frac6r}
				\notag\\
				&\le
				\eta\|\nabla^3\rho\|_{L^2}^2
				+
				C_\eta\|\nabla^2\rho\|_{L^2}^2
				\notag\\
				&\le
				\eta\|\nabla\Delta\rho^\alpha\|_{L^2}^2
				+\eta A(t)
				+
				C_\eta\mathcal E(t),
			\end{aligned}
		\end{equation}
		where we use \eqref{est: rho H_dot3 3D} and \eqref{est: ho1 E} in the last step. Integrating by parts once more, we have
		\begin{equation}
			C\int_{\mathbb{R}^3}|\nabla\rho|^2|\nabla v|^2\,dx
			\le
			C\int_{\mathbb{R}^3}|v||\nabla v||\nabla\rho||\nabla^2\rho|\,dx
			+
			C\int_{\mathbb{R}^3}|v||\nabla^2v||\nabla\rho|^2\,dx
			=:J_{31}+J_{32}.
		\end{equation}
		Applying Hölder's and Young's inequalities, along with \eqref{est: v Lr} and the Sobolev embedding, yields
		\begin{equation}
			\begin{aligned}
				J_{31}
				&\le
				C\|v\|_{L^r}\,\||\nabla\rho||\nabla^2\rho|\|_{L^2}\,\|\nabla v\|_{L^{\frac{2r}{r-2}}}\\
				&\le
				C A(t)^{1/2}\|\nabla v\|_{L^2}^{1-\frac3r}\|\nabla^2 v\|_{L^2}^{\frac3r}\\
				&\le
				\eta A(t)+\eta\|\nabla^2 v\|_{L^2}^2+C_\eta\|\nabla v\|_{L^2}^2.
			\end{aligned}
		\end{equation}
		Noting that 
		\begin{equation}\label{est: nabla rho L12}
			\begin{aligned}
				\left\lVert \nabla(\rho^{\alpha}) \right\rVert _{L^{12}}^{12}&= -\int \rho^{\alpha}\text{div}\left(\left| \nabla(\rho^{\alpha}) \right| ^{10}\nabla \rho^{\alpha}\right)dx\\
				&\leq C\int \left| \nabla (\rho^{\alpha}) \right| ^{10}\left| \nabla^{2}(\rho^{\alpha}) \right| dx\\
				&\leq \frac{1}{2}\left\lVert \nabla(\rho^{\alpha}) \right\rVert _{L^{12}}^{12}+C\left\lVert \nabla^{2}(\rho^{\alpha}) \right\rVert _{L^{6}}^{6},
			\end{aligned}
		\end{equation}
		combining \eqref{est: v Lr}, \eqref{est: basic}, boundness of $\rho,$ together with $r > 3$, we have 
		\begin{equation}
			\begin{aligned}
				J_{32}&\leq C\left\lVert v \right\rVert _{L^{r}}\left\lVert \nabla^{2}v \right\rVert _{L^{2}}\left\lVert \nabla \rho^{\alpha} \right\rVert _{L^{\frac{4r}{r-2}}}^{2}\\
				&\leq C\left\lVert \nabla^{2}v \right\rVert _{L^{2}}\left\lVert \nabla (\rho^{\alpha}) \right\rVert _{L^{12}}^{\frac{6(r+2)}{5r}}\left\lVert \nabla (\rho^{\alpha}) \right\rVert _{L^{2}}^{\frac{4(r-3)}{5r}}\\
				&\leq C\left\lVert \nabla^{2}v \right\rVert _{L^{2}}\left\lVert \nabla\Delta (\rho^{\alpha}) \right\rVert _{L^{2}}^{\frac{3(r+2)}{5r}}\left\lVert \nabla(\rho^{\alpha}) \right\rVert _{L^{2}}^{\frac{4(r-3)}{5r}}\\
				&\leq \eta \left\lVert \nabla^{2}v \right\rVert ^{2}_{L^{2}}+\eta \left\lVert \nabla\Delta \rho^{\alpha} \right\rVert ^{2}_{L^{2}}+C_{\eta}.
			\end{aligned}
		\end{equation}
		Therefore, we obtain
		\begin{equation}\label{est: ho1 3D tmp3}
			C\int_{\mathbb{R}^3}|\nabla\rho|^2|\nabla v|^2\,dx
			\le
			2\eta\|\nabla^2 v\|_{L^2}^2
			+
			2\eta A(t)
			+
			C_\eta\bigl(1+\|\nabla v\|_{L^2}^2\bigr).
		\end{equation}
		\eqref{est: nabla rho L6 3D} directly yields 
		\begin{equation}
			\int_{\mathbb{R}^3}\Bigl(|\nabla\rho|^2|\nabla^2\rho|^2+|\nabla\rho|^6\Bigr)\,dx
			\le C A(t).
		\end{equation}
		It holds that 
		\begin{equation}\label{est: ho1 3D tmp2}
			\int_{\mathbb{R}^3}|\nabla\rho|^2\,dx\le C.
		\end{equation}
		Inserting \eqref{est: ho1 3D tpm1}-\eqref{est: ho1 3D tmp2} into \eqref{est: ho1Step4} and choosing $\eta>0$ sufficiently small, we obtain
		\begin{equation}\label{est: ho1 3D tmp7}
			\begin{aligned}
				&\frac{d}{dt}\mathcal E(t)
				+\frac14\|\nabla\rho_t\|_{L^2}^2
				+\frac{c^2}{32}\|\nabla\Delta\rho^\alpha\|_{L^2}^2
				+\frac{C_{\rho,v}}{\varepsilon_1}\int_{\mathbb{R}^3}\rho|v_t|^2\,dx
				+\frac{C_{\rho,v}}{2}\|\nabla^2 v\|_{L^2}^2
				\\
				\le&
				C_{\mathcal E,A}\Bigl(1+\mathcal E(t)+A(t)\Bigr).
			\end{aligned}
		\end{equation}
		To close the estimate, it remains to deal with $A(T).$ Recalling \eqref{eq135}, we have 
		\begin{equation}\label{est: ho1 3D tmp5}
			\begin{aligned}
				\frac1{4c}\frac{d}{dt}\|\nabla\rho\|_{L^4}^4
				+\frac{\varepsilon_2}{2}A(t)
				\le
				C\int_{\mathbb{R}^3}
				\Bigl(
				|\nabla\rho|^2|\nabla v|^2+|v|^2|\nabla\rho|^4
				\Bigr)\,dx.
			\end{aligned}
		\end{equation}
		It follows from H\"{o}lder's inequality, Sobolev embedding, Young's inequality, \eqref{est: basic}, \eqref{est: v Lr}, \eqref{est: nabla rho L12} and boundness of $\rho$ that 
		\begin{equation}\label{est: ho1 3D tmp4}
			\begin{aligned}
				C\int_{\mathbb{R}^{3}}|v|^{2}|\nabla \rho|^{4}dx &\leq C\int_{\mathbb{R}^{3}}|v|^{2}|\nabla(\rho^{\alpha})|^{4}dx \\
				&\leq C\| v \|_{L^{r}}^{2}\| \nabla(\rho^{\alpha}) \|_{L^{\frac{4r}{r-2}}}^{4}\\
				&\leq C\| v \|_{L^{r}}^{2}\| \nabla(\rho^{\alpha}) \|_{L^{12}}^{\frac{12(r+2)}{5r}}\| \nabla(\rho^{\alpha}) \|_{L^{2}}^{\frac{8(r-3)}{5r}}\\
				&\leq C\| \nabla^{2}(\rho^{\alpha}) \|_{L^{6}}^{\frac{6(r+2)}{5r}}\\
				&\leq C\| \nabla\Delta(\rho^{\alpha}) \|_{L^{2}}^{\frac{6(r+2)}{5r}}\\
				&\leq \eta \| \nabla\Delta(\rho^{\alpha}) \|_{L^{2}}^{2} + C_{\eta}.
			\end{aligned}
		\end{equation}
		Substituting \eqref{est: ho1 3D tmp3}, \eqref{est: ho1 3D tmp4} into \eqref{est: ho1 3D tmp5} and combining \eqref{est: ho1 E}, we have
		\begin{equation}\label{est: ho1 3D tmp6}
			\frac1{4c}\frac{d}{dt}\|\nabla\rho\|_{L^4}^4
			+\frac{\varepsilon_2}{4}A(t)
			\le
			2\eta\|\nabla^2 v\|_{L^2}^2
			+\eta\|\nabla\Delta(\rho^{\alpha})\|_{L^2}^2+
			C_\eta\bigl(1+\mathcal E(t)\bigr),
		\end{equation}
		With \eqref{est: ho1 3D tmp6} multiplied by $\frac{6C_{\mathcal E,A}}{\varepsilon_2}$ and added to \eqref{est: ho1 3D tmp7}, we choose $\eta>0$ sufficiently small and arrive at
		\begin{equation}
			\begin{aligned}
				&\frac{d}{dt}\mathcal E_3(t)
				+\frac14\|\nabla\rho_t\|_{L^2}^2
				+\frac{c^2}{64}\|\nabla\Delta\rho^\alpha\|_{L^2}^2
				+\frac{C_{\rho,v}}{\varepsilon_1}\int_{\mathbb{R}^3}\rho|v_t|^2\,dx
				+\frac{C_{\rho,v}}{4}\|\nabla^2 v\|_{L^2}^2
				+\frac{C_{\mathcal E,A}}{2}A(t)
				\notag\\
				\le&
				C\bigl(1+\mathcal E(t)\bigr),
			\end{aligned}
		\end{equation}
		where 
		\begin{equation}
			{ \mathcal E_3(t)}
			:=
			\mathcal E(t)
			+\frac{3C_{\mathcal E,A}}{2c\varepsilon_2}\|\nabla\rho(t)\|_{L^4(\mathbb{R}^3)}^4.
		\end{equation}
		Noting that $\mathcal E(t) \le \mathcal E_3(t)$, by Gronwall's inequality, we obtain
		\begin{equation}\label{est: ho1 3D final1}
			\sup_{0\le t\le T}\mathcal E_3(t)
			+
			\int_0^T
			\Bigl(
			\|\nabla\rho_t\|_{L^2}^2
			+\|\nabla\Delta\rho^\alpha\|_{L^2}^2
			+\|v_t\|_{L^2}^2
			+\|\nabla^2 v\|_{L^2}^2
			+A(t)
			\Bigr)dt
			\le C.
		\end{equation}
		Then, \eqref{est: ho1 E} immediately yields
		\begin{equation}\label{est: ho1 3D final2}
			\sup_{0\le t\le T}
			\Bigl(
			\|\nabla^2\rho(t)\|_{L^2}^2+\|\nabla v(t)\|_{L^2}^2+\|\nabla\rho(t)\|_{L^4}^4
			\Bigr)
			\le C.
		\end{equation}
		This completes the proof of \eqref{est: ho1.2}.

		\noindent\textbf{Step5: Third-order derivative bounds on the density}

		For $N = 2$, combining \eqref{est: rho H_dot3 2D} with \eqref{est: ho1 2D final} yields  
		\begin{equation}\label{est: ho1 2D rhoH3dot}
			\int_0^T\|\nabla^3\rho(t)\|_{L^2}^2\,dt
			\le
			C\int_0^T\|\nabla\Delta\rho^\alpha(t)\|_{L^2}^2\,dt
			\le C.
		\end{equation}
		For $N = 3$, it follows from \eqref{est: rho H_dot3 3D} and \eqref{est: ho1 3D final1} that
		\begin{equation}\label{est: ho1 3D rhoH3dot}
			\int_0^T\|\nabla^3\rho(t)\|_{L^2}^2\,dt
			\le
			C\int_0^T\Bigl(\|\nabla\Delta\rho^\alpha(t)\|_{L^2}^2+A(t)\Bigr)\,dt
			\le C.
		\end{equation}

		\noindent\textbf{Step6: Time derivative bounds on the density}

		Note that $\|\nabla\rho\|_{L^4}^2\le C\|\nabla\rho\|_{L^2}\|\nabla^2\rho\|_{L^2}$ for $N=2,$ and $\sup_{0\le t\le T}\|\nabla\rho(t)\|_{L^4}\le C$ for $N=3$ due to \eqref{est: ho1 3D final2}. Consequently, by the boundedness of $\rho$, $\eqref{sys of v}_1,$ \eqref{est: ho1 2D final}, \eqref{est: ho1 3D final2}, \eqref{est: v Lr}, H\"{o}lder's inequality and Sobolev embedding, we obtain
		\begin{equation}\label{est: ho1 rho_t}
			\begin{aligned}
				\left\lVert \rho_{t} \right\rVert _{L^{2}}\leq& C\left( \left\lVert v\cdot \nabla \rho \right\rVert _{L^{2}}+\left\lVert \rho \text{div}v\right\rVert _{L^{2}}+\left\lVert \nabla^{2}\rho \right\rVert _{L^{2}}+\left\lVert \left| \nabla \rho \right| ^{2} \right\rVert _{L^{2}} \right) \\
				\leq& C\left( \left\lVert v \right\rVert _{L^{r}}\left\lVert \nabla \rho \right\rVert_{L^{\frac{2r}{r-2}}}+\left\lVert \nabla v \right\rVert _{L^{2}} +\left\lVert \nabla^{2}\rho \right\rVert _{L^{2}}+\left\lVert \nabla \rho \right\rVert _{L^{4}}^{2} \right) \\
				\leq&C\left( \left\lVert \nabla \rho \right\rVert _{H^{1}}+\left\lVert \nabla v \right\rVert _{L^{2}} +\left\lVert \nabla \rho \right\rVert ^{2}_{L^{4}}\right) \leq C.
			\end{aligned}
		\end{equation}
		Collecting the estimates \eqref{est: ho1 2D final}, \eqref{est: ho1 3D final1}, \eqref{est: ho1 3D final2}, \eqref{est: ho1 2D rhoH3dot}, \eqref{est: ho1 3D rhoH3dot} and \eqref{est: ho1 rho_t}, together with the boundedness of $\rho$, this completes the proof.
		
	\end{proof}
	Building upon the first higher-order estimates, we turn to improve the regularity of the effective velocity.
	\begin{proposition}
		Given the conditions in Proposition \ref{prop: ho1}, we have
		\begin{equation}\label{est: ho2}
			\sup_{0\le t\le T}\Bigl(\|v(t)\|_{H^2}^2+\|v_t(t)\|_{L^2}^2\Bigr)
			+\int_0^T\Bigl(\|v(t)\|_{H^3}^2+\|v_t(t)\|_{H^1}^2\Bigr)\,dt
			\le C,
		\end{equation}
		for some constant $C > 0$ depending only on the quantities therein.
		
	\end{proposition}
	\begin{proof}
		Denote \begin{equation}\label{def: Psi}
			\Psi(t):=
			1+\|\nabla\rho(t)\|_{H^2}^2+\|\rho_t(t)\|_{H^1}^2+\|v(t)\|_{H^1}^2+\|\nabla^2v(t)\|_{L^2}^2.
		\end{equation}
		We already have 
		\begin{equation}\label{est: Psi}
			\int_0^T \Psi(t)\,dt\le C
		\end{equation} 
		by applying \eqref{eq24}, \eqref{eq31}, boundness of the density and \eqref{est: ho1.1}. 
		
		Since the density has a positive lower bound, we can rewrite system \eqref{sys of v} as
		\begin{equation}\label{Equ7}
			\left\{
			\begin{aligned}
				&\rho_t - \alpha c\rho^{\alpha-1} \Delta \rho - \alpha(\alpha-1)c\rho^{\alpha-2}|\nabla \rho|^2 = -\operatorname{div}(\rho v), \\
				&v_t - \nu \rho^{\alpha-1} \Delta v - \left[ \nu - c + (\alpha-1)(2\nu-c) \right] \rho^{\alpha-1} \nabla \operatorname{div} v \\
				&\quad = - \left( v - c \alpha \rho^{\alpha-2} \nabla \rho \right) \cdot \nabla v - \gamma \rho^{\gamma-2} \nabla \rho \\
				&\qquad + \alpha \rho^{\alpha-2} \nabla \rho \cdot \left[ \nu \nabla v + (\nu-c) (\nabla v)^t + (\alpha-1)(2\nu-c) (\operatorname{div} v) \mathbb{I} \right] .\\
			\end{aligned}
			\right.
		\end{equation}
		Taking the partial derivative of $\eqref{Equ7}_2$ with respect to $t,$ testing with $v_t$, and applying integration by parts, H\"{o}lder's inequality and Sobolev embedding, we obtain
		\begin{equation}\label{est: ho2tmp1}
			\begin{aligned}
				&\frac12\frac{d}{dt}\|v_t\|_{L^2}^2 +\nu\int \rho^{\alpha-1}|\nabla v_t|^2\,dx +a_0\int \rho^{\alpha-1}\nabla v_t^{\top}:\nabla v_t\,dx \\
				\le& C\bigg( \int |\rho_t|\,|\nabla^2 v|\,|v_t|\,dx +\int |\nabla\rho|\,|\nabla v_t|\,|v_t|\,dx +\int |\nabla v|\,|v_t|^2\,dx \\
				&+\int |v|\,|\nabla v_t|\,|v_t|\,dx +\int |\rho_t|\,|\nabla\rho|\,|\nabla v|\,|v_t|\,dx \\
				& +\int |\rho_t|\,|\nabla v|\,|\nabla v_t|\,dx +\int |\rho_t|\,|\nabla\rho|\,|v_t|\,dx +\int |\nabla\rho_t|\,|v_t|\,dx \bigg)\\
				\leq&C\left( \left\lVert \rho_{t} \right\rVert _{L^{2}}\left\lVert \nabla^{2}v \right\rVert _{L^{3}}\left\lVert v_{t} \right\rVert _{L^{6}}+\left\lVert \nabla \rho \right\rVert _{L^{\infty}}\left\lVert \nabla v_{t} \right\rVert _{L^{2}}\left\lVert v_{t} \right\rVert _{L^{2}}
				\right.\\
				&+\left.
				\left\lVert \nabla v \right\rVert _{L^{3}}\left\lVert v_{t} \right\rVert _{L^{2}}\left\lVert v_{t} \right\rVert _{L^{6}}+\left\lVert v \right\rVert _{L^{\infty}}\left\lVert \nabla v_{t} \right\rVert_{L^{2}} \left\lVert v_{t} \right\rVert _{L^{2}} 
				\right.\\
				&+\left.
				\left\lVert \rho_{t} \right\rVert_{L^{6}}\left\lVert \nabla \rho \right\rVert_{L^{6}}\left\lVert \nabla v \right\rVert _{L^{2}}\left\lVert v_{t} \right\rVert_{L^{6}}+\left\lVert \rho_{t} \right\rVert _{L^{2}} \left\lVert \nabla v \right\rVert _{L^{\infty}}\left\lVert \nabla v_{t} \right\rVert _{L^{2}}
				\right.\\
				&+\left.
				\left\lVert \rho_{t} \right\rVert _{L^{6}}\left\lVert \nabla \rho \right\rVert _{L^{3}}\left\lVert v_{t} \right\rVert _{L^{2}}+\left\lVert \nabla \rho_{t} \right\rVert_{L^{2}}\left\lVert v_{t} \right\rVert _{L^{2}}   \right) \\
				\leq&C\left( \left\lVert \nabla^{3}v \right\rVert _{L^{2}}^{\frac{N}{6}}\left\lVert \nabla^{2}v \right\rVert _{L^{2}}^{1- \frac{N}{6}}\left\lVert v_{t} \right\rVert _{H^{1}}
				+\Psi(t)^{\frac{1}{2}}\left\lVert \nabla v_{t} \right\rVert _{L^{2}}\left\lVert v_{t} \right\rVert _{L^{2}}
				\right.\\
				&+\left.
				\Psi(t)^{\frac{1}{2}}\left\lVert v_{t} \right\rVert _{L^{2}}\left\lVert v_{t} \right\rVert _{L^{6}}
				+\Psi(t)^{\frac{1}{2}}\left\lVert \nabla v_{t} \right\rVert_{L^{2}} \left\lVert v_{t} \right\rVert _{L^{2}} 
				\right.\\
				&+\left.
				\Psi(t)^{\frac{1}{2}}\left\lVert \nabla v_{t} \right\rVert_{L^{2}}
				+ \left\lVert \nabla ^{3}v \right\rVert _{L^{2}}^{\frac{1}{2}}\Psi(t)^{\frac{N-2}{4}}\left\lVert \nabla v_{t} \right\rVert _{L^{2}}
				+\Psi(t)^{\frac{1}{2}}\left\lVert v_{t} \right\rVert _{L^{2}}+\Psi(t)^{\frac{1}{2}}\left\lVert v_{t} \right\rVert _{L^{2}}   \right)\\
				\leq& \left( \eta \left\lVert v_{t} \right\rVert ^{2}_{H^{1}} +\eta \left\lVert \nabla^{3}v \right\rVert ^{2}_{L^{2}}+C\Psi(t)\right)+\left( \eta \left\lVert \nabla v_{t} \right\rVert ^{2}_{L^{2}}+C\Psi(t)\left\lVert v_{t} \right\rVert ^{2}_{L^{2}} \right)\\
				&+\left( \eta \left\lVert v_{t} \right\rVert_{H^{1}}^{2}+C\Psi(t)\left\lVert v_{t} \right\rVert ^{2}_{L^{2}}  \right)+\left( \eta \left\lVert\nabla v_{t} \right\rVert_{L^{2}}^{2}+C\Psi(t)\left\lVert v_{t} \right\rVert ^{2}_{L^{2}}  \right) \\
				&+\left( \eta \left\lVert v_{t} \right\rVert_{H^{1}}^{2}+C\Psi(t) \right)+\left( \eta \left\lVert \nabla v_{t} \right\rVert ^{2}_{L^{2}}+\eta \left\lVert \nabla^{3}v \right\rVert ^{2}_{L^{2}}+C\Psi(t) \right) \\
				& + C\Psi(t)\left( 1+\left\lVert v_{t} \right\rVert ^{2}_{L^{2}} \right) +C\Psi(t)\left( 1+\left\lVert v_{t} \right\rVert ^{2}_{L^{2}} \right)\\ 
				\leq& 6\eta \left\lVert \nabla v_{t} \right\rVert_{L^{2}}+2\eta \left\lVert \nabla^{3}v \right\rVert ^{2}_{L^{2}}+C\Psi(t)(1+\left\lVert v_{t} \right\rVert ^{2}_{L^{2}})  ,
			\end{aligned}
		\end{equation}
		where $a_{0}:=(\nu-c)+(\alpha-1)(2\nu-c)<0.$ Furthermore, here we use $\left\lVert \nabla \rho \right\rVert_{L^{\infty}}\leq C\Psi(t)^{\frac{1}{2}}$ and $\left\lVert v \right\rVert_{L^{\infty}}\leq C\Psi(t)^{\frac{1}{2}}$ due to the embedding $H^{2}(\mathbb{R}^{N})\hookrightarrow L^{\infty}(\mathbb{R}^{N})$, as well as 
		\begin{equation}
			\left\lVert \nabla v \right\rVert _{L^{\infty}}\leq C\left\lVert \nabla^{3}v \right\rVert ^{\frac{1}{2}}_{L^{2}}\left\lVert \nabla^{2}v \right\rVert _{L^{2}}^{\frac{N-2}{2}}\left\lVert \nabla v \right\rVert _{L^{2}}^{\frac{3-N}{2}}\leq C\left\lVert \nabla^{3} v \right\rVert _{L^{2}}^{\frac{1}{2}}\Psi(t)^{\frac{N-2}{4}}.
		\end{equation}
		Noting that 
		\begin{equation}
			\nu \left| \nabla v_{t} \right| ^{2}+a_{0}\nabla v_{t}^{\top}:\nabla v_{t}\geq (\nu+a_{0})\left| \nabla v_{t} \right| ^{2}=\alpha(2\nu-c)\left| \nabla v_{t}\right|^{2},
		\end{equation}
		we have
		\begin{equation}
			\nu\int \rho^{\alpha-1}|\nabla v_t|^{2}\,dx +a_0\int \rho^{\alpha-1}\nabla v_t^{\top}:\nabla v_t\,dx \ge{\varepsilon_{3}}\left\lVert \nabla v_{t} \right\rVert ^{2}_{L^{2}}.
		\end{equation}
		Choosing $\eta< \frac{\varepsilon_{3}}{12}$ leads to 
		\begin{equation}\label{v_t}
			\frac{1}{2}\frac{d}{dt}\|v_t\|_{L^2}^2
			+\frac{\varepsilon_3}{2}\|\nabla v_t\|_{L^2}^2
			\le 2\eta\|\nabla^3v\|_{L^2}^2
			+C\Psi(t)\bigl(1+\|v_t\|_{L^2}^2\bigr).
		\end{equation}
		Next, we proceed to estimate $\left\lVert \nabla^{3}v \right\rVert^{2}_{L^{2}}.$ Reformulating $\eqref{Equ7}_2$, we obtain the equivalent form
		\begin{equation}\label{elliptic}
			\begin{aligned}
				-\nu\Delta v-a_0\nabla\operatorname{div}v
				=F,
			\end{aligned}
		\end{equation}where 
		\begin{equation}
			\begin{aligned}
				F:=&\rho^{1-\alpha}\Big\{
				-v_t-\bigl(v-c\alpha\rho^{\alpha-2}\nabla\rho\bigr)\cdot\nabla v
				-\gamma\rho^{\gamma-2}\nabla\rho \\
				&+\alpha\rho^{\alpha-2}\nabla\rho\cdot
				\bigl[\nu\nabla v+(\nu-c)(\nabla v)^t+(\alpha-1)(2\nu-c)(\operatorname{div}v)\mathbb{I}\bigr]
				\Big\}.
			\end{aligned}
		\end{equation}
		Thanks to the condition $\nu+a_0=\alpha(2\nu-c)>0$, standard elliptic regularity for the Lam\'e system on $\mathbb{R}^N$ implies
		\begin{equation}
			\begin{aligned}
				\left\lVert \nabla^{3}v \right\rVert ^{2}_{L^{2}}\leq &C\left\lVert \nabla F \right\rVert ^{2}_{L^{2}}
				\leq C\left( \int \left| \nabla v_{t} \right| ^{2}+\left| \nabla \rho \right|^{2}\left| v_{t} \right| ^{2}+ \left| \nabla v \right| ^{4}+\left| v \right|^{2} \left| \nabla^{2}v \right| ^{2}+\left| \nabla \rho \right|^{2}\left| v \right| ^{2}\left| \nabla v \right| ^{2}\right.\\
				&+\left. \left| \nabla^{2}\rho \right| ^{2}\left| \nabla v \right| ^{2}+\left| \nabla \rho \right| ^{2}\left| \nabla^{2}v \right|^{2} +\left| \nabla \rho \right|^{4}\left| \nabla v \right|^{2} +\left| \nabla^{2}\rho \right|^{2} +\left| \nabla \rho \right|^{4}   \right)dx \\
				\leq& C\bigg(\left\lVert \nabla v_{t} \right\rVert ^{2}_{L^{2}}+\left\lVert \nabla \rho \right\rVert^{2}_{L^{\infty}} \left\lVert v_{t} \right\rVert ^{2}_{L^{2}} +\left\lVert \nabla v \right\rVert^{2}_{L^{\infty}}\left\lVert \nabla v \right\rVert ^{2}_{L^{2}}+\left\lVert v \right\rVert _{L^{6}}^{2}\left\lVert \nabla^{2}v \right\rVert ^{2}_{L^{3}}
				\\
				&+\left\lVert \nabla \rho \right\rVert _{L^{6}}^{2} \left\lVert v \right\rVert _{L^{6}}^{2}\left\lVert \nabla v \right\rVert ^{2}_{L^{6}}
				+\left\lVert \nabla^{2}\rho \right\rVert_{L^{2}}^{2}\left\lVert \nabla v \right\rVert ^{2}_{L^{\infty}}
				\\
				&+\left\lVert \nabla \rho \right\rVert ^{2}_{L^{6}}\left\lVert \nabla^{2}v \right\rVert _{L^{3}}^{2}+ \left\lVert \nabla \rho \right\rVert^{4}_{L^{6}}\left\lVert \nabla v \right\rVert ^{2}_{L^{6}} +\left\lVert \nabla^{2}\rho \right\rVert^{2}_{L^{2}}+\left\lVert \nabla \rho \right\rVert _{L^{2}}^{2}\left\lVert \nabla \rho \right\rVert ^{2}_{L^{\infty}} \bigg)\\
				\leq&C\left\lVert \nabla v_{t} \right\rVert ^{2}_{L^{2}}
				+C\Psi(t)\left\lVert v_{t} \right\rVert _{L^{2}}^{2}
				+C\left\lVert \nabla^{3} v \right\rVert _{L^{2}}\Psi(t)^{\frac{N-2}{2}}
				+C\left\lVert \nabla^{3}v \right\rVert _{L^{2}}^{\frac{N}{3}}\left\lVert \nabla^{2}v \right\rVert _{L^{2}}^{2- \frac{N}{3}}
				\\
				&+C\Psi(t)
				+C\left\lVert \nabla^{3} v \right\rVert _{L^{2}}\Psi(t)^{\frac{N-2}{2}}
				\\
				&+C\left\lVert \nabla^{3} v \right\rVert _{L^{2}}\Psi(t)^{\frac{N-2}{2}}
				+C\Psi(t)+C+C\Psi(t)
				\\
				\leq&\delta \left\lVert \nabla^{3}v \right\rVert ^{2}_{L^{2}}+C\left\lVert \nabla v_{t} \right\rVert ^{2}_{L^{2}}+C_{\delta}\Psi(t)\left( 1+\left\lVert v_{t} \right\rVert ^{2}_{L^{2}} \right) .
			\end{aligned}
		\end{equation}
		By taking $\delta>0$ to be sufficiently small, we arrive at 
		\begin{equation}\label{est: ho2tmp2}
			\|\nabla^3v\|_{L^2}^2 \le
			{C_1}\|\nabla v_t\|_{L^2}^2 +C\Psi(t)\bigl(1+\|v_t\|_{L^2}^2\bigr),
		\end{equation}       
		for some positive constant $C_1.$ Substituting \eqref{est: ho2tmp2} into \eqref{v_t} and choosing $\eta< \frac{\varepsilon_{3}}{8C_{1}},$ we obtain 
		\begin{equation}
			\begin{aligned}
				\frac12\frac{d}{dt}\|v_t\|_{L^2}^2 +\frac{\varepsilon_{3}}{4}\left\lVert \nabla v_{t} \right\rVert ^{2}_{L^{2}} 
				\leq C\Psi(t)(1+\left\lVert v_{t} \right\rVert ^{2}_{L^{2}}) .
			\end{aligned}
		\end{equation} 
		Combining Gronwall’s inequality with \eqref{est: Psi} yields 
		\begin{equation}\label{est: hotmp3}
			\sup_{0\le t\le T}\|v_t(t)\|_{L^2}^2
			+\int_0^T\Bigl(\|\nabla v_t(t)\|_{L^2}^2+\|\nabla^3v(t)\|_{L^2}^2\Bigr)\,dt
			\le C.
		\end{equation}
		Finally, we estimate $\left\lVert \nabla^{2}v \right\rVert_{L^{\infty}_tL^{2}}.$ Employing standard elliptic regularity for \eqref{elliptic} we conclude that
		\begin{equation}\label{est: ho2tmp4}
			\begin{aligned}
				\left\lVert \nabla^2 v \right\rVert _{L^2}^2 \leq& C \int \Big( |v_t|^2 + |v|^2|\nabla v|^2 + |\nabla\rho|^2|\nabla v|^2 + |\nabla\rho|^2 \Big) dx\\
				\leq&C\left( \left\lVert v_{t} \right\rVert ^{2}_{L^{2}}+\left\lVert v \right\rVert ^{2}_{L^{6}}\left\lVert \nabla v \right\rVert ^{2}_{L^{3}}+\left\lVert \nabla \rho \right\rVert _{L^{6}}^{2}\left\lVert \nabla v \right\rVert ^{2}_{L^{3}}+\left\lVert \nabla \rho \right\rVert _{L^{2}}^{2} \right) \\
				\leq& C+C\left\lVert \nabla^{2}v \right\rVert_{L^2} ^{\frac{N}{3}}\left\lVert \nabla v \right\rVert _{L^{2}}^{2- \frac{N}{3}}\\
				\leq & \frac{1}{2}\left\lVert \nabla^{2}v \right\rVert _{L^{2}}^2+C.
			\end{aligned} 
		\end{equation}
		Combining  \eqref{eq24}, \eqref{est: ho1.1}, \eqref{est: hotmp3} as well as \eqref{est: ho2tmp4}, we complete the proof.
	\end{proof}
	Now, it remains to close the highest-order regularity bound for the density.
	\begin{proposition}\label{prop: ho3}
		Given the conditions in Proposition \ref{prop: ho1}, we have
		\begin{equation}\label{est: ho3}
			\sup_{0 \le t \le T} \left( \|\rho(t)-1\|_{H^3}^2 + \|\rho_t(t)\|_{H^1}^2 \right)
			+\int_0^T \left( \|\rho(t)-1\|_{H^4}^2 + \|\rho_t(t)\|_{H^2}^2 + \|\rho_{tt}(t)\|_{L^2}^2 \right) dt \le C,
		\end{equation}
		for some constant $C > 0$ depending only on the quantities therein.
	\end{proposition}
	\begin{proof}
		From \eqref{eq24}, \eqref{eq31}, boundness of the density, \eqref{est: ho1.1} and \eqref{est: ho2}, we know that
		\begin{equation}\label{est: ho3tmp1}
			\begin{aligned}
				& \sup_{0 \le t \le T} \left( \|\rho(t)-1\|_{H^2}^2 + \|\rho_t(t)\|_{L^{2}}^2+\|v(t)\|^2_{H^2} \right)
				\\
				&+\int_0^T \left( \|\rho(t)-1\|_{H^3}^2 + \|\rho_t(t)\|_{H^1}^2 + \|v_t(t)\| ^2_{H^1}+\|v(t)\|_{H^3}\right) dt \le C.
			\end{aligned}
		\end{equation}
		Now, we first derive the estimates for $\left\lVert \rho_{tt} \right\rVert_{L^{2}}$ and $\left\lVert \nabla^{2} \rho_{t} \right\rVert_{L^{2}}.$ Differentiating $\eqref{sys of v}_1$ with respect to $t$ yields 
		\begin{equation}\label{eq: ho3tmp1}
			\rho_{tt} - c\alpha \text{div}(\rho^{\alpha-1}\nabla\rho_t) = G,
		\end{equation}
		where 
		\begin{equation}
			\begin{aligned}
				G:=&-\text{div}(\rho_t v + \rho v_t) + c\alpha(\alpha-1)\text{div}(\rho^{\alpha-2}\rho_t\nabla\rho) \\
				=& -v\cdot\nabla\rho_t - \rho_t\text{div}v - \rho\text{div}v_t - v_t\cdot\nabla\rho \\&+ c\alpha(\alpha-1)\left[ \rho^{\alpha-2}\nabla\rho_t\cdot\nabla\rho + \rho^{\alpha-2}\rho_t\Delta\rho + (\alpha-2)\rho^{\alpha-3}\rho_t|\nabla\rho|^2 \right].
			\end{aligned}
		\end{equation}
		Testing \eqref{eq: ho3tmp1} against $\rho_{tt}$ we obtain 
		\begin{equation}\label{eq: ho3tmp2}
			\int \rho_{tt}^2 dx + \frac{c\alpha}{2} \frac{d}{dt} \int \rho^{\alpha-1}|\nabla\rho_t|^2 dx =\int G \rho_{tt} dx + \frac{c\alpha(\alpha-1)}{2} \int \rho^{\alpha-2}\rho_t|\nabla\rho_t|^2 dx.
		\end{equation}
		By H\"{o}lder's inequaity, Sobolev embedding, \eqref{est: ho2} and \eqref{est: ho3tmp1}, it holds that
		\begin{equation}\label{est: ho3tmp2}
			\begin{aligned}
				\left| \int G\rho_{tt}dx \right| \leq &\frac{1}{4}\left\lVert \rho_{tt} \right\rVert ^{2}_{L^{2}}+C\left\lVert G \right\rVert ^{2}_{L^{2}}
				\\
				\leq&\frac{1}{4}\left\lVert \rho_{tt} \right\rVert ^{2}_{L^{2}}+C\left( \left\lVert v \right\rVert _{L^{\infty}}^{2}\left\lVert \nabla \rho_{t} \right\rVert _{L^{2}}^{2}+ \left\lVert \rho_{t} \right\rVert _{L^{6}}^{2}\left\lVert \nabla v \right\rVert _{L^{3}}^{2}\right.\\
				&+\left.\left\lVert \rho \right\rVert_{L^{\infty}}^{2}\left\lVert \nabla v_{t} \right\rVert _{L^{2}}^{2}+{\left\lVert v_{t} \right\rVert _{L^{6}}^{2}\left\lVert \nabla \rho \right\rVert _{L^{3}}^{2}}\right.\\
				&+\left.{\left\lVert \nabla \rho \right\rVert _{L^{6}}^{2}}\left\lVert \nabla \rho_{t} \right\rVert _{L^{3}}^{2}+ \left\lVert \rho_{t} \right\rVert _{L^{6}}^{2}{\left\lVert \nabla^{2}\rho \right\rVert _{L^{3}}^{2}}+\left\lVert \rho_{t} \right\rVert _{L^{6}}^{2}\left\lVert \nabla \rho \right\rVert_{L^{6}}^{4} \right) 
				\\
				\leq&\frac{1}{4}\left\lVert \rho_{tt} \right\rVert ^{2}_{L^{2}}+C\left( \left\lVert v \right\rVert _{H^2}^{2}\left\lVert  \rho_{t} \right\rVert _{H^{1}}^{2} 
				\right.\\
				&+\left.\left\lVert \nabla v_{t} \right\rVert _{L^{2}}^{2}+{\left\lVert v_{t} \right\rVert _{H^1}^{2}\left\lVert \nabla \rho \right\rVert _{H^1}^{2}}\right.\\
				&+\left.{\left\lVert \nabla \rho \right\rVert _{H^1}^{2}}\left\lVert \nabla^{2}\rho_{t} \right\rVert _{L^{2}}^{\frac{N}{3}}\left\lVert \nabla \rho_{t} \right\rVert_{L^2}^{2- \frac{N}{3}}+ \left\lVert \nabla\rho_{t} \right\rVert _{L^{2}}^{2}\left\lVert \nabla^{2}\rho \right\rVert _{L^{2}}^{2- \frac{N}{3}}\left\lVert \nabla^{3}\rho \right\rVert^{\frac{N}{3}}_{L^{2}} +\left\lVert \nabla\rho_{t} \right\rVert _{L^{2}}^{2}\left\lVert \nabla \rho \right\rVert_{H^{1}}^{4} \right) 
				\\
				\leq&\frac{1}{4}\left\lVert \rho_{tt} \right\rVert ^{2}_{L^{2}}+C\left(1+\left\lVert  v_{t} \right\rVert _{H^{1}}^{2}+\left\lVert \nabla^{2}\rho_{t} \right\rVert _{L^{2}}^{\frac{N}{3}}\left\lVert \nabla \rho_{t} \right\rVert_{L^2}^{2- \frac{N}{3}}+\left\lVert \nabla \rho_{t} \right\rVert ^{2}_{L^{2}}\left( 1+\left\lVert \nabla^{3}\rho \right\rVert ^{2} _{L^{2}} \right) \right) 
				\\
				\leq& \frac{1}{4}\left\lVert \rho_{tt} \right\rVert ^{2}_{L^{2}}+ \delta \left\lVert \nabla^{2}\rho_{t} \right\rVert ^{2}_{L^{2}}+C_{\delta}\left( 1+\left\lVert v_{t} \right\rVert _{H^{1}}^{2} +\left\lVert \nabla^{3}\rho \right\rVert^{2}_{L^{2}} \right)\left( 1+\left\lVert \nabla \rho_{t} \right\rVert^{2}_{L^{2}}  \right)  .
			\end{aligned}
		\end{equation}
		Meanwhile, we have 
		\begin{equation}\label{est: ho3tmp2_}
			\begin{aligned}
				\left|  \frac{c\alpha(\alpha-1)}{2} \int \rho^{\alpha-2}\rho_t|\nabla\rho_t|^2 dx \right| \leq&C\left\lVert \rho_{t} \right\rVert _{L^{2}}\left\lVert \nabla \rho_{t} \right\rVert _{L^{3}}\left\lVert \nabla \rho_{t} \right\rVert _{L^{6}}\\
				\leq&C\left\lVert \nabla^{2}\rho_{t} \right\rVert _{L^{2}}^{\frac{N}{6}}\left\lVert \nabla \rho_{t} \right\rVert _{L^{2}}^{1- \frac{N}{6}}\left\lVert \nabla^{2}\rho_{t} \right\rVert_{L^2} ^{\frac{N}{3}}\left\lVert \nabla \rho_{t} \right\rVert _{L^{2}}^{1- \frac{N}{3}}\\
				\leq&C\left\lVert \nabla^{2}\rho_{t} \right\rVert_{L^2}^{\frac{N}{2}} \left\lVert \nabla \rho_{t} \right\rVert_{L^{2}}^{2- \frac{N}{2}}\\
				\leq&\delta \left\lVert \nabla^{2}\rho_{t} \right\rVert ^{2}_{L^{2}}+C_{\delta}\left\lVert \nabla \rho_{t} \right\rVert _{L^{2}}^{2}.
			\end{aligned}
		\end{equation}
		To handle the term $\left\lVert \nabla^{2}\rho_{t} \right\rVert_{L^{2}}$ on the RHS, we rewrite \eqref{eq: ho3tmp1} as
		\begin{equation}
			-c\alpha \Delta\rho_t = \rho^{1-\alpha}(-\rho_{tt} + G + c\alpha(\alpha-1)\rho^{\alpha-2}\nabla\rho\cdot\nabla\rho_t ).
		\end{equation}
		Standard elliptic estimate gives
		\begin{equation}
			\begin{aligned}
				\left\lVert \nabla^2\rho_t \right\rVert _{L^2}^2 \le& C \left( \left\lVert \rho_{tt} \right\rVert _{L^2}^2 + \left\lVert G \right\rVert _{L^2}^2 + \left\lVert \nabla\rho \right\rVert _{L^{6}}^2\left\lVert \nabla\rho_t \right\rVert _{L^3}^2 \right) \\
				\leq&C\left\lVert \rho_{tt} \right\rVert ^{2}_{L^{2}}+ \frac{1}{2}\left\lVert \nabla^{2}\rho_{t} \right\rVert ^{2}_{L^{2}}+C\left( 1+\left\lVert v_{t} \right\rVert _{H^{1}}^{2} +\left\lVert \nabla^{3}\rho \right\rVert^{2}_{L^{2}} \right)\left( 1+\left\lVert \nabla \rho_{t} \right\rVert^{2}_{L^{2}}  \right) ,
			\end{aligned}
		\end{equation}
		i.e. 
		\begin{equation}\label{est: ho3tmp3}
			\begin{aligned}
				\left\lVert \nabla^2\rho_t \right\rVert _{L^2}^2
				\leq&{C_{2}}\left\lVert \rho_{tt} \right\rVert ^{2}_{L^{2}}+C\left( 1+\left\lVert v_{t} \right\rVert _{H^{1}}^{2} +\left\lVert \nabla^{3}\rho \right\rVert^{2}_{L^{2}} \right)\left( 1+\left\lVert \nabla \rho_{t} \right\rVert^{2}_{L^{2}}  \right) .
			\end{aligned}
		\end{equation}
		Substituting \eqref{est: ho3tmp2}, \eqref{est: ho3tmp2_}, \eqref{est: ho3tmp3} into \eqref{eq: ho3tmp2} and choosing $\delta< \frac{1}{8C_{2}},$ we deduce
		\begin{equation}\label{est: ho3tmp4}
			\begin{aligned}
				& \frac{c\alpha}{2} \frac{d}{dt} \int \rho^{\alpha-1}|\nabla\rho_t|^2 dx +\int \rho_{tt}^2 dx
				\\
				\leq&\frac{1}{4}\left\lVert \rho_{tt} \right\rVert ^{2}_{L^{2}}+ 2\delta \left\lVert \nabla^{2}\rho_{t} \right\rVert ^{2}_{L^{2}}+C_{\delta}\left( 1+\left\lVert v_{t} \right\rVert _{H^{1}}^{2} +\left\lVert \nabla^{3}\rho \right\rVert^{2}_{L^{2}} \right)\left( 1+\left\lVert \nabla \rho_{t} \right\rVert^{2}_{L^{2}}  \right)  
				\\
				\leq& \frac{1}{2}\left\lVert \rho_{tt} \right\rVert ^{2}_{L^{2}}+ C\left( 1+\left\lVert v_{t} \right\rVert _{H^{1}}^{2} +\left\lVert \nabla^{3}\rho \right\rVert^{2}_{L^{2}} \right)\left( 1+\left\lVert \nabla \rho_{t} \right\rVert^{2}_{L^{2}}  \right)  .
			\end{aligned}
		\end{equation}
		Applying Gronwall's inequality to \eqref{est: ho3tmp4}, along with \eqref{est: ho3tmp1}, \eqref{est: ho3tmp3}, we obtain
		\begin{equation}\label{est: ho3tmp5}
			\sup_{0\le t\le T}\|\rho_t(t)\|_{H^1}^2
			+
			\int_0^T\|\rho_t(t)\|_{H^2}^2\,dt
			\le C.
		\end{equation}
		Here, we use the fact that $\rho_{0}-1\in H^{3}(\mathbb{R}^{N})$ and $v_{0}\in H^{2}(\mathbb{R}^{N})$, and thus $ \rho_{t}(0)=-\text{div}\left(\rho_{0}v_{0} \right)+c{ \Delta(\rho_{0}^{\alpha}) }\in H^{1}(\mathbb{R}^{N}).$

		Next, we prove that $\nabla^{3}\rho \in L^{\infty}(0,T; L^{2}(\mathbb{R}^N)).$ 
		Expressing $\eqref{sys of v}_1$ in the form of 
		\begin{equation}\label{eq: elliptic of rho}
			-\alpha c\,\Delta\rho
			=
			-\rho^{1-\alpha}\rho_t
			+\alpha(\alpha-1)c\,\rho^{-1}|\nabla\rho|^2
			-\rho^{1-\alpha}\operatorname{div}(\rho v),
		\end{equation} 
		and applying elliptic estimate, H\"{o}lder's inequality, Sobolev embedding, \eqref{est: ho3tmp1} together with \eqref{est: ho3tmp5}, it follows that
		\begin{equation}
			\begin{aligned}
				\left\lVert \nabla^{3}\rho \right\rVert ^{2}_{L^{2}}\leq&C\left( \left\lVert  \nabla \rho \rho_{t}\right\rVert  _{L^{2}}^{2}+\left\lVert  \nabla \rho_{t} \right\rVert _{L^{2}}^{2}+\left\lVert \left| \nabla \rho \right| ^{3} \right\rVert _{L^{2}}^{2}+\left\lVert \left| \nabla \rho  \right| \left| \nabla^{2}\rho \right| \right\rVert^{2}_{L^{2}}\right.\\
				&+\left.\left\lVert \left| \nabla \rho \right| ^{2} \left| v \right| \right\rVert  ^{2}_{L^{2}}+\left\lVert  \left| \nabla^{2}\rho \right| v\right\rVert ^{2}_{L^{2}}+\left\lVert\left|  \nabla \rho  \right| \left| \nabla v \right| \right\rVert^{2}_{L^{2}}+\left\lVert \nabla^{2}v \right\rVert ^{2}_{L^{2}}  \right) \\
				\leq&C\left( \left\lVert \rho_{t} \right\rVert _{L^{6}}^{2}\left\lVert \nabla \rho \right\rVert _{L^{3}} ^{2}+\left\lVert  \nabla \rho_{t} \right\rVert _{L^{2}}^{2}+\left\lVert \nabla \rho \right\rVert^{6}_{L^{6}} +\left\lVert \nabla \rho \right\rVert _{L^{6}}^{2}\left\lVert \nabla^{2}\rho \right\rVert ^{2}_{L^{3}}\right.\\
				&+\left.\left\lVert \nabla \rho \right\rVert ^{4}_{L^{6}}\left\lVert v \right\rVert_{L^{6}} ^{2}+\left\lVert \nabla^{2}\rho \right\rVert_{L^{3}}^{2}\left\lVert v \right\rVert _{L^{6}}^{2}+\left\lVert \nabla \rho \right\rVert _{L^{3}}^{2} \left\lVert \nabla v \right\rVert _{L^{6}}^{2}+\left\lVert \nabla^{2}v \right\rVert ^{2}_{L^{2}} \right) 
				\\
				\leq&C\left( \left\lVert \rho_{t} \right\rVert^{2}_{H^{1}} \left\lVert \nabla\rho \right\rVert ^{2}_{H^{1}}+\left\lVert \rho_{t} \right\rVert ^{2}_{H^{1}}+ \left\lVert \nabla \rho \right\rVert _{H^{1}}^{6}+\|\nabla\rho\|_{H^1}^2\left\lVert \nabla^{3}\rho \right\rVert _{L^{2}}^{\frac{N}{3}}\left\lVert \nabla^{2}\rho \right\rVert_{L^{2}}^{2- \frac{N}{3}} \right.\\
				&+\left.\left\lVert \nabla \rho \right\rVert^{4}_{H^{1}}\left\lVert v \right\rVert^{2}_{H^{1}} +\|\nabla^3\rho\|_{L^2}^{\frac{N}{3}} \|\nabla^2\rho\|_{L^2}^{2-\frac{N}{3}}\|v\|_{H^1}^2 +\left\lVert \nabla \rho \right\rVert ^{2}_{H^{1}}\left\lVert \nabla v \right\rVert ^{2}_{H^{1}}+\left\lVert v \right\rVert ^{2}_{H^{2}}  \right) 
				\\
				\leq& \frac{1}{2}\left\lVert \nabla^{3}\rho \right\rVert ^{2}_{L^{2}}+C,
			\end{aligned}
		\end{equation}
		which implies 
		\begin{equation}\label{est: ho3tmp6}
			\sup_{0\le t\le T}\|\nabla^3\rho(t)\|_{L^2}^2\le C.
		\end{equation}
		Finally, we turn to estimate $\left\lVert \nabla^{4}\rho \right\rVert _{L^{2}}.$ Using the elliptic regularity theory on \eqref{eq: elliptic of rho} once again, we deduce that
		\begin{equation}\label{est: ho3tmp7}
			\begin{aligned}
				\left\lVert \nabla^{4}\rho \right\rVert _{L^{2}}\leq& C\left( \left\lVert \rho^{1-\alpha}\rho_{t} \right\rVert _{H^{2}}+\left\lVert \rho^{-1}\left| \nabla \rho \right| ^{2} \right\rVert _{H^{2}}+\left\lVert \rho^{1-\alpha}\text{div}\left(\rho v\right) \right\rVert_{H^{2}}  \right) \\
				\leq&C\left( \left\lVert (\rho^{1-\alpha}-1)\rho_{t} \right\rVert_{H^{2}}+\left\lVert \rho_{t} \right\rVert _{H^{2}}+\left\lVert (\rho^{-1}-1)\left| \nabla \rho \right| ^{2} \right\rVert _{H^{2}}\right.\\
				&+\left. \left\lVert \left| \nabla \rho  \right| ^{2}\right\rVert _{H^{2}} +\left\lVert (\rho^{1-\alpha}-1)\text{div}\left(\rho v\right) \right\rVert _{H^{2}}+\left\lVert \text{div}\left(\rho v\right) \right\rVert_{H^{2}} \right) \\
				\leq&C\big[ \left\lVert \rho-1 \right\rVert_{H^{2}}\left\lVert \rho_{t} \right\rVert _{H^{2}}+\left\lVert \rho_{t} \right\rVert  _{H^{2}}+\left\lVert \rho-1 \right\rVert _{H^{2}}\left\lVert \nabla\rho \right\rVert_{H^{2}}^{2}+\left\lVert \nabla \rho \right\rVert^{2}_{H^{2}}\\
				&+(\left\lVert \rho-1 \right\rVert _{H^{2}}+1)(\left\lVert \nabla\rho \right\rVert _{H^{2}}\left\lVert v \right\rVert  _{H^{2}}+ \left\lVert \rho -1\right\rVert _{H^{2}}\left\lVert \nabla v \right\rVert _{H^{2}}+\left\lVert \nabla v \right\rVert _{H^{2}} )\big] \\
				\leq&C \Big( 1 + \left\lVert \rho_{t} \right\rVert _{H^{2}} + \left\lVert \nabla v \right\rVert _{H^{2}} \Big) \\
				\leq&C \Big( 1 + \left\lVert \nabla^{2}\rho_{t} \right\rVert _{L^{2}} + \left\lVert \nabla v \right\rVert _{H^{2}} \Big) .
			\end{aligned}
		\end{equation}
		Integrating \eqref{est: ho3tmp7} with respect to $t,$ along with \eqref{est: ho3tmp1} and \eqref{est: ho3tmp5} yields 
		\begin{equation}\label{est: ho3tmp8}
			\int_0^T\|\nabla^4\rho(t)\|_{L^2}^2\,dt\le C.
		\end{equation}
		Collecting \eqref{est: ho3tmp1}, \eqref{est: ho3tmp5}, \eqref{est: ho3tmp6} and \eqref{est: ho3tmp8}, we complete the proof.
	\end{proof}
	\section{Proof of Theorem \ref{thm1}}\label{sec3}
	Assume, by contradiction, that the maximal existence time $T^*$ is finite. For any $T \in (0, T^*)$, it follows from Propositions \ref{prop: ho1}-\ref{prop: ho3} that 
	\begin{equation}\label{est: ho total}
		\begin{aligned}
			&\sup_{0 \le t \le T} \left( \|\rho(t)-1\|_{H^3}^2
			+ \|\rho_t(t)\|_{H^1}^2 +\|v(t)\|_{H^2}^2+\|v_t(t)\|_{L^2}^2\right)\\&
			+\int_0^T \left( \|\rho(t)-1\|_{H^4}^2 + \|\rho_t(t)\|_{H^2}^2 + \|\rho_{tt}(t)\|_{L^2}^2+\|v(t)\|_{H^3}^2+\|v_{t}(t)\|_{H^1}^2 \right) dt \le C,
		\end{aligned}
	\end{equation}
	for some constant $C > 0$ depending solely on $N$, $\alpha$, $\gamma$, $\nu$, $\varepsilon,$ $E_0,$ $T^*,$ $r,$ $\underline{\rho_{0}}$, $\overline{\rho_{0}}$, $\|\rho_0^{1/r}v_0\|_{L^r},$ $\|\rho_0-1\|_{H^3},$ $\|v_0\|_{H^2},$ but not on $T.$ 
	From \eqref{est: ho total} we deduce the time continuity:$$\rho - 1 \in C([0,T]; H^3(\mathbb{R}^N)), \quad v \in C([0,T]; H^2(\mathbb{R}^N)).$$
	This allows us to continuously extend the solution to $t = T^*$ by setting$$(\rho(T^*), v(T^*)) := \lim_{t \to T^*} (\rho(t), v(t)).$$
	Passing to the limit preserves the far-field behavior and the regularity, yielding $\rho(T^*)-1 \in H^3$ and $v(T^*) \in H^2.$ Moreover, since we have the uniform bound $\sup_{0\le t < T}\|\rho^{-1}(t)\|_{L^{\infty}}\le C$, where $C$ is independent of $T,$ it follows that $\|\rho^{-1}(T^*)\|_{L^{\infty}}\le C.$ 
	
	Taking $(\rho(T^*), v(T^*))$ as initial data allows us to extend the solution to $[0, T^* + \delta)$, for some $\delta>0,$ which contradicts the maximality of $T^*$. Thus, $T^ *= \infty$. The uniqueness proof is standard and hence omitted.
	\section*{Acknowledgments}
	X. Huang is partially supported by Chinese Academy of Sciences Project for Young Scientists in Basic Research (Grant No. YSBR-031), National Natural Science Foundation of China (Grant Nos. 12494542, 11688101) and National Key R\&D Program of China (Grant No. 2021YFA1000801). 
	
	\vspace{1cm}
	\noindent\textbf{Data availability statement.} Data sharing is not applicable to this article.
	
	\vspace{0.3cm}
	\noindent\textbf{Conflict of interest.} The authors declare that they have no conflict of interest.

	\printindex
\end{document}